\documentclass[12pt,reqno]{article}

\usepackage{amsmath,amsthm,amssymb}
\usepackage{graphicx,xcolor,enumerate,caption,subcaption}
\usepackage{hyperref}
\usepackage{tikz}

\def\R{\mathbb{R}}

\def\P{\mathbb{P}}
\def\E{\mathbb{E}}

\def\Z{\mathbb{Z}}
\def\T{\mathbb{T}}

\newcommand{\1}{\mathbf{1}}

\DeclareMathOperator{\var}{Var}

\newtheorem{theorem}{Theorem}[section]
\newtheorem{proposition}[theorem]{Proposition}
\newtheorem{lemma}[theorem]{Lemma}

\newtheorem{corollary}[theorem]{Corollary}
\newtheorem{claim}[theorem]{Claim}

\def\eps{\varepsilon}

\def\eqd{\,{\buildrel d \over =}\,}
\DeclareMathOperator\supp{supp}

\renewcommand{\Pr}{{\mathbb P}}

\newcommand{\Ze}{\Z_{\mathrm{even}}}

\newcommand{\intB}{\partial_{\bullet}}
\newcommand{\extB}{\partial_{\circ}}
\newcommand{\intextB}{\partial_{\ins \out}}
\newcommand{\ins}{\bullet}
\newcommand{\out}{\circ}

\DeclareMathOperator\rev{rev}

\newcommand{\phasedom}{\mathcal{P}}
\newcommand{\overlap}{\mathsf{overlap}}
\newcommand{\bad}{\mathsf{none}}

\newcommand{\unbal}{\mathsf{unbal}}
\newcommand{\rest}{\mathsf{rest}}

\newcommand{\unique}{\mathsf{uniq}}
\newcommand{\nondom}{\mathsf{nondom}}

\newcommand{\cA}{\mathcal{A}}
\newcommand{\cI}{\mathcal{I}}
\newcommand{\Ent}{\mathsf{Ent}}
\newcommand{\Even}{\mathrm{Even}}
\newcommand{\Odd}{\mathrm{Odd}}
\newcommand{\dpartial}{\vec{\partial}}

\title{Three lectures on random proper colorings of $\Z^d$}

\author{Ron Peled\thanks{School of Mathematical Sciences, Tel Aviv University, Tel Aviv, Israel. Research supported by Israeli Science Foundation grants 861/15 and 1971/19 and by the European Research Council starting grant 678520 (LocalOrder). Email: {\tt peledron@tauex.tau.ac.il}.} \and Yinon Spinka\thanks{Department of Mathematics, University of British Columbia, Vancouver, Canada. Research supported in part by NSERC of Canada. Email: {\tt yinon@math.ubc.ca}.}}

\begin{document}
\maketitle

\begin{abstract}
  A proper $q$-coloring of a graph is an assignment of one of $q$ colors to each vertex of the graph so that adjacent vertices are colored differently. Sample uniformly among all proper $q$-colorings of a large discrete cube in the integer lattice $\Z^d$. Does the random coloring obtained exhibit any large-scale structure? Does it have fast decay of correlations? We discuss these questions and the way their answers depend on the dimension $d$ and the number of colors $q$.  The questions are motivated by statistical physics (anti-ferromagnetic materials, square ice), combinatorics (proper colorings, independent sets) and the study of random Lipschitz functions on a lattice. The discussion introduces a diverse set of tools, useful for this purpose and for other problems, including spatial mixing, entropy and coupling methods, Gibbs measures and their classification and refined contour analysis.
\end{abstract}

\begin{center}
	\vspace{10pt}
	\includegraphics[scale=0.43]{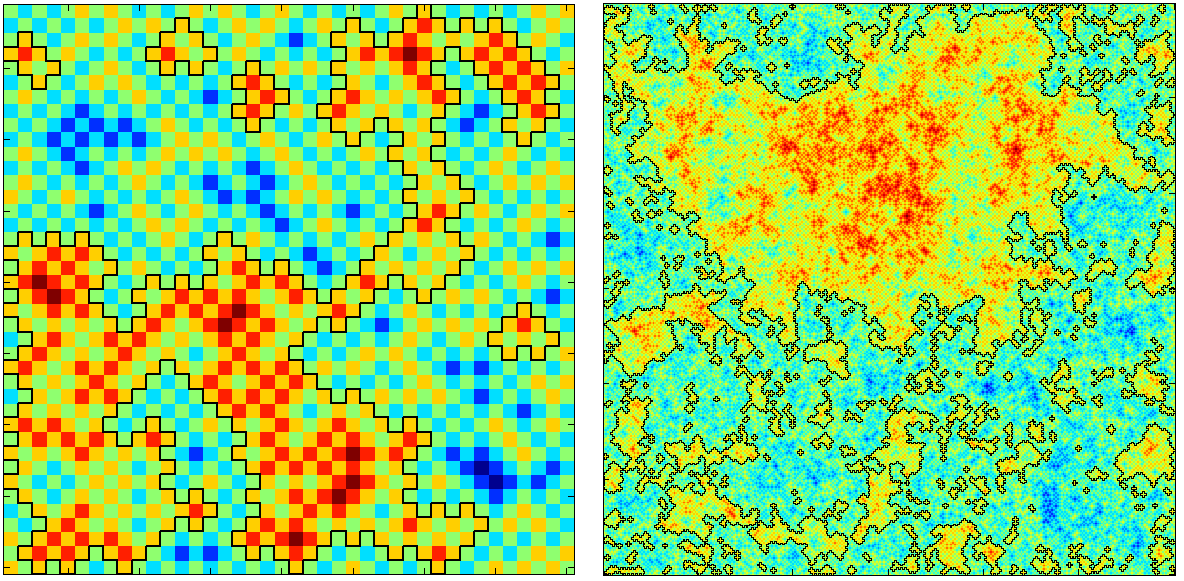}
\end{center}

\newpage

\tableofcontents

\bigskip

These notes were originally prepared for a 3-lecture mini-course on random proper colorings of $\Z^d$ and have been expanded to include additional material. The lectures were given by the first author at the workshop on `Random Walks, Random Graphs and Random Media', September 9-13, 2019, Munich and in the school `Lectures on Probability and Stochastic Processes XIV', December 2-6, 2019, ISI Delhi. The authors are grateful to the organizers of these meetings:  Noam Berger, Alexander Drewitz, Nina Gantert, Markus Heydenreich and Alejandro Ramirez; Arijit Chakrabarty, Rajat Subhra Hazra, Manjunath Krishnapur and Parthanil Roy, for their kind invitation to present this material there.

\bigskip
\textbf{\emph{The notes are in a preliminary state! Comments are welcome.}}

\section{Lecture 1 -- Introduction and Disordered Regime}
A proper $q$-coloring of a graph $G=(V,E)$ is an assignment of the colors $\{1,\ldots,q\}$ to $V$ so that adjacent vertices are colored differently.

\begin{figure}
 \centering
 \includegraphics[scale=1.4]{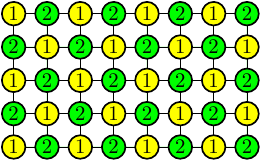}\qquad\qquad
 \includegraphics[scale=1.4]{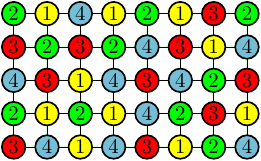}
 \caption{A proper 2-coloring and proper 4-coloring of a rectangular region in $\Z^2$.}
\end{figure}

We wish to study the \emph{uniform distribution} on all proper $q$-colorings of a finite $G$. Our focus in this course is on the case that $G$ is a subset of the $d$-dimensional integer lattice $\Z^d$.
Among the basic difficulties are the facts that:
\begin{itemize}
  \item There is no closed-form formula for the number of proper $q$-colorings. Indeed, on a general graph, it may even be that no coloring exists. In the main case discussed here, when $G$ is a subset of $\Z^d$, colorings exist (for $q\ge 2$) as $\Z^d$ is bipartite, but the counting problem remains difficult.

      One checks simply that the number of proper $q$-colorings, $q\ge 3$, of a bipartite graph increases exponentially with its size. An exponentially large family of colorings is obtained, for instance, by assigning colors in $\{1,2\}$ to one of the two bipartition classes and the color $3$ to the other class. However, finding the precise rate of exponential growth on natural families of graphs (e.g., sub-cubes of $\Z^d$, $d\ge 2$) again seems difficult.

  \item Related to the above is the fact that there is no simple way to sample uniformly from all proper $q$-colorings. One way which one may try is to sample vertex-by-vertex: Each time drawing the color of a vertex uniformly from the colors allowed to it given the previous choices. This method, however, leads in general to a biased sample (though it is exact on trees). More sophisticated sampling algorithms, based on Markov Chain Monte Carlo methods, exist, but provide only an approximate sample and may have long running time.
\end{itemize}

Stated roughly, our focus is on the following questions:
\begin{itemize}
  \item How does a typical proper $q$-coloring \emph{look like}?
  \item Does it exhibit any \emph{structure}?
  \item How strong are the \emph{correlations} of the colors within it?
\end{itemize}

\subsection{Motivation}
The study of the uniform distribution on proper $q$-colorings may be motivated from several points of view:
\begin{enumerate}
  \item Statistical physics: In the Potts model of statistical physics, vertices of $\Z^d$ may be thought of as the atoms/molecules of a crystal, each of which comes equipped with a magnetic spin which takes one of $q$ values.

      In a \emph{ferromagnetic} material, the spins at adjacent vertices have a tendency to be equal. In an \emph{anti-ferromagnetic} material, they have a tendency to differ. The strength of these tendencies is governed by the temperature, with the tendencies made absolute in the \emph{zero-temperature limit}.

      With this terminology, the uniform distribution on proper $q$-colorings is equivalent to the zero-temperature anti-ferromagnetic $q$-state Potts model. Ordering phenomena in anti-ferromagnetic materials are sometimes termed N\'{e}el order.

      While our focus in these notes is on proper $q$-colorings, we wish to emphasize that these also serve as a \emph{paradigm for other statistical physics models, with or without hard constraints}. In particular, some of the methods discussed below apply in some generality.

  \item Combinatorics: Understanding whether a uniformly sampled proper $q$-coloring typically exhibits any large-scale structures, or long-range correlations, is closely related to the counting and sampling problems mentioned above: How many proper $q$-colorings does a grid graph have? How can one sample from these uniformly?

  \item Random Lipschitz functions: Proper $3$-colorings of $\Z^d$ admit an interpretation as \emph{discrete Lipschitz functions}, themselves objects of active research (see Lecture 2).
\end{enumerate}

\subsection{Concrete questions and simple cases}\label{sec:concrete questions}
Uniformly sample a proper $q$-coloring of the cube $\Lambda_L:=\{1,\ldots,L\}^d\subset\Z^d$.
How strong is the influence of one region of the coloring on a distant region? Does it decay to zero with the distance? As a concrete instance of this question, we may ask: as $L\to\infty$, understand
\begin{equation}\label{eq:probability of equal corners}
  \P((1,\ldots,1)\text{ and }(L,\ldots, L)\text{ are equally colored}) - \frac{1}{q}.
\end{equation}
The answer is sought as a function of the number of colors $q$ and the dimension $d$.
The vertices $(1,\ldots,1)$ and $(L,\ldots,L)$ are chosen as they are as distant from each other as possible on $\Lambda_L$. The reason for the subtraction of the factor $\frac{1}{q}$ is that the color assigned to each of the vertices is uniformly picked from $\{1,\ldots, q\}$. Were the two colors independent, they would be equal with probability $\frac{1}{q}$. Thus~\eqref{eq:probability of equal corners} may be regarded as a measure of the correlation between the two colors.

As an alternative measure of correlations in the coloring, one may uniformly sample the coloring subject to prescribed values on the boundary of $\Lambda_L$ and study the effect this has on the coloring in the interior.

Three plausible behaviors for \eqref{eq:probability of equal corners} are highlighted:
\begin{itemize}
  \item \emph{Disorder}: The quantity~\eqref{eq:probability of equal corners} decays to zero exponentially in $L$.
  \item \emph{Criticalilty}: The quantity~\eqref{eq:probability of equal corners} decays to zero as a power-law in $L$.
  \item \emph{Long-range order}: The quantity~\eqref{eq:probability of equal corners} does not decay to zero.
\end{itemize}
There are two cases in which it is easy to decide the behavior:
\begin{itemize}
  \item Long-range order occurs in all dimensions when using only two colors ($q=2$). This occurs as the coloring necessarily has a chessboard pattern (the coloring has exactly two possibilities), which is fully determined by the color of a single vertex.

   \item Disorder occurs in one dimension ($d=1$) for any number $q\ge3$ of colors. Indeed, in this case the coloring is a Markov chain which converges exponentially fast to stationarity (alternatively, explicit calculations are possible).
\end{itemize}
In the rest of the course we discuss other values of $q$ and $d$, attempting to classify their behavior according to the above types: In the next section we discuss the disordered regime, the second lecture considers a case with critical behavior while the third lecture is devoted to long-range order.

\subsection{The disordered regime -- Dobrushin's uniqueness condition}\label{sec:Dobrushin uniqueness}
When $q\gg d$ the colors of the neighbors of a vertex do not limit much the color of the vertex itself. Intuitively, this should imply disorder. Such ideas go back to Dobrushin~\cite{Dobrushin1968TheDe} who found a general condition for the uniqueness of the Gibbs measure. We proceed to describe a general result, of wide applicability, which gives a sufficient condition for a model to be disordered (in the technical sense of satisfying strong spatial mixing, as defined below).

Let $G=(V,E)$ be a finite (simple) connected graph. Let $S$ be a finite `spin space' (the restriction that $S$ be finite is technically convenient but may be relaxed). Let $\mu$ be a probability measure on the \emph{configuration space}, the space of functions $f:V\to S$.
A partial configuration $\tau:B\to S$, defined on a subset $B\subset V$, is called \emph{feasible} if when $f$ is sampled from $\mu$ then $\P(f|_B = \tau)>0$. For a feasible $\tau$, let $\mu^\tau$ be the measure $\mu$ conditioned on the configuration equalling $\tau$ on $B$. Given also $U\subset V$, let $\mu^\tau_U$ be the restriction of this measure to $U$. When $U$ is a singleton $\{u\}$, we write $\mu^\tau_u$ as shorthand for $\mu^\tau_U$.

The notion of spatial mixing quantifies the idea that if $U$ is far from $B$, then the distribution $\mu^\tau_U$ is hardly influenced by the choice of $\tau:B\to S$.
For a precise definition, recall that the total variation distance $d_{\text{TV}}(\nu_1, \nu_2)$ of two probability distributions $\nu_1, \nu_2$ is the maximal value of $|\nu_1(A) - \nu_2(A)|$ over all events $A$. Alternatively, it equals the minimal probability $\P(X\neq Y)$ over all random variables $(X,Y)$ with $X$ distributed $\nu_1$ and $Y$ distributed $\nu_2$ (i.e., over all \emph{couplings} of $\nu_1$ and $\nu_2$. Any coupling which achieves the minimal probability is termed an \emph{optimal coupling}). We say that $\mu$ satisfies \emph{weak spatial mixing} with constants $C,c>0$ if for any $B,U\subset V$ and feasible $\tau_1, \tau_2:B\to S$ it holds that
\begin{equation}\label{eq:weak spatial mixing}
  d_{\text{TV}}(\mu^{\tau_1}_U, \mu^{\tau_2}_U)\le C|U|\exp(-c d_G(U,B)),
\end{equation}
where $d_G$ is the graph distance in $G$. This is one way of making the notion of spatial mixing precise. A more restrictive way requires $\mu^{\tau_1}_U$ and $\mu^{\tau_2}_U$ to be close even when $U$ is close to $B$, as long as it is far from the \emph{disagreement set}
\begin{equation}
  B_{\tau_1, \tau_2} := \{v\in B\colon \tau_1(v) \neq \tau_2(v)\}.
\end{equation}
We say that $\mu$ satisfies \emph{strong spatial mixing} with constants $C,c>0$ if for any $B,U\subset V$ and feasible $\tau_1, \tau_2:B\to S$ it holds that
\begin{equation}\label{eq:strong spatial mixing}
  d_{\text{TV}}(\mu^{\tau_1}_U, \mu^{\tau_2}_U)\le C|U|\exp(-c d_G(U,B_{\tau_1,\tau_2})).
\end{equation}
Evidently, strong spatial mixing implies weak spatial mixing. These definitions are most often used in the context of a sequence of measures $\mu_n$, defined on a sequence of graphs $G_n$, where one usually requires the constants $C,c$ to be uniform in $n$.

We proceed to describe \emph{Dobrushin's uniqueness condition}. We now restrict to the case that $\mu$ is \emph{fully supported}, i.e., $\mu(f)>0$ for all configurations $f:V\to S$ (so that all $\tau$ are feasible). Define the \emph{influence} $I_{u\to v}$ of $u$ on $v$ by
\begin{equation}\label{eq:influence def}
  I_{u\to v} := \max_{\substack{\tau_1,\tau_2:V\setminus\{v\}\to S\\\tau_1=\tau_2\text{ except on $u$}}} d_{\text{TV}}(\mu^{\tau_1}_v, \mu^{\tau_2}_v).
\end{equation}
We say that $\mu$ is \emph{nearest-neighbor} if $I_{u\to v} = 0$ whenever $u$ is not a neighbor of $v$.

\begin{theorem}(Dobrushin's uniqueness condition implies strong spatial mixing)\label{thm:Dobrushin uniqueness} Let $\mu$ be a fully supported and nearest-neighbor probability measure. Assume that $\mu$ satisfies the following condition (known as Dobrushin's uniqueness condition):
\begin{equation}\label{eq:Dobrushins uniqueness condition}
  \alpha := \max_{v\in V} \sum_{u\in V\setminus\{v\}} I_{u\to v}<1.
\end{equation}
Then $\mu$ satisfies strong spatial mixing with constants $C=1$ and $c = -\log\alpha$. That is, for any $B,U\subset V$ and any $\tau_1, \tau_2:B\to S$, we have
\begin{equation}\label{eq:strong spatial mixing conclusion}
  d_{\text{TV}}(\mu^{\tau_1}_U, \mu^{\tau_2}_U)\le |U|\alpha^{d_G(U,B_{\tau_1,\tau_2})}.
\end{equation}
Moreover, for any $B\subset V$ and any $\tau_1, \tau_2:B\to S$ there exist random $f^1, f^2:V\to S$ whose distributions are $\mu^{\tau_1},\mu^{\tau_2}$, respectively (i.e., $f^1, f^2$ are a coupling of $\mu^{\tau_1}, \mu^{\tau_2})$, such that
\begin{equation}\label{eq:coupling inequality}
  \P(f^1(u)\neq f^2(u))\le \alpha^{d_G(u,B_{\tau_1,\tau_2})},\quad u\in V.
\end{equation}
\end{theorem}

We remark that the theorem may be extended further to measures $\mu$ which are not nearest-neighbor.
Indeed, the proof below shows that the conclusion of the theorem remains valid for general interactions under a modified definition/assumption in~\eqref{eq:Dobrushins uniqueness condition}. Namely, for any fully supported measure $\mu$, if instead of~\eqref{eq:Dobrushins uniqueness condition} we suppose the existence of some $0<\bar\alpha<1$ such that
\begin{equation}\label{eq:Dobrushins uniqueness condition2}
\max_{v \in V} \sum_{u\in V\setminus\{v\}} \frac{I_{u \to v}}{\bar\alpha^{d_G(u,v)}} \le 1 ,
\end{equation}
then~\eqref{eq:strong spatial mixing conclusion} and~\eqref{eq:coupling inequality} continue to hold with $\alpha$ replaced by $\bar\alpha$. In particular, if $\mu$ is fully supported, has a finite interaction range $R$ in the sense that $I_{u \to v}=0$ whenever $d_G(u,v)>R$, and satisfies~\eqref{eq:Dobrushins uniqueness condition}, then one easily checks that \eqref{eq:Dobrushins uniqueness condition2} holds with $\bar\alpha=\alpha^{1/R}$ so that strong spatial mixing still holds, but now with constants $C=1$ and $c=-\frac 1R \log \alpha$.

The theorem may also be extended to measures which are not fully supported. A version of the theorem for such measures may be described with the so-called notion of a \emph{specification}. While we do not give the details here, we mention that a specification may be seen as a consistent way to define the measures $\mu^\tau$ for non-feasible $\tau$. Once $\mu^\tau$ is defined for all partial configurations $\tau$, the existing definition~\eqref{eq:influence def} of influence remains valid, and Theorem~\ref{thm:Dobrushin uniqueness} holds for $\mu$ under the same assumption~\eqref{eq:Dobrushins uniqueness condition}. In fact, it is common to define weak/strong spatial mixing for specifications, rather than for measures (for fully supported measures, the definitions coincide).
Finally, we mention that the original aim of Dobrushin~\cite{Dobrushin1968TheDe} was to prove the uniqueness of Gibbs measures for a given specification (with interactions of possibly unbounded range) under the condition~\eqref{eq:Dobrushins uniqueness condition}; see, e.g., ~\cite[Chapter 6]{friedli2017statistical}.

We start the proof of the theorem with a preliminary lemma.
\begin{lemma}\label{lem:total variation distance with influences}
  Let $\mu$ be a fully supported measure. For any $v\in V$ and any partial configurations $f,g:V\setminus\{v\}\to S$ we have \begin{equation}
    d_{\text{TV}}(\mu^f_v, \mu^g_v)\le \sum_{u\in V\setminus\{v\}} \1_{f(u)\neq g(u)} I_{u\to v}.
  \end{equation}
\end{lemma}
\begin{proof}
  Let $(u_m)$, $1\le m\le |V|-1$, be an arbitrary ordering of $V\setminus \{v\}$. Define a sequence of partial configurations $(f_j)$, $0\le j\le |V|-1$, with $f_j:V\setminus\{v\}\to S$, by setting
  \begin{equation*}
    f_j(u_m) := \begin{cases} f(u_m)& m> j\\ g(u_m)& m\le j\end{cases}\quad\text{for}\quad 1\le m\le |V|-1.
  \end{equation*}
  Observe that $f_0 = f$ while $f_{|V|-1} = g$. In addition, $f_{j-1}$ and $f_j$ may differ only at the single vertex $u_j$, where they satisfy $f_{j-1}(u_j)=f(u_j)$ and $f_j(u_j) = g(u_j)$. Thus, as $\mu$ is fully supported, the triangle inequality for the total variation distance and the definition of $I_{u\to v}$ imply that
  \begin{equation*}
    d_{\text{TV}}(\mu^f_v, \mu^g_v)\le \sum_{j=1}^{|V|-1} d_{\text{TV}}(\mu^{f_{j-1}}_v, \mu^{f_j}_v)\le \sum_{j=1}^{|V|-1} \1_{f(u_j)\neq g(u_j)}I_{u_j\to v},
  \end{equation*}
  as we wanted to prove.
\end{proof}

\begin{proof}[Proof of Theorem~\ref{thm:Dobrushin uniqueness}]
  It suffices to prove the moreover part, as the $f^1$ and $f^2$ obtained there, when restricted to $U$, provide a coupling of $\mu^{\tau_1}_U$ and $\mu^{\tau_2}_U$ which proves~\eqref{eq:strong spatial mixing conclusion}.

  To construct $f^1, f^2$, we define a sequence of random pairs of functions $\{(f_n, g_n)\}_{n \ge 0}$ such that, for each $n\ge 0$,
  \begin{enumerate}
    \item $f_n\sim\mu^{\tau_1}$ and $g_n\sim\mu^{\tau_2}$.
    \item For each $v\in V$,
  \begin{equation}\label{eq:prob to differ at v}
    \P(f_n(v)\neq g_n(v))\le \alpha^{\min\{n,d_G(v,B_{\tau_1,\tau_2})\}}.
  \end{equation}
  \end{enumerate}
  Then we may take $(f^1, f^2)$ to be $(f_n, g_n)$ for any $n$ larger than the diameter of $G$.

  To start, let $f_0, g_0$ be sampled independently from $\mu^{\tau_1}$, $\mu^{\tau_2}$ respectively. The above properties clearly hold. Now assume that, for some $n\ge 1$, the pair $(f_{n-1},g_{n-1})$ has already been defined and satisfies the above properties. We define $(f_n,g_n)$ as follows: Let $(v_k)$, $1\le k\le |V\setminus B|$, be an arbitrary ordering of $V\setminus B$. Set $f_{n, 0}:=f_{n-1}$, $g_{n,0}:=g_{n-1}$. Iteratively, for $1\le k\le |V\setminus B|$, define $(f_{n,k},g_{n,k})$ by
  \begin{equation*}
    (f_{n,k}(u),g_{n,k}(u)) := (f_{n,k-1}(u), g_{n,k-1}(u)),\quad u\neq v_k,
  \end{equation*}
  and sampling $(f_{n, k}(v_k), g_{n,k}(v_k))$ by an optimal coupling of $\mu^{f_{n,k}|_{V\setminus\{v_k\}}}_{v_k}$ and $\mu^{g_{n,k}|_{V\setminus\{v_k\}}}_{v_k}$. Set $f_n := f_{n, |V\setminus B|}$ and $g_n := g_{n, |V\setminus B|}$ (in other words, generate $(f_n, g_n)$ from $(f_{n-1}, g_{n-1})$ by a systematic scan, updating the value at each $v_k$ by an optimal coupling given the previous values).

  The fact that $(f_{n-1},g_{n-1})$ satisfies the first property above implies the same for $(f_n,g_n)$. To see the second property, it suffices to show that for $1\le k\le |V\setminus B|$,
  \begin{equation}\label{eq:prob to differ at v with k}
    \P(f_{n,k}(v_k)\neq g_{n,k}(v_k))\le \alpha^{\min\{n,d_G(v_k,B_{\tau_1,\tau_2})\}}.
  \end{equation}
  We proceed to prove~\eqref{eq:prob to differ at v with k} inductively on $k$. The left-hand side of~\eqref{eq:prob to differ at v with k}, conditioned on $(f_{n,k-1}, g_{n,k-1})$, equals $d_{\text{TV}}(\mu^{f_{n,k}|_{V\setminus\{v_k\}}}_{v_k}, \mu^{g_{n,k}|_{V\setminus\{v_k\}}}_{v_k})$ (as an optimal coupling is used). Thus we may apply Lemma~\ref{lem:total variation distance with influences} to obtain
  \begin{equation}\label{eq:prob to differ at v with k_2}
    \P(f_{n,k}(v_k)\neq g_{n,k}(v_k)\,|\,f_{n,k-1}, g_{n,k-1})\le \sum_{u:u\neq v_k}\1_{\{f_{n,k-1}(u)\neq g_{n,k-1}(u)\}} I_{u\to v_k}.
  \end{equation}
  Note that
  \[ \P(f_{n,k-1}(v_\ell) \neq g_{n,k-1}(v_\ell))
  = \begin{cases}
   \P(f_{n,\ell}(v_\ell) \neq g_{n,\ell}(v_\ell)) &\text{if } \ell < k \\
   \P(f_{n-1}(v_\ell) \neq g_{n-1}(v_\ell)) &\text{if }\ell > k
   \end{cases} ,\]
   so that by the inductive statements~\eqref{eq:prob to differ at v with k} on $k$ and~\eqref{eq:prob to differ at v} on $n$,
  \[ \P(f_{n,k-1}(v) \neq g_{n,k-1}(v)) \le \alpha^{\min\{n-1,d_G(v,B_{\tau_1,\tau_2})\}}, \quad v \in V .\]
  Thus, taking expectation in~\eqref{eq:prob to differ at v with k_2}, we obtain that
  \begin{equation}
  \begin{split}
    \P(f_{n,k}(v_k)\neq g_{n,k}(v_k))&\le \sum_{u:u\neq v_k} \P(f_{n,k-1}(u)\neq g_{n,k-1}(u)) I_{u\to v_k}\\
    &\le \sum_{u:u\neq v_k} \alpha^{\min\{n-1,d_G(u,B_{\tau_1,\tau_2})\}}I_{u\to v_k}\\
    &\le \alpha^{\min\{n, d_G(v_k, B_{\tau_1,\tau_2})\}}\sum_{u:u\neq v_k} \frac{I_{u\to v_k}}{\alpha^{d_G(u,v_k)}} \le \alpha^{\min\{n, d_G(v_k, B_{\tau_1,\tau_2})\}},
  \end{split}
  \end{equation}
  where in the third inequality we used that $\min\{n-1,d_G(u,B_{\tau_1,\tau_2})\} \ge \min\{n, d_G(v_k, B_{\tau_1,\tau_2})\} - d_G(u,v_k)$, and in the last inequality we used the assumption that $\mu$ is nearest-neighbor and the definition~\eqref{eq:Dobrushins uniqueness condition} of $\alpha$. This finishes the proof of the theorem.
\end{proof}

Theorem~\ref{thm:Dobrushin uniqueness} postulates that $\mu$ is fully supported. We note first that this assumption cannot be removed without a suitable replacement (such as the notion of specification mentioned above). Indeed, consider for instance the case that $G$ is a graph with at least three vertices, $S=\{1,2\}$ and the measure $\mu$ is uniform on the two constant configurations. It is clear that $\mu$ does not satisfy strong spatial mixing. Moreover, the influences $I_{u\to v}$ are not well defined as $\mu^\tau$ is only defined for feasible $\tau$. However, if one restricts the definition~\eqref{eq:influence def} to use only feasible $\tau$, then it is straightforward that all resulting influences $I_{u\to v}$ are zero, so that Dobrushin's condition~\eqref{eq:Dobrushins uniqueness condition} is trivially satisfied.

Our next goal is to apply the theorem to proper colorings. The theorem does not apply directly, as the uniform distribution on proper $q$-colorings is not fully supported. One remedy is to introduce the anti-ferromagnetic $q$-state Potts model. At `inverse temperature' $\beta\ge 0$, this is the measure assigning probability proportional to
\begin{equation*}
  \exp\left(-\beta\sum_{u\sim v} \1_{f(u) = f(v)}\right)
\end{equation*}
to every $f:V\to\{1,\ldots, q\}$. It is readily verified that, on the one hand, this measure is fully supported and, on the other hand, the measure tends to the proper $q$-coloring measure as $\beta$ tends to infinity. One checks that that anti-ferromagnetic $q$-state Potts model satisfies Dobrushin's uniqueness condition whenever either $\beta\le \frac{C_q}{\Delta}$ or $q>2\Delta$, with $\Delta$ the maximal degree in $G$. This implies that the proper $q$-coloring model satisfies strong spatial mixing when $q>2\Delta$.

An alternative remedy that may be used to apply Dobrushin's uniqueness condition for proper $q$-colorings is the following observation. The assumption that $\mu$ is fully supported is only used in the proof of Lemma~\ref{lem:total variation distance with influences} and it thus suffices to find an extension of this lemma to a class of non-fully-supported measures which includes the proper $q$-coloring model. The lemma indeed admits such an extension to measures $\mu$ with a representation
\begin{equation}\label{eq:mu pair interaction}
  \mu(f) = \frac{1}{Z}\prod_{(u,v)\in\vec{E}}h_{(u,v)}(f(u), f(v))\prod_{v\in V} \lambda_v(f(v))
\end{equation}
where $\lambda_v:S\to(0,\infty)$, $h_{(u,v)}:S\times S\to [0,\infty)$ and $\vec{E}$ is an arbitrary orientation of the edges of $E$ (it is straightforward to find such a representation for the proper $q$-coloring model). Lemma~\ref{lem:total variation distance with influences} extends to such measures provided the influences $I_{u\to v}$ are set to $0$ when $u\not\sim v$ and the nearest-neighbor influences $I_{u\to v}$, $u\sim v$, are calculated in the star graph $G_v$, i.e., the graph whose vertex set is $v$ and the $G$-neighbors of $v$ and whose edge set is the set of edges between $v$ and each of these neighbors. The measure $\mu$ admits a natural restriction to $G_v$ via the formula~\eqref{eq:mu pair interaction} restricted to vertices and edges of $G_v$ (with $Z$ suitably modified), and the influences $I_{u\to v}$, $u\sim v$, should then be calculated with respect to this restricted measure. The proof starts by noting that the total variation distance for the marginal distribution at $v$ which appears in the statement of Lemma~\ref{lem:total variation distance with influences} may only increase when restricting the graph to $G_v$ and using the restricted version of $\mu$. From this one continues along the same steps of the above proof of the lemma.

We record the obtained conclusion for the proper $q$-coloring model.
\begin{corollary}
Uniform proper colorings satisfy strong spatial mixing whenever the number of colors is greater than twice the maximal degree. More precisely, if $G=(V,E)$ is a finite connected graph of maximal degree $\Delta$, and $\mu$ is the uniform distribution on proper $q$-colorings of $G$, then for any $B,U\subset V$ and feasible $\tau_1, \tau_2:B\to S$ it holds that
\begin{equation}\label{eq:strong spatial mixing_colorings}
  d_{\text{TV}}(\mu^{\tau_1}_U, \mu^{\tau_2}_U)\le |U| \left(\frac{\Delta}{q-\Delta}\right)^{d_G(U,B_{\tau_1,\tau_2})}.
\end{equation}
\end{corollary}
\begin{proof}
Since the proper $q$-coloring model has the form~\eqref{eq:mu pair interaction}, it suffices to show that~\eqref{eq:Dobrushins uniqueness condition} holds and that $\alpha \le \frac{\Delta}{q-\Delta}$. To this end, it suffices to show that $I_{u \to v} \le \frac{1}{q-\Delta}$ for any adjacent vertices $u,v \in V$ (when calculating the influences on the star graph $G_v$).
Towards showing this, let $\tau_1,\tau_2 \colon V \setminus\{v\} \to S$ agree except on $u$. Let $A_1$ and $A_2$ be the set of colors appearing on $N(v)$ in $\tau_1$ and $\tau_2$, respectively. Note that either $A_1=A_2$ in which case $d_{\text{TV}}(\mu^{\tau_1}_v, \mu^{\tau_2}_v)=0$, or $(|A_1 \setminus A_2|,|A_2 \setminus A_1|) \in \{(0,1),(1,0),(1,1)\}$ in which case
\[ d_{\text{TV}}(\mu^{\tau_1}_v, \mu^{\tau_2}_v) = \frac{1}{q-\min\{|A_1|,|A_2|\}} .\]
Thus, since $|A_1| \le \Delta$, we always have that $d_{\text{TV}}(\mu^{\tau_1}_v, \mu^{\tau_2}_v) \le \frac{1}{q-\Delta}$.
\end{proof}

We conclude this section with several remarks.

There are extensions of Dobrushin's condition by Dobrushin and Shlosman~\cite{dobrushin1985constructive, dobrushin1985completely, dobrushin1987completely} which involve influences between a vertex and the joint distribution on a collection of other vertices. These can improve upon Dobrushin's uniqueness condition but are generally harder to check. We do not go into their theory here. A different consequence of Dobrushin's condition, or the improved Dobrushin--Shlosman conditions, concerns dynamical processes associated with the given measure -- Glauber and block-Glauber dynamics. It is known that strong spatial mixing is equivalent to rapid mixing (in time $O(|V|\log|V|)$) of such dynamics. For more on these topics, we refer the reader to Dyer--Sinclair--Vigoda--Weitz~\cite{dyer2004mixing}, Martinelli~\cite{martinelli1999lectures},
Martinelli--Olivieri~\cite{martinelli1994approach,martinelli1994approach2} and Weitz~\cite{weitz2005combinatorial}.

A different method of general applicability for showing strong spatial mixing and uniqueness of Gibbs measures is the method of disagreement percolation, introduced by van den Berg~\cite{van1993uniqueness} and further developed by van den Berg and Maes~\cite{van1994disagreement}. The method has the potential to improve upon Dobrushin's uniqueness condition, though such improvements, when present, are mostly significant in low dimensions. The method applies to the proper $q$-coloring model on $\Z^d$ but allows to deduce strong spatial mixing only when $q>Cd^2$, for some $C>0$, while Dobrushin's condition applies for $q>4d$.

As discussed, Dobrushin's uniqueness condition implies strong spatial mixing for the proper $q$-coloring model when $q>2\Delta$ with $\Delta$ the maximal degree in $G=(V,E)$. Vigoda~\cite{vigoda2000improved} proved that a natural `flip dynamics' mixes in time $O(|V|\log|V|)$ and Glauber dynamics mixes in time polynomial in $|V|$ when $q\ge\frac{11}{6}\Delta$. This was improved to $q\ge(\frac{11}{6}-\eps)\Delta$ for a fixed $\eps>0$ by Chen--Delcourt--Moitra--Perarnau--Postle~\cite{chen2019improved} (strong spatial mixing is not discussed in these papers but it is proved that uniqueness of the Gibbs measure of proper $q$-colorings of $\Z^d$ holds under the stated conditions on $q$). When $G$ is triangle-free and has $\Delta\ge 3$ (e.g., for domains in $\Z^d$), strong spatial mixing is proved under the weaker assumption $q>\alpha\Delta - \gamma$ with $\alpha$ the solution of $\alpha\log\alpha=1$ (so that $\alpha\approx 1.76$) and $\gamma = \frac{4\alpha^3-6\alpha^2-3\alpha+4}{2(\alpha^2-1)}\approx 0.47$ by Goldberg--Martin--Paterson~\cite{goldberg2005strong}. Using these general bounds and employing computer assistance for checking more refined conditions of the Dobrushin--Shlosman type, the current state of the art on $\Z^2$ is that strong spatial mixing is proved for $q\ge 6$ colors, and is believed to occur also for $q=4,5$ (see~\cite{salas1997absence, bubley1999approximately, molloy2004sampling, goldberg2004strong, goldberg2006improved}. See also~\cite{jalsenius2009strong} for the Kagome lattice). The case of $q=3$ colors is critical, and is the subject of the next lecture.

\section{Lecture 2 -- Criticality}
Recall the three types of behavior highlighted in Section~\ref{sec:concrete questions} for the quantity~\eqref{eq:probability of equal corners}. In this lecture, we discuss the uniform distribution on proper $3$-colorings of $\Z^2$, which is predicted in the physics literature to behave critically. We mention in passing that the uniform distribution on proper $4$-colorings of the triangular lattice is also predicted to behave critically (it is equivalent to the loop $O(2)$ model at $x=\infty$; see~\cite[Chapter 3]{peled2019lectures}). However, in this latter case, there are currently no rigorous results.

\subsection{Frozen colorings}\label{sec:frozen}

A \emph{frozen coloring} of an infinite graph is a proper coloring of the graph such that no finite region of the coloring can be modified while keeping the coloring proper. In other words, a proper $q$-coloring is frozen if any proper $q$-coloring which differs from it on only finitely many vertices must coincide with it.

It is possible to find a frozen 3-coloring of $\Z^2$ (more generally, frozen $q$-colorings of $\Z^d$ exist if and only if $q\le d+1$; see Alon--Brice\~no--Chandgotia--Magazinov--Spinka~\cite{alon2019mixing}). Indeed, it is straightforward to check that the coloring $f$ defined by $f(x,y):=x + y\pmod 3$ is a frozen coloring; see Figure~\ref{fig:frozen}.
\begin{figure}
 \centering
 \includegraphics[scale=1.4]{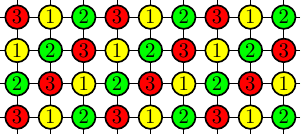}
 \caption{A frozen 3-coloring of $\Z^2$.}
 \label{fig:frozen}
\end{figure}
The existence of a frozen coloring already shows that correlations cannot decay for all boundary conditions (no weak or strong spatial mixing). However, there is great interest also in placing no boundary conditions (free boundary conditions on a domain) or using only a restricted set of boundary conditions, especially those for which the number of extensions to a proper coloring of the domain is close to the total number of proper colorings of the domain. Specifically, one may require the ratio of the logarithms of these two quantities tend to $1$ along a sequence of growing domains with specified boundary conditions, and the boundary conditions are then said to achieve \emph{maximal entropy} (see Section~\ref{sec:long-range order} for a related notion).

\subsection{Constant boundary conditions}

Define the (graph) ball of radius $L$ in $d$ dimensions by
\begin{equation}\label{eq:Lambda L def}
  \Lambda(L):=\{v\in\Z^d\colon \|v\|_1 \le L\},\quad \text{$L\ge 0$ integer}.
\end{equation}
The internal vertex boundary of a set $\Lambda\subset\Z^d$ is denoted
\begin{equation}\label{eq:internal vertex boundary}
  \intB\Lambda:=\{v\in\Z^d\colon v\in \Lambda, \exists w\sim v, w\notin\Lambda\}.
\end{equation}
We consider proper $3$-colorings of $\Lambda(L)$ for which all of $\intB\Lambda(L) = \Lambda(L) \setminus \Lambda(L-1)$ is assigned the same color. Such boundary conditions achieve maximal entropy, as shown by Galvin--Kahn--Randall--Sorkin~\cite[Lemma 5.1]{galvin2012phase} using \emph{Kempe chains}. Our goal is to show that if $d=2$ and $f_L$ is a uniformly sampled proper $3$-coloring of this type, then
\begin{equation}\label{eq:uniformity of color}
  \lim_{L\to\infty} \P(f_L(0,0) = i) = \frac{1}{3},\quad 1\le i\le 3,
\end{equation}
which is a form of correlation decay. While~\eqref{eq:uniformity of color} does not distinguish between the disordered and critical cases (in the sense discussed in Section~\ref{sec:concrete questions}), recent work of Duminil-Copin--Harel--Laslier--Raoufi--Ray~\cite{duminil2019logarithmic} can be used to obtain a rate of convergence in~\eqref{eq:uniformity of color}, showing that
\begin{equation}\label{eq:uniformity of color with power estimates}
  \frac{c}{L^{\alpha_1}}\le \left|\P(f_L(0,0) = i) - \frac{1}{3}\right| \le \frac{C}{L^{\alpha_2}},
\end{equation}
for some $C,c,\alpha_1, \alpha_2>0$ (see~\cite{rayspinka2020} for additional details). We proceed to describe the proof technique for~\eqref{eq:uniformity of color}, which is of interest also in other contexts.

\subsection{Height function}
We continue the discussion of proper $3$-colorings in general dimension $d$, specializing to the two-dimensional case later. A homomorphism height function (or $\Z$-homomorphism) on $\Z^d$ is a function $h:\Z^d\to\Z$ satisfying
\begin{equation}
  |h(u) - h(v)| = 1\quad\text{when $u$ is adjacent to $v$,}
\end{equation}
and taking even values on the even sublattice
\begin{equation*}
  \Ze^d := \bigg\{(x_1,x_2,\ldots,x_d)\in\Z^d\colon \sum_{i=1}^d x_i\text{ even}\bigg\}.
\end{equation*}

Observe that, trivially, if $h$ is a homomorphism height function, then the (pointwise) modulo $3$ of $h$ is a proper $3$-coloring of $\Z^d$. In the other direction, we have the following.
\begin{claim}
  For each proper $3$-coloring $f$ of $\Z^d$, there exists a homomorphism height function $h$ such that $f \equiv h \pmod{3}$.
\end{claim}
\begin{proof}
  Set $h$ at the origin to be an arbitrary even integer congruent to $f$ modulo 3 (e.g., 0,4,2 according to whether $f$ is 0,1,2, respectively). As $f$ uniquely defines $h(u) - h(v)$ for adjacent $u,v$, one needs only to check the consistency of this definition along cycles of $\Z^d$. As the cycle space is spanned by basic $4$-cycles (plaquettes -- sets of the form $\{v, v+e_i, v+e_i+e_j, v+e_j\}$ for $v\in\Z^d$, $1\le i,j\le d$ distinct, where $e_k$ denotes the $k$th unit vector), it suffices to check consistency on these, and this may be done simply by going over all $6$ options for the coloring of such a $4$-cycle (with the color of one vertex fixed).
\end{proof}
Thus the modulo $3$ mapping is a bijection between homomorphism height functions defined up to the addition of a global constant in $6\Z$ and proper $3$-colorings of $\Z^d$. The same bijection holds also between homomorphism height functions and proper $3$-colorings on $\Lambda(L)$. Our analysis of the coloring will go through the height function.

In one dimension, homomorphism height functions on an interval are simply paths of simple random walk. Thus, uniformly sampled homomorphisms on an interval of length $L$, fixed at one endpoint, have fluctuations of order $\sqrt{L}$. In two dimensions, homomorphism height functions are in bijection also with the uniform six-vertex model (square ice).

Now specializing to two dimensions, we study height functions $h_L$ on $\Lambda(L)$, $L$ even, which are uniformly sampled from the set of height functions fixed to equal zero on $\intB\Lambda(L)$ (Figure~\ref{fig:hom_with_level_line} shows such a function sampled on a square domain). What is the analogue of the limit statement~\eqref{eq:uniformity of color}? It is clear that $\E(h_L(0,0)) = 0$ by symmetry. It is then natural to associate the statement~\eqref{eq:uniformity of color} on the modulo $3$ of $h_L(0,0)$ with the statement that the fluctuations of $h_L(0,0)$ grow unboundedly with $L$ (see Section~\ref{sec:implication on coloring} below).
\begin{theorem}(Chandgotia--Peled--Sheffield--Tassy~\cite{chandgotia2018delocalization})\label{thm:delocalization}
When $d=2$, we have
\begin{equation}
  \lim_{L\to\infty} \var(h_L(0,0)) = \infty.
\end{equation}
\end{theorem}
\begin{figure}
	\begin{center}
		\includegraphics[scale=0.43]{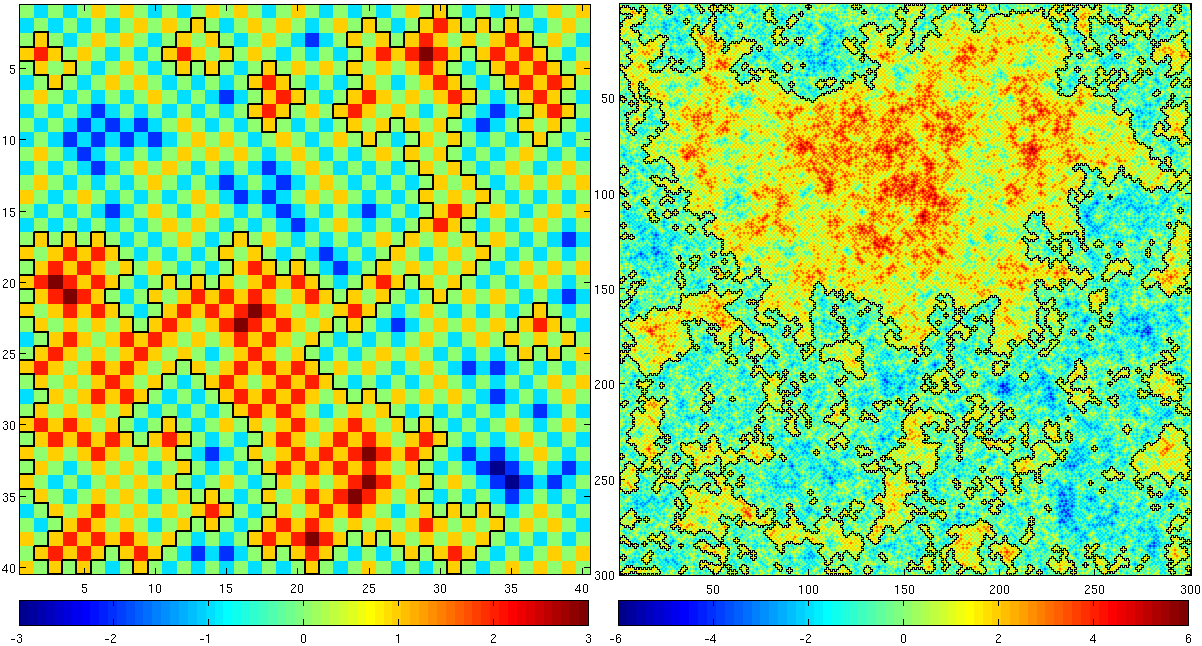}
	\end{center}
	\caption{A uniformly sampled homomorphism height function with zero boundary values (on every second vertex) on a $40\times40$ and $300\times300$ squares, sampled using coupling from the past~\cite{propp1996exact}. The non-trivial outermost level sets separating zeros and ones are highlighted in black. Prepared with help from Steven M. Heilman. Theorem~\ref{thm:delocalization} implies that the height at the center of the squares diverges as the square size increases.}
	\label{fig:hom_with_level_line}
\end{figure}

The work~\cite{duminil2019logarithmic} extends this result to prove that the variance is of order $\log(L)$. It is further conjectured that the scaling limit of $h_L$ is the continuum Gaussian free field, and that the level lines of $h_L$ (see Figure~\ref{fig:hom_with_level_line}) scale to the Conformal Loop Ensemble (CLE) with parameter $\kappa = 4$. In contrast, it is conjectured that in dimensions $d\ge 3$ the fluctuations (at a given vertex) of uniformly sampled homomorphism height functions on a domain with zero boundary conditions remain bounded uniformly in the domain. This is presently known only in high dimensions~\cite{Kahn2001hypercube, Galvin2003hammingcube, peled2010high}. The works~\cite{benjamini1994tree, Benjamini2000, benjamini2000upper, loebl2003note, benjamini2007random, erschler2009random, peled2012expanders, peled2013grounded, peled2014random, wu2016average, bok2018graphspecial, bok2018algorithmic, bok2018graph, berger2019distribution} explore the properties of homomorphism height functions (and related Lipschitz functions) on general graphs.

\subsubsection{Implication for colorings}\label{sec:implication on coloring}
When trying to deduce information on the distribution of the modulo $3$ of $h_L(0,0)$ from Theorem~\ref{thm:delocalization} one is naturally led to consider the regularity of the distribution of $h_L(0,0)$. The following fact shows that the distribution is log-concave in a natural sense and thus cannot be too irregular. Log-concavity of the single-site marginal distributions is proved for homomorphism height functions by Kahn~\cite[Proposition 2.1]{Kahn2001hypercube} and established more generally for random surfaces with nearest-neighbor convex potentials by Sheffield~\cite[Lemma 8.2.4]{Sheffield2005random}. We repeat the argument given in~\cite[Proposition 7.1]{chandgotia2018delocalization}.
\begin{lemma}(Log-concavity)\label{lem:log-concavity} Let $\Lambda\subset\Z^d$ be finite and $\tau:\intB\Lambda\to\Z$ be such that there exist homomorphism height functions on $\Lambda$ which equal $\tau$ on $\intB\Lambda$. Let $h$ be sampled uniformly from the set of such homomorphism height functions. Then
\begin{equation}\label{eq:log-concave}
  \P(h(v) = i)^2\ge \P(h(v) = i+2j)\P(h(v) = i-2j)
\end{equation}
for any $v\in\Lambda$, any integer $j$ and any integer $i$ of the same parity as $\|v\|_1$.
\end{lemma}
\begin{proof}
  Fix an integer $i$ with the same parity as $\|v\|_1$ and a positive integer $j$. For an integer $k$, let $H_k$ be the set of homomorphism height functions on $\Lambda$ which coincide with $\tau$ on $\intB\Lambda$ and equal $k$ at $v$. To prove~\eqref{eq:log-concave} it suffices to build an injection from $H_{i+2j}\times H_{i-2j}$ to $H_i\times H_i$.

Let $h^+\in H_{i+2j}$ and $h^-\in H_{i-2j}$. Let $\Lambda'\subset\Z^d$ be the largest connected set containing $v$ on which $h^+>h^-+2j$. As $h^+ = h^-$ on $\intB\Lambda$ we must have that $\Lambda'\cap \intB\Lambda = \emptyset$. We may thus define homomorphism height functions $h, h'\in H_i$ by
\begin{align*}
  &h_{\Lambda'}:=(h^+-2j)_{\Lambda'},\quad h_{\Lambda\setminus{\Lambda'}}:=h^-_{\Lambda\setminus{\Lambda'}},\\
  &h'_{\Lambda'}:=(h^-+2j)_{\Lambda'},\quad h'_{\Lambda\setminus{\Lambda'}}:=h^+_{\Lambda\setminus{\Lambda'}}.
\end{align*}
Furthermore, $\Lambda'$ can be recovered from the pair $(h, h')$ as the largest connected set containing $v$ on which $h>h'-2j$. Thus the map $(h^+, h^-)\mapsto (h, h')$ is injective.
\end{proof}

It is straightforward to conclude from Lemma~\ref{lem:log-concavity} and Theorem~\ref{thm:delocalization} that
\begin{equation}\label{eq:unlikeliness of specific heights}
  \lim_{L\to\infty} \P(h_L(0,0) = i) = 0\quad\text{for each $i$}.
\end{equation}
We now have all the tools necessary to conclude the asymptotic uniformity of the modulo $3$ of $h_L(0,0)$.
\begin{corollary}
\begin{equation}
  \lim_{L\to\infty} \P(h_L(0,0)\equiv i\!\!\!\!\pmod{3}) = \frac{1}{3}\quad\text{for each $i$}.
\end{equation}
\end{corollary}
\begin{proof}
Fix $L$ and denote $p_i := \Pr(h_L(0,0)=2i)$ for $i \in \Z$.
By symmetry, $p_i = p_{-i}$ for all $i$.
Let $q_0,q_1,q_2$ denote the probabilities that $h_L(0,0) \pmod 3$ equals $0,1,2$, respectively. Then
\[ q_0 = p_0 + 2(p_3+p_6+\cdots) \qquad\text{and}\qquad q_1=q_2=(p_1+p_2)+(p_4+p_5)+\cdots .\]
By~\eqref{eq:log-concave}, we have $p_0^2 \ge p_1 p_{-1} = p_1^2$, so that $p_0 \ge p_1$. Now suppose we have already shown that $p_{i-1} \ge p_i$ for some $i \ge 1$. Then using~\eqref{eq:log-concave} again, we have $p_i^2 \ge p_{i-1}p_{i+1} \ge p_i p_{i+1}$, so that $p_i \ge p_{i+1}$ (note that $p_i=0$ implies that $p_{i+1}=0$). Thus, $(p_i)_{i \ge 0}$ is a non-increasing sequence. It follows that $q_0 \le p_0 + q_1$ and $q_0 \ge q_1 - p_1$. Since~\eqref{eq:unlikeliness of specific heights} implies that $p_0 \to 0$ and $p_1 \to 0$ as $L \to \infty$, it follows that $q_0-q_1 \to 0$ as $L \to \infty$. Hence, in the limit as $L \to \infty$, the probabilities $q_0,q_1,q_2$ are equal and sum up to 1, and must therefore all equal $1/3$.
\end{proof}

\subsection{Delocalization of the height function}
\subsubsection{Gibbs measures}
The proof of Theorem~\ref{thm:delocalization} relies on methods from ergodic theory, making use of homomorphism height functions defined on the whole of $\Z^d$. The notion of a uniformly sampled homomorphism height function in the whole of $\Z^d$ is not well defined, and the standard substitute for it is the notion of a \emph{Gibbs measure} (for the uniform specification). A measure $\mu$ on homomorphism height functions on $\Z^d$ is called \emph{Gibbs} if the following holds: Let $f$ be sampled from $\mu$. For each finite subset $V$ of $\Z^d$, almost surely, conditioned on $f|_{V^c}$ the distribution of $f|_V$ is uniform on all extensions of $f|_{V^c}$ to a homomorphism height function. The set of Gibbs measures forms a convex set.

\emph{Example:} Let $h$ be any height function whose modulo $3$ is the frozen proper $3$-coloring discussed in Section~\ref{sec:frozen}. Then the delta measure on $h$ is a Gibbs measure. Our interest, however, is in Gibbs measures with translation-invariance properties as we now describe.

For a sublattice $\mathcal{L}\subset\Z^d$, a Gibbs measure is called \emph{$\mathcal{L}$-translation-invariant} if samples from the measure are invariant in distribution to translations from $\mathcal{L}$. An $\mathcal{L}$-translation-invariant measure is called \emph{$\mathcal{L}$-ergodic} if it gives probability zero or one to each event which is invariant under translations from $\mathcal{L}$.

The set of \emph{$\mathcal{L}$-translation-invariant} Gibbs measures forms a convex set, whose extreme points are exactly the \emph{$\mathcal{L}$-ergodic} Gibbs measures~\cite[Chapter 14]{georgii2011gibbs}.

A (not necessarily invariant) measure is called \emph{extremal} (or tail-trivial) if it assigns probability zero or one to each event which can be determined from the values of the sample outside every finite set. These are exactly the extreme points of the set of all Gibbs measures~\cite[Chapter 7]{georgii2011gibbs}.

\subsubsection{Delocalization}
We proceed to discuss the delocalization of $h_L$ as stated in Theorem~\ref{thm:delocalization}. One may show, by an argument that we do not detail here which makes use of the positive association (FKG) of the absolute value of $h_L$ (proved in~\cite[Proposition 2.3]{Benjamini2000}), that if Theorem~\ref{thm:delocalization} does not hold, i.e.,
\begin{equation}
  \liminf_{L\to\infty} \var(h_L(0,0)) < \infty,
\end{equation}
then the distribution of $h_L$ converges locally as $L\to\infty$ to a $\Ze^2$-translation-invariant Gibbs measure (an analogous statement holds in any dimension~\cite[Theorem 1.1]{chandgotia2018delocalization}). Thus, Theorem~\ref{thm:delocalization} is a consequence of the following statement.
\begin{theorem}\label{thm:no ergodic Gibbs measures}
  There are no $\Ze^2$-ergodic Gibbs measures for two-dimensional homomorphism height functions.
\end{theorem}
To prove this theorem we require the following strong result of Sheffield~\cite{Sheffield2005random}, which applies in much greater generality to two-dimensional random surfaces with nearest-neighbor convex potentials. An alternative proof, for the case of homomorphism height functions, is given in~\cite[Theorem 3.1]{chandgotia2018delocalization}, the main ideas of which are described in Section~\ref{sec:uniqueness-of-ergodic-Gibbs} below.

\begin{theorem}(uniqueness of ergodic Gibbs measures)\label{thm:uniqueness_up_to_additive_constant}
      In dimension $d=2$: Let $\mu, \mu'$ be $\Ze^2$-ergodic Gibbs measures. Then there is an integer $k$ and a coupling of $\mu,\mu'$ such that if $(\tilde{f},\tilde{g})$ are sampled from the coupling then, almost surely,
      \begin{equation}\label{eq:f_equal_g_plus_2k}
        \tilde{f} = \tilde{g} + 2k.
      \end{equation}
    \end{theorem}
    Theorem~\ref{thm:no ergodic Gibbs measures} is derived from this statement with the following additional argument. Suppose, in order to obtain a contradiction, that $\mu$ is a $\Ze^2$-ergodic Gibbs measure. Let $f$ be sampled from $\mu$. Define a homomorphism height function $g$ on $\Z^2$ by
    \begin{equation}\label{eq:g_from_f}
      g(v) := f(-v + (1,0)) - 1.
    \end{equation}
    One checks in a straightforward way that the distribution of $g$ is also a $\Ze^2$-ergodic Gibbs measure, which we denote by $\mu'$. Thus, by Theorem~\ref{thm:uniqueness_up_to_additive_constant}, there exists an integer $k$ and a coupling of $\mu, \mu'$ such that, when sampling $(\tilde{f},\tilde{g})$ from this coupling, the equality \eqref{eq:f_equal_g_plus_2k} holds almost surely. Continuing, we observe that \eqref{eq:g_from_f} implies that
    \begin{equation*}
      f((0,0)) + f((1,0)) = g((1,0)) + g((0,0)) + 2.
    \end{equation*}
    Thus, the equality in distribution
    \begin{equation*}
      \tilde{f}((0,0)) + \tilde{f}((1,0)) \eqd \tilde{g}((0,0)) + \tilde{g}((1,0)) + 2
    \end{equation*}
    also holds.
    However, by \eqref{eq:f_equal_g_plus_2k},
    \begin{equation*}
      \tilde{f}((0,0)) + \tilde{f}((1,0)) = \tilde{g}((0,0)) + \tilde{g}((1,0)) + 4k.
    \end{equation*}
    This implies that $4k = 2$, which contradicts the fact that $k$ is an integer. The contradiction establishes Theorem~\ref{thm:no ergodic Gibbs measures}.

\subsection{Uniqueness of ergodic Gibbs measures}
\label{sec:uniqueness-of-ergodic-Gibbs}
In this section we discuss the main ideas in the proof of Theorem~\ref{thm:uniqueness_up_to_additive_constant} following~\cite[Theorem 3.1]{chandgotia2018delocalization}. The results in Section~\ref{sec:disagreement percolation and cluster swapping} and Section~\ref{sec:positive association and extremality} hold in any dimension $d\ge 1$ while the results of Section~\ref{sec:monotone sequence of percolation configurations} are restricted to dimension $d=2$.

\subsubsection{Disagreement percolation and cluster swapping}\label{sec:disagreement percolation and cluster swapping}
    In order to characterize all $\Ze^2$-ergodic Gibbs measures we need a method to compare two Gibbs measures. The main tool is the following lemma which appeared in~\cite{Sheffield2005random} and applies in all dimensions. For two Gibbs measures $\mu_1, \mu_2$ on homomorphism height functions, we say that $\mu_1$ \emph{stochastically dominates} $\mu_2$ if there exists a coupling of $\mu_1$ and $\mu_2$ such that when $(\tilde f, \tilde g)$ is sampled from the coupling (i.e., $\tilde{f}\sim \mu_1$ and $\tilde{g}\sim\mu_2$) then $\tilde{f}\ge \tilde{g}$ everywhere, almost surely.

    \begin{lemma}\label{lem:stochastic domination}
      Let $\mu_1, \mu_2$ be Gibbs measures and $f,g$ be \emph{independently} sampled from $\mu_1, \mu_2$, respectively. If
      \begin{equation}\label{eq:no_f_less_than_g_infinite_component}
        \P(\text{there is an infinite connected component of $\{f<g\}$}) = 0,
      \end{equation}
      then $\mu_1$ stochastically dominates $\mu_2$.
    \end{lemma}
    The following corollary already appeared in van den Berg~\cite{van1993uniqueness} in the context of Gibbs measures which are Markov random fields.
    \begin{corollary}\label{cor:coupling_to_get_extremality}
      Let $\mu_1, \mu_2$ be Gibbs measures and $f,g$ be \emph{independently} sampled from $\mu_1, \mu_2$, respectively. If
      \begin{equation}\label{eq:no_f_different_from_g_infinite_component}
        \P(\text{there is an infinite connected component of $\{f\neq g\}$}) = 0,
      \end{equation}
      then $\mu_1=\mu_2$ and the measures are \emph{extremal}.
    \end{corollary}
\begin{proof}
Lemma~\ref{lem:stochastic domination} implies that both $\mu_1$ stochastically dominates $\mu_2$ and $\mu_2$ stochastically dominates $\mu_1$, whence $\mu_1 = \mu_2$. We are left to prove that $\mu=\mu_1=\mu_2$ is extremal. Otherwise $\mu=\frac{1}{2}(\nu_1+\nu_2)$ for distinct Gibbs measures $\nu_1$ and $\nu_2$. Thereby with positive probability $(f,g)$ has the distribution of independent samples from $\nu_1$ and $\nu_2$ respectively. As~\eqref{eq:no_f_different_from_g_infinite_component} holds almost surely for $(f,g)$ it follows, with the same argument as in the beginning of the proof, that $\nu_1 = \nu_2$, a contradiction.
\end{proof}

    The proof of Lemma~\ref{lem:stochastic domination} is based on the idea of swapping finite connected components of vertices on which $f\neq g$ (disagreement clusters). A generalization of this method is used in~\cite{Sheffield2005random} to study random surfaces with nearest-neighbor convex potentials and is termed there \emph{cluster swapping}. The interested reader is referred to~\cite[Section 2]{cohen2017rarity} for a survey of related ideas.
    \begin{lemma}\label{lem:finite_cluster_swap}
      Let $\mu_1, \mu_2$ be Gibbs measures and $f,g$ be \emph{independently} sampled from $\mu_1, \mu_2$, respectively. Define a new pair of homomorphism height functions $(\tilde{f}, \tilde{g})$ as follows:
      \begin{equation}\label{eq:finite_cluster_swap_def}
        (\tilde{f},\tilde{g})(v) = \begin{cases}
          (g,f)(v)& \text{there is a \emph{finite} connected component of $\{f\neq g\}$ containing $v$},\\
          (f,g)(v)& \text{otherwise}.
        \end{cases}
      \end{equation}
      Then $(\tilde{f}, \tilde{g})$ has the same distribution as $(f,g)$.
    \end{lemma}

    \begin{proof}[Proof of Lemma~\ref{lem:finite_cluster_swap}]
    Define $(\tilde{f}_n, \tilde{g}_n)$ for $n\ge 1$ by
    \begin{equation}\label{eq:finite cluster swap in finite volume def}
        (\tilde{f}_n,\tilde{g}_n)(v) = \begin{cases}
          (g,f)(v)& \substack{\text{\normalsize there is a finite connected component}\\\text{\normalsize $C\subset\Lambda(n)$ of $\{f\neq g\}$ containing $v$}},\\
          (f,g)(v)& \text{otherwise},
        \end{cases}
      \end{equation}
      with $\Lambda(n)$ defined in~\eqref{eq:Lambda L def}.
      As $(\tilde{f}_n, \tilde{g}_n)$ converges to $(\tilde{f}, \tilde{g})$ pointwise as $n\to\infty$, it suffices to show that $(\tilde{f}_n, \tilde{g}_n)$ has the same distribution as $(f,g)$ for each $n$.

      Fix $n$. By definition, $(\tilde{f}_n,\tilde{g}_n)_{\Lambda(n)^c} = (f,g)_{\Lambda(n)^c}$. In addition, conditioned on $(f,g)_{\Lambda(n)^c}$, the distribution of $(f,g)_{\Lambda(n)}$ is uniform over all pairs of homomorphism height functions extending the boundary conditions, as $f$ and $g$ are sampled independently from Gibbs measures. It thus suffices to prove that this latter uniformity statement holds also for $(\tilde{f}_n, \tilde{g}_n)$. This now follows from the straightforward fact that conditioned on $(f,g)_{\Lambda(n)^c}$ (which equals $(\tilde{f}_n,\tilde{g}_n)_{\Lambda(n)^c}$), the definition \eqref{eq:finite cluster swap in finite volume def} yields a bijection (in fact, an involution) between the homomorphism pairs $(f,g)_{\Lambda(n)}$ and the homomorphism pairs $(\tilde{f}_n, \tilde{g}_n)_{\Lambda(n)}$ which extend the boundary conditions.
    \end{proof}

    \begin{proof}[Proof of Lemma~\ref{lem:stochastic domination}]
    Let $(f,g)$ be as in the lemma and $(\tilde{f}, \tilde{g})$ be defined by~\eqref{eq:finite_cluster_swap_def} so that $(f,g)$ and $(\tilde{f}, \tilde{g})$ have the same distribution. In particular, $g$ has the same distribution as $\tilde{g}$. Thus, $f$ and $\tilde{g}$ form a coupling of $\mu_1$ and $\mu_2$ which, by its definition and the assumption~\eqref{eq:no_f_less_than_g_infinite_component} satisfies that $f\ge\tilde{g}$ everywhere, almost surely.
    \end{proof}

\subsubsection{Positive association and extremality of ergodic Gibbs measures}\label{sec:positive association and extremality}

We place a \emph{partial order} on functions $f:\Z^d\to\R$ by saying that $f_1\succeq f_2$ if $f_1(v)\ge f_2(v)$ for all $v$. A mapping $T$ from such functions to $\R$ is called \emph{increasing} if $T(f_1)\ge T(f_2)$ whenever $f_1\succeq f_2$. A measure $\mu$ on functions $f:\Z^d\to\R$ is said to be \emph{positively associated} if when $f$ is sampled from the measure and $T_1,T_2$ are bounded, measurable and increasing then $\E(T_1(f) T_2(f))\ge \E(T_1(f))\E(T_2(f))$.

The well-known FKG inequality~\cite{fortuin1971correlation} immediately implies the following:
Let $\Lambda\subset\Z^d$ be finite and $\tau:\intB\Lambda\to\Z$ be such that there exist homomorphism height functions on $\Lambda$ which equal $\tau$ on $\intB\Lambda$. Then the uniform measure on these homomorphism height functions is positively associated.

Unfortunately, it does not follow from the above fact that all Gibbs measures are also positively associated (see~\cite[Example 6.64]{friedli2017statistical} for an example due to Miyamoto in the context of the Ising model). Still, one may deduce that all \emph{extremal} Gibbs measures are positively associated. This makes the following fact valuable.
\begin{lemma}\label{lem:ergodic measures are extremal}
  In every dimension $d$: Every $\Ze^d$-ergodic Gibbs measure is extremal.
\end{lemma}
We do not provide a proof of the lemma here (see~\cite[Section 5]{chandgotia2018delocalization}) and content ourselves with a description of the main steps.

Let $f,g$ be independently sampled from the same $\Ze^d$-ergodic Gibbs measure. Set $\sigma^+:=\{v\in\Z^d\colon f(v)>g(v)\}$, $\sigma^-:=\{v\in\Z^d\colon f(v)<g(v)\}$ and let $E^+$ and $E^-$ be the events that $\sigma^+$ and $\sigma^-$ have an infinite connected component, respectively. By Corollary~\ref{cor:coupling_to_get_extremality}, it suffices to prove that
\begin{equation*}
\P(\text{there is an infinite connected component of $\{f\neq g\}$}) = 0,
\end{equation*}
which is itself implied by showing that $\P(E^+) = \P(E^-) = 0$.

As the first step, a standard theorem in percolation theory, applicable to translation-invariant percolation measures on $\Z^d$ (more generally, on amenable graphs) satisfying a finite energy condition (see~\cite{Burton1989density}), is that a percolation configuration can have at most one infinite connected component, almost surely. The theorem is applicable to $\sigma^+$ and $\sigma^-$ (strictly speaking, these percolation configurations do not satisfy the finite energy condition, but a suitable replacment may be devised, see~\cite[Section 4.4]{chandgotia2018delocalization}).

As a second step, it is shown that $\P(E^+\cap E^-)=0$. Indeed, on the event $E^+ \cap E^-$, one may apply a cluster swapping operation, of the same nature as in the previous section, to swap $f$ and $g$ on the infinite connected component of $\sigma^-$. Swapping on an infinite connected component no longer needs to preserve the joint distribution of $(f,g)$, but it may be shown that it preserves their joint Gibbs property, i.e., the fact that they are sampled from a Gibbs measure of the product specification. After the swapping operation, the percolation configuration where $f>g$ has two distinct infinite connected components, a contradiction to the first step.

Lastly, assume in order to get a contradiction that $\P(E^+)>0$ (the case that $\P(E^-)>0$ is similar), so that, by the previous step, $\P(E^+\setminus E^-)>0$. Let $C^+$ be the infinite connected component of $\sigma^+$ and note that $C^+$ has positive density, almost surely on the event $E^+$, by the translation invariance of $\sigma^+$. It follows from the proof of Lemma~\ref{lem:finite_cluster_swap} that there is a coupling $(\tilde{f},\tilde{g})$ of $f,g$ such that, on the event $(E^-)^c$, $\tilde{f}\ge \tilde{g}$ everywhere and $\tilde{f}>\tilde{g}$ on $C^+$. In particular, there is positive probability that $\tilde{f}\ge \tilde{g}$ everywhere, with a strict inequality holding on a positive density set. However, this contradicts the fact that $\tilde{f}$ and $\tilde{g}$ are sampled from the same $\Ze^d$-ergodic distribution (since for each integer $k$, the densities of the sets where $\tilde{f}=k$ and where $\tilde{g}=k$ are almost surely equal).

\subsubsection{A monotone sequence of percolation configurations}\label{sec:monotone sequence of percolation configurations}
    We describe the remaining ideas in the proof of Theorem~\ref{thm:uniqueness_up_to_additive_constant} following~\cite[Section 6]{chandgotia2018delocalization}).

    Fix the dimension $d=2$ throughout the section. Let $f,g$ be independent samples from $\Ze^2$-ergodic Gibbs measures $\mu, \mu'$. Our goal is to show that there is an integer $k_0$ such that the distribution of $f$ equals the distribution of $g + 2k_0$.

    Define the sequence of percolation configurations $(\sigma^k)$, $k$ integer, by
    \begin{equation*}
      \sigma^k(v):=\1_{f(v)\ge g(v) + 2k}.
    \end{equation*}
    By our definitions, $\sigma^k$ is $\Ze^2$-translation-invariant for every $k$ and $\sigma^k$ decreases with $k$. Moreover, Lemma~\ref{lem:ergodic measures are extremal} implies that $\mu$ and $\mu'$ are extremal and thus positively associated, which implies that, for every $k$, $\sigma^k$ is also positively associated and extremal (in the sense that every event which is measurable with respect to $\sigma^k$ and invariant under changing finitely many of the values of $\sigma^k$ has probability zero or one).

    For integer $k$ and $s\in\{0,1\}$ set $E_k^s$ to be the event that there is an infinite connected component on which $\sigma^k = s$. The above properties imply that for each $k$,
    \begin{align}
      &\P(E_k^0),\, \P(E_k^1)\in \{0,1\},\label{eq:E_+_E_-_zero_one}\\
      &\P(E_{k+1}^0)\ge \P(E_k^0)\quad\text{and}\quad \P(E_{k+1}^1)\le \P(E_k^1).\label{eq:E_+_E_-_monotonocity}
    \end{align}
    To this we add a fact, which relies on the planarity of $\Z^2$, stating that for each $k$,
    \begin{equation}
      \P(E_k^0\cap E_k^1) = 0.\label{eq:no coexistence}
    \end{equation}
    This fact is a consequence of the invariance of $\sigma^k$, its positive association and the uniqueness of infinite connected components discussed in Section~\ref{sec:positive association and extremality}. General results of this kind are provided in~\cite[Theorem 9.3.1 and Corollary 9.4.6]{Sheffield2005random} or~\cite[Theorem 1.5]{duminil2019sharp} though in our case a simpler argument of Zhang which utilizes additional symmetries of $\sigma^k$ may be used~\cite[Theorem 14.3]{Haggstrom2006uniqueness}.

    Combining the relations~\eqref{eq:E_+_E_-_zero_one}, \eqref{eq:E_+_E_-_monotonocity} and~\eqref{eq:no coexistence}, and exchanging the roles of $f$ and $g$ if necessary, we see that one of the following cases must occur:
    \begin{enumerate}
      \item There exists an integer $k_0$ such that $\P(E^1_{k_0+1})=0$ and $\P(E^0_{k_0}) = 0$.
      \item For all integer $k$, $\P(E^0_k) = 0$.
    \end{enumerate}
    Note, however, that if $\P(E^0_k) = 0$ for some $k$ then the distribution of $f$ stochastically dominates the distribution of $g+2k$ by Lemma~\ref{lem:stochastic domination}. Thus the second case implies that $f$ stochastically dominates $g+2k$ for all integer $k$, which cannot occur. Suppose then that the first case occurs for some integer $k_0$. Applying Corollary~\ref{cor:coupling_to_get_extremality} shows that the distribution of $f$ equals the distribution of $g+2k_0$, completing the proof of Theorem~\ref{thm:uniqueness_up_to_additive_constant}.

\section{Lecture 3 -- Long-range order}
Recall the three types of behavior highlighted in Section~\ref{sec:concrete questions} for the quantity~\eqref{eq:probability of equal corners}. In the first lecture we have proved that uniformly sampled proper $q$-colorings of $\Z^d$ are disordered when $q$ is large compared with $d$. In this lecture we study an opposite regime, in which $q$ is small compared with $d$, and discuss phenomena of long-range order.

The technique used to establish the disordered regime (Dobrushin's uniqueness condition) applies to many probabilistic models. The techniques used in the second lecture to discuss criticality are more specific. The techniques of this lecture, though described here for the specific case of proper $q$-colorings, again admit extensions to a wide class of models (see~\cite{peled2017condition,peledspinka2018spin}).

\subsection{Long-range order}\label{sec:long-range order}
Proper $2$-colorings exhibit long-range order in all dimensions. Can long-range order occur for any higher value of $q$? As Dobrushin's uniqueness condition, Theorem~\ref{thm:Dobrushin uniqueness}, applies when $q>4d$, we see that the relevant parameter range for long-range order is $q$ small compared with $d$. In ferromagnetic systems like the Ising model, the system orders by setting most spins to the same state (see Section~\ref{sec:Peierls argument}). What kind of ordered structure can arise here?

In estimating the number of proper $q$-colorings of a box $\Lambda\subset\Z^d$ the following argument may be used. Partition the $q$ colors into two subsets $A,B$ and consider the family of colorings obtained by coloring sites in the even sublattice with colors from $A$ and sites in the odd sublattice with colors from $B$. When $\Lambda$ has an equal number of even and odd sites this gives $(|A|\cdot |B|)^{|\Lambda|/2}$ colorings, and this quantity is maximized when $\{|A|,|B|\}=\{\lfloor \frac{q}{2}\rfloor, \lceil \frac{q}{2}\rceil\}$. Certainly most colorings are not obtained this way, but could it be that, when $q$ is small compared with $d$, most colorings coincide with such a ``pure $(A,B)$-coloring'' at \emph{most vertices}? This is the idea behind the following results which are proved in~\cite{peled2018rigidity}. The idea is formalized in two ways: In finite volume by prescribing suitable boundary conditions, and in infinite volume by describing maximal-entropy invariant Gibbs measures.

We state our first result following required notation. A \emph{pattern} is a pair $(A,B)$ of disjoint subsets of $\{1,\ldots, q\}$ (we stress that $(A,B)$ and $(B,A)$ are distinct patterns). It is called \emph{dominant} if $\{|A|,|B|\}=\left\{\lfloor \tfrac{q}{2}\rfloor, \lceil \tfrac{q}{2}\rceil\right\}$.
A \emph{domain} is a non-empty finite $\Lambda \subset \Z^d$ such that both $\Lambda$ and $\Z^d\setminus\Lambda$ are connected. Its \emph{internal vertex-boundary}, denoted $\intB \Lambda$, is the set of vertices in $\Lambda$ adjacent to a vertex outside $\Lambda$ (see~\eqref{eq:internal vertex boundary}).
Given a proper $q$-coloring $f$, we say that
\begin{align*}
&\text{a vertex $v$ is \emph{in the $(A,B)$-pattern} if}\\
&\text{either $v$ is even and $f(v) \in A$, or $v$ is odd and $f(v) \in B$}.
\end{align*}
We also say that a set of vertices is in the $(A,B)$-pattern if all its elements are such.

\begin{theorem}\label{thm:long-range-order}
There exist $C,c>0$ such that the following holds for any number of colors $q \ge 3$ and any dimension
\begin{equation}\label{eq:dim-assump}
d \ge Cq^{10} \log^2 q.
\end{equation}
Let $\Lambda\subset\Z^d$ be a domain and let $(A,B)$ be a dominant pattern. Let $\Pr_{\Lambda,(A,B)}$ be the uniform measure on proper $q$-colorings $f$ of $\Lambda$ satisfying that $\intB \Lambda$ is in the $(A,B)$-pattern. Then
\begin{equation}\label{eq:main_thm_bound}
 \Pr_{\Lambda,(A,B)}\big(v\text{ is not in the $(A,B)$-pattern}\big) \le \exp\left(-\frac{cd}{q^3(q+\log d)}\right),\qquad v\in\Lambda.
\end{equation}
\end{theorem}

It is natural to wonder whether other restrictions on the boundary values besides the one used in Theorem~\ref{thm:long-range-order} would lead to other behaviors of the coloring in the bulk of the domain. This idea is captured by the notion of a Gibbs measure: a probability measure on proper $q$-colorings of $\Z^d$ for which the conditional distribution of the coloring on any finite set, given the coloring outside the set, is uniform on the proper colorings extending the boundary values. As discussed before, frozen configurations give rise to trivial Gibbs measures, supported on a single frozen configuration. To avoid such degenerate situations, one often restricts attention to \emph{maximal entropy} Gibbs measures
-- Gibbs measures invariant under translations by a full-rank sublattice of $\Z^d$, termed \emph{periodic} Gibbs measures, whose measure-theoretic entropy equals the topological entropy of proper $q$-colorings. Let us define the latter terms precisely. The \emph{topological entropy} of proper $q$-colorings is defined as
\[ h_{\text{top}} := \lim_{n \to \infty} \frac{\log \big|\Omega^{\text{free}}_{\{1,\ldots, n\}^d}\big|}{n^d}\]
with $\Omega^{\text{free}}_\Lambda$ the set of proper $q$-colorings of $\Lambda$\footnote{Note that every coloring in $\Omega^{\text{free}}_{\{1,\dots,n\}^d}$ may be extended to a proper coloring of all of $\Z^d$, e.g., by iterated reflections, so that this set coincides with the set of all ``globally admissible'' proper colorings of $\{1,\dots,n\}^d$.}, and where the above limit exists by subadditivity. The \emph{measure-theoretic entropy} (also known as Kolmogorov--Sinai entropy) of a periodic measure $\mu$ supported on proper $q$-colorings of $\Z^d$ is
\[ h(\mu) := \lim_{n \to \infty} \frac{\Ent(\mu|_{\{1,\ldots,n\}^d})}{n^d} ,\]
with $\mu|_{\Lambda}$ being the marginal distribution of $\mu$ on $\Lambda$, with $\Ent$ standing for Shannon's entropy (see Section~\ref{sec:entropy}), and with the limit  existing by subadditivity. Since Shannon entropy is maximized by the uniform distribution, it follows that $h(\mu) \le h_{\text{top}}$ for any such $\mu$.
The variational principle tells us that equality is achieved by some $\mu$. Any such $\mu$ is said to be of \emph{maximal entropy}.
A theorem of Lanford--Ruelle (see, e.g.,~\cite{misiurewicz1976short}) tells us that every measure of maximal entropy is also a Gibbs measure (so that there is some redundancy when speaking about a maximal-entropy Gibbs measure). We stress that a measure of maximal entropy is, by definition, always assumed to be periodic.

A concrete question, which has received significant attention in the literature (see Section~\ref{sec:remarks on long-range order}), is to determine whether \emph{multiple} Gibbs measures of maximal entropy exist for any number of colors~$q$, when the dimension~$d$ is sufficiently high. In fact, Theorem~\ref{thm:long-range-order} implies the existence of multiple Gibbs measures, one for each dominant pattern $(A,B)$, and it is not overly difficult to establish that these have maximal entropy. This fact, along with additional properties, is formulated in the following result.

\begin{theorem}\label{thm:existence_Gibbs_states}
Let $q \ge 3$ and suppose that the dimension $d$ satisfies~\eqref{eq:dim-assump}. For each dominant pattern $(A,B)$ there exists a Gibbs measure $\mu_{(A,B)}$ such that, for any sequence of domains $\Lambda_n$ increasing to $\Z^d$, the measures $\Pr_{\Lambda_n,(A,B)}$ converge weakly to $\mu_{(A,B)}$ as $n \to \infty$. In particular, $\mu_{(A,B)}$ is invariant to automorphisms of $\Z^d$ preserving the two sublattices. Moreover, the $(\mu_{(A,B)})$ are distinct, extremal and of maximal entropy.
\end{theorem}

Together with Theorem~\ref{thm:long-range-order} we see that the Gibbs measure $\mu_{(A,B)}$ has a tendency towards the $(A,B)$-pattern at all vertices. The proof yields additional facts, that large spatial deviations from the $(A,B)$-pattern are exponentially suppressed and that the measure $\mu_{(A,B)}$ is strongly mixing with an exponential rate (see~\cite{peled2018rigidity} for exact definitions and proofs).

Theorem~\ref{thm:existence_Gibbs_states} shows that there are at least $\binom{q}{q/2}$ extremal maximal-entropy Gibbs measures for even~$q$ and $2\binom{q}{\lfloor q/2 \rfloor}$ such Gibbs measures for odd $q$. The following result shows that these exhaust all possibilities.
\begin{theorem}\label{thm:characterization_of_Gibbs_states}
  Let $q \ge 3$ and suppose that the dimension $d$ satisfies~\eqref{eq:dim-assump}. Then any (periodic) maximal-entropy Gibbs measure is a mixture of the measures $\{\mu_{(A,B)}\}$.
\end{theorem}

\subsection{Remarks on the main results}\label{sec:remarks on long-range order}
In the physics literature, to the authors' knowledge, the problem was first considered by Berker--Kadanoff~\cite{berker1980ground} who suggested in 1980 that a phase with algebraically decaying correlations may occur at low temperatures (including zero temperature) with fixed $q$ when $d$ is large. This prediction was challenged by numerical simulations and an $\varepsilon$-expansion argument of Banavar--Grest--Jasnow~\cite{banavar1980ordering} who predicted a Broken-Sublattice-Symmetry (BSS) phase at low temperatures for the $3$ and $4$-state models in three dimensions. The BSS phase is exactly of the type proved to occur here, with a global tendency towards a pure $(A,B)$-ordering for a dominant pattern $(A,B)$. Koteck\'y~\cite{kotecky1985long} argued for the existence of the BSS phase at low temperature when $q=3$ and $d$ is large by analyzing the model on a decorated lattice. This prediction became known as \emph{Koteck\'y's conjecture}.

In the mathematically rigorous literature, Koteck\'y's conjecture remained open for 25 years until it was finally answered by the first author~\cite{peled2010high} and by Galvin--Kahn--Randall--Sorkin~\cite{galvin2012phase} (following closely-related papers by Galvin--Randall~\cite{galvin2007torpid} and Galvin--Kahn~\cite{galvin2004phase}). Its extension to the low-temperature anti-ferromagnetic $3$-state Potts model was resolved by Feldheim and the second author~\cite{feldheim2015long}. The results of \cite{peled2010high, galvin2012phase} correspond to the $q=3$ case of Theorem~\ref{thm:long-range-order}, and to the existence of $6$ extremal maximal-entropy Gibbs states which results from it (the fact that the measures have maximal entropy is shown in~\cite[Section~5]{galvin2012phase}), while the characterization result given in Theorem~\ref{thm:characterization_of_Gibbs_states} is new also for this case (the convergence result in Theorem~\ref{thm:existence_Gibbs_states} is shown in~\cite{feldheim2015long} for this case). Periodic boundary conditions were considered in~\cite{galvin2007torpid, feldheim2013rigidity} and in~\cite{peled2010high} for the corresponding height function (also on tori with non-equal side lengths).

Engbers and Galvin~\cite{engbers2012h2} establish long-range order on hypercube graphs, $\{1,\ldots, n\}^d$ for fixed $n$ and $d$ tending to infinity, for the wide class of weighted graph homomorphism models, which includes the proper $q$-coloring model as a special case. The methods used to prove the theorems of Section~\ref{sec:long-range order} admit extensions to this class, as well as to more general discrete spin systems with nearest-neighbor interactions, though under the assumption that all dominant patterns, suitably defined, are equivalent (a form of symmetry condition); see~\cite{peled2017condition,peledspinka2018spin} for details. This shows, for instance, that the long-range order established for proper $q$-colorings persists to the low-temperature anti-ferromagnetic $q$-state Potts model, even for temperatures \emph{growing} as a power of the dimension.

The results of Section~\ref{sec:long-range order} are not valid in low dimensions due to Dobrushin's uniqueness condition. Nonetheless, they are applicable in any dimension $d\ge 2$ provided the underlying graph is suitably modified. Precisely, the results remain true when $\Z^d$ is replaced by a graph of the form $\Z^{d_1}\times\T_{2m}^{d_2}$, $m\ge 1$ integer, provided $d_1\ge 2$ and $d=d_1+d_2$ satisfies \eqref{eq:dim-assump}, where $\T_{2m}$ is the cycle graph on $2m$ vertices. The graph $\Z^{d_1}\times\T_{2m}^{d_2}$ may be viewed as a subset of $\Z^d$ in which the last $d_2$ coordinates are restricted to take value in $\{0,1,\ldots, 2m-1\}$ and are endowed with periodic boundary conditions. In this sense, it is only the local structure of $\Z^d$ which matters to the results.

The emergent long-range order is lattice dependent. Irregularities in the lattice (i.e., having different sublattice densities) often promote the formation of order. This may be used, for instance, to find for each $q$ a \emph{planar} lattice on which the proper $q$-coloring model is ordered~\cite{huang2013two}. However, irregularities also modify the nature of the resulting phase, leading to long-range order in which a single color appears on most of the lower-density sublattice~\cite{kotecky2014entropy}, or to partially ordered states~\cite{qin2014partial}. As an illustration of this~\cite{kotecky1985long, huang2013two}, consider replacing each edge of a domain in $\Z^2$ by a large number $M$ of parallel paths of length $2$. On this graph, the restriction of a uniformly sampled proper $q$-coloring to the vertices of the original lattice is a \emph{ferromagnetic} $q$-state Potts model at low temperature. Thus, when $q$ is small compared to $M$, a single color will be assigned to most vertices of the original lattice, resulting in more available colors for the vertices on the added paths.

\subsection{Overview of the proof of long-range order}
\label{sec:proof-overview}

In this section, we give a high-level view of the proof of Theorem~\ref{thm:long-range-order}. The basic methodology used is based on the classical \emph{Peierls argument} which was introduced in order to establish long-range order in the low-temperature Ising model. It is instructive to review the key steps in this argument before discussing the significantly more complicated case of proper colorings.
We thus begin in Section~\ref{sec:Peierls argument} with an exposition of the classical Peierls argument for the Ising model, which consists of three key steps. We then describe in Section~\ref{sec:difficulties} the difficulties that arise in applying this type of approach to the proper coloring model. We then give details on each of these three steps in Sections~\ref{sec:ordered-regions}, \ref{sec:proof-overview-breakup} and~\ref{sec:approx}.

We use the following notation. Let $U\subset\Z^d$ be a set. We let $\partial U$ denote its edge-boundary, $N(U)$ denote its neighborhood (vertices adjacent to some vertex in $U$), $U^+ := U \cup N(U)$ denote its $1$-extension, $\intB U := U \cap N(U^c)$ denote its internal vertex boundary (as in~\eqref{eq:internal vertex boundary}), $\extB U := N(U) \setminus U$ denote its external vertex boundary and $\intextB U := \intB U \cup \extB U$ denote both boundaries. We say that $U$ is an even (odd) set if $\intB U$ is contained in the even (odd) sublattice of $\Z^d$. An even (odd) set $U$ is called \emph{regular} if both it and its complement contain no isolated vertices. Let $\Even$ ($\Odd$) denote the set of even (odd) vertices in $\Z^d$.

\subsubsection{The Ising model and the classical Peierls argument}\label{sec:Peierls argument}
We review here the key steps in the classical Peierls argument~\cite{peierls1936ising} used to establish long-range order in the low-temperature Ising model. Other reviews are given in~\cite[Section 2.5]{peled2019lectures} and~\cite[Section~3.7.2]{friedli2017statistical}.

The Ising model at inverse temperature $\beta\ge 0$ on a domain $\Lambda\subset\Z^d$ is the probability measure on the configuration space $\Omega:=\{\sigma:\Lambda\to\{-1,1\}\}$ defined by
\begin{equation}\label{eq:Ising model def}
  \P^{\text{Ising}}_{\beta, \Lambda}(\sigma):=\frac{1}{Z_{\beta,\Lambda}} \exp\left[-\beta\sum_{\substack{u,v\in\Lambda\\ u\sim v}} \sigma(u)\sigma(v)\right].
\end{equation}
At zero temperature, i.e., in the limit $\beta\to\infty$, the model is supported on the two constant configurations. Constant configurations play an analogous role in the Ising model to the role played ``pure $(A,B)$-colorings'', with $(A,B)$ a dominant pattern, in the proper $q$-coloring model. The Ising model analogue to Theorem~\ref{thm:long-range-order} is that at low temperature, when conditioning the configuration to take the same value on all boundary vertices of $\Lambda$, the value at each of the interior vertices gains a significant bias towards the boundary value, uniformly in the domain $\Lambda$. To state this precisely, let \begin{equation*}
  \Omega_+:=\{\sigma\in\Omega\colon \sigma|_{\intB \Lambda}\equiv 1\}
\end{equation*}
and let $\P^{\text{Ising}}_{\beta, \Lambda, +}$ denote the measure $\P^{\text{Ising}}_{\beta, \Lambda}$ conditioned on $\sigma\in\Omega_+$.
\begin{theorem}(ordering at low-temperature for the Ising model)\label{thm:low temperature order Ising}
  There exists $C>0$ such that for all dimensions $d\ge 2$, inverse temperature $\beta\ge \frac{C\log d}{d}$, domains $\Lambda\subset\Z^d$ and $v\in\Lambda$,
  \begin{equation}\label{eq:Ising ordering}
    \P^{\text{Ising}}_{\beta, \Lambda, +}(\sigma(v) = -1)\le \frac{1}{4}.
  \end{equation}
\end{theorem}
We remark that the assumption of low temperature is required for the conclusion. Indeed, a calculation shows that Dobrushin's uniqueness condition (as given in Theorem~\ref{thm:Dobrushin uniqueness}) is satisfied when $e^{4\beta}<1+\frac{2}{d-1}$ so that in this regime the probability in~\eqref{eq:Ising ordering} equals $\tfrac{1}{2}$ plus a factor which decays exponentially in the distance of $v$ to $\intB\Lambda$.

We proceed to describe the main steps in the proof of the theorem.

\smallskip
\noindent
\textbf{Ordered regions, contours and domain walls:} Given a configuration $\sigma\in\Omega$, one may consider the regions
\begin{equation*}
  Z_-:=\{v\in\Lambda\colon \sigma(v) = -1\}\quad\text{and}\quad Z_+:=\{v\in\Lambda\colon \sigma(v) = 1\}.
\end{equation*}
These may be considered as ordered regions for $\sigma$ and our focus is on the the edges separating vertices in $Z_-$ and $Z_+$. Identifying a vertex $v$ in $\Z^d$ with the cube $v + [-\tfrac{1}{2},\tfrac{1}{2}]^d$ allows to identify the edge between adjacent $v,w$ with the $(d-1)$-dimensional face common to the cubes of $v$ and $w$. Such faces are termed \emph{plaquettes} and with this identification we can think of the edges between $Z_-$ and $Z_+$ as forming a collection of $(d-1)$-dimensional closed surfaces separating ordered regions in the configuration $\sigma$.

A \emph{contour} is the edge boundary $\partial U$ of a domain $U\subset\Z^d$. With the above identification, it may be thought of as a $(d-1)$-dimensional surface. The contour is said to be a \emph{$-+$ domain wall} in $\sigma$ if $\sigma|_{\intB U}\equiv -1$ and $\sigma|_{\extB U}\equiv 1$ (assuming also that $U^+\subset\Lambda$ for these to be well defined). If $\sigma\in\Omega_+$ then in order to have $\sigma(v)=-1$ there must exist at least one contour $\partial U$ with $v\in U$ which is a $-+$ domain wall in $\sigma$. Thus we have
\begin{equation}\label{eq:Ising probability of excitation}
  \P^{\text{Ising}}_{\beta, \Lambda, +}(\sigma(v) = -1)\le \sum_{\substack{U\text{ domain,}\\v\in U\text{ and }U^+\subset\Lambda}}\P(\partial U\text{ is a $-+$ domain wall in $\sigma$}).
\end{equation}

\smallskip
\noindent
\textbf{The probability that a contour is a domain wall:} Let $U$ be a domain with $U^+\subset\Lambda$. For a configuration $\sigma\in\Omega$ define a new configuration $\sigma^U$ by
\begin{equation}\label{eq:ising-sign-flip}
\sigma^U(v):=\begin{cases}
  -\sigma(v)&v\in U\\
  \sigma(v)&v\notin U
\end{cases}.
\end{equation}
Thus the transformation $\sigma\mapsto\sigma^U$ flips the sign of $\sigma$ on $U$. This is a bijection (even an involution) on $\Omega$, as we assume that $U$ is fixed.

Now if $\partial U$ a $-+$ domain wall for $\sigma$, one checks directly from the definition~\eqref{eq:Ising model def} that
\begin{equation}\label{eq:ising-domain-wall-cost}
  \frac{\P^{\text{Ising}}_{\beta, \Lambda, +}(\sigma)}{\P^{\text{Ising}}_{\beta, \Lambda, +}(\sigma^U)} = \exp(-2\beta|\partial U|).
\end{equation}
Writing $E_U$ for the set of $\sigma\in\Omega_+$ with $\partial U$ a $-+$ domain wall for $\sigma$, we conclude that
\begin{equation}\label{eq:ising-domain-wall-bound}
  \P^{\text{Ising}}_{\beta, \Lambda, +}(E_U) = \frac{\sum_{\sigma\in E_U}\P^{\text{Ising}}_{\beta, \Lambda, +}(\sigma)}{Z_{\beta,\Lambda}} = \frac{\sum_{\sigma\in E_U}\P^{\text{Ising}}_{\beta, \Lambda, +}(\sigma)}{\sum_{\sigma\in\Omega_+}\P^{\text{Ising}}_{\beta, \Lambda, +}(\sigma)}\le \frac{\sum_{\sigma\in E_U}\P^{\text{Ising}}_{\beta, \Lambda, +}(\sigma)}{\sum_{\sigma\in E_U}\P^{\text{Ising}}_{\beta, \Lambda, +}(\sigma^U)}=\exp(-2\beta|\partial U|).
\end{equation}
This can be interpreted as saying that domain walls are energetically penalized by a factor of $\exp(-2\beta)$ per edge. Substituting this estimate in~\eqref{eq:Ising probability of excitation} shows that
\begin{equation}\label{eq:Ising estimate via contours}
  \P^{\text{Ising}}_{\beta, \Lambda, +}(\sigma(v) = -1)\le \sum_{\ell=1}^\infty N_\ell \exp(-2\beta \ell),
\end{equation}
where $N_\ell$ is the number of domains $U$ with $v\in U$ and $|\partial U| = \ell$ (the condition $U^+\subset\Lambda$ may also be added but is not necessary for the sequel), i.e., the number of contours of length $\ell$ which surround $v$.

\smallskip
\noindent
\textbf{The number of contours of a given length:} To finish the proof of Theorem~\ref{thm:low temperature order Ising} we require an estimate on $N_\ell$. Using our assumption that $d\ge 2$, the boundary $\partial U$ of a domain $U$ may be shown to be connected in a suitable sense (i.e., its plaquettes form a connected $(d-1)$-dimensional surface. See Tim{\'a}r~\cite{timar2013boundary} for combinatorial proofs). In addition, it is straightforward that if $v\in U$ and $|\partial U| = \ell$ then there exists some $0\le k\le \ell$ for which $v + ke_1\in \intB U$ (with $e_1 = (1,0,\ldots, 0)$). It follows, using general results on the number of connected sets containing a given vertex in a graph with given maximal degree~\cite[Chapter 45]{Bol06}, that $N_\ell\le \ell C_d^\ell$ for some constant $C_d>0$ depending only on $d$. In fact, due to the importance of estimating $N_\ell$ in this and other problems, good bounds for the constant $C_d$ have also been determined, with the state of the art due to Lebowitz--Mazel~\cite{lebowitz1998improved} and Balister--Bollob{\'a}s~\cite{balister2007counting} whose works imply that
\begin{equation}\label{eq:bounds on the number of contours}
  \exp\left(\frac{c \log d}{d} \ell\right)\le N_\ell \le \exp\left(\frac{C \log d}{d} \ell\right)
\end{equation}
for positive absolute constants $C,c>0$, with the lower bound holding for even values of $\ell$ which are sufficiently large as a function of $d$. Lastly, it is simple to see that $N_\ell = 0$ if either $\ell$ is odd or $\ell<2d$. Theorem~\ref{thm:low temperature order Ising} now follows with $\beta_0(d)=\frac{C\log d}{d}$, for an absolute constant $C>0$, by plugging the estimate~\eqref{eq:bounds on the number of contours} into~\eqref{eq:Ising estimate via contours}.

The meticulous reader may notice the gap between the disordered regime $e^{4\beta}< 1 + \frac{2}{d-1}$ (which is approximately $\beta<\tfrac{1}{2(d-1)}$ in high dimensions) in which Dobrushin's uniqueness condition is satisfied and the ordered regime $\beta\ge\frac{C\log d}{d}$ in which Theorem~\ref{thm:low temperature order Ising} applies. In fact, the critical $\beta$ for long-range order is asymptotic to $\frac{1}{2d}$ as $d$ tends to infinity. Aizenman, Bricmont and Lebowitz~\cite{aizenman1987percolation} point out that a gap between the critical $\beta$ and the bound on it obtained from the Peierls argument is unavoidable in high dimensions. They point out that the Peierls argument, when it applies, excludes the possibility of \emph{minority percolation}. That is, the possibility that there is an infinite connected component of the value $-1$ in the infinite-volume limit obtained with $+1$ boundary conditions. However, as they show, such minority percolation does occur in high dimensions when $\beta\le \frac{c\log d}{d}$, yielding a lower bound on the minimal inverse temperature at which the Peierls argument applies.

\subsubsection{The difficulties to be addressed}\label{sec:difficulties}
Recall that $(A,B)$ is a dominant pattern if $A,B\subset\{1,\dots,q\}$ are disjoint and satisfy that $\{|A|,|B|\}=\{\lfloor \frac q2 \rfloor, \lceil \frac q2 \rceil\}$. Throughout we fix a domain $\Lambda \subset \Z^d$ and a dominant pattern
\begin{equation}\label{eq:P_0_def}
P_0=(A_0,B_0) \qquad\text{such that}\qquad |A_0| = \lfloor\tfrac{q}{2}\rfloor,~|B_0| =\lceil\tfrac{q}{2}\rceil .
\end{equation}
We consider proper $q$-colorings $f$ chosen from $\Pr_{\Lambda,P_0}$ so that $P_0$ is the ``boundary pattern''. We wish to implement a Peierls-type argument to show long-range order in $f$. To this end, following the steps described in Section~\ref{sec:Peierls argument}, we need to:
\begin{enumerate}
  \item Identify ordered regions in $f$.
  \item Show that the probability of any given set of contours being the domain walls between different ordered regions is exponentially small in their total length.
  \item Sum over contours to conclude that it is unlikely to have a domain wall surrounding a given vertex.
\end{enumerate}
Unfortunately, the method encounters difficulties at each of these steps. We briefly summarize here the issues that need to be addressed and expand on these in the following sections.

\smallskip
\noindent
\textbf{Ordered regions:} Theorem~\ref{thm:long-range-order} suggests that a region is ordered according to a dominant pattern $(A,B)$ if $f$ coincides with a ``pure $(A,B)$ coloring'' in that region. Thus dominant patterns replace the two possibilities for ordering present in the Ising model -- the $+1$ and $-1$ orderings. In the Ising model, a vertex was classified to the ordered regions $Z_+$ and $Z_-$ according to its value. For proper $q$-colorings, this is insufficient. Indeed, if $f(v)=i$ for some even vertex $v$, then $v$ may be in the ordered region of any dominant pattern $(A,B)$ with $i\in A$. This difficulty is addressed by classifying vertices into ordered regions according to the colors assigned to their \emph{neighbors}. This leads to another difficulty, which was not present in the Ising model. The colors assigned to the neighbors of a vertex  may be inconsistent with all dominant patterns, or may still be consistent with more than one dominant pattern. Thus we will also need to allow the possibility of disordered regions and the possibility of overlap between the ordered regions corresponding to different dominant patterns. In addition to these, and for reasons that will be explained below, the ordered region of each dominant pattern $(A,B)$ is defined in such a way that it is an even set if $|A|\le |B|$ and an odd set if $|A|>|B|$ (recall the definitions from the beginning of Section~\ref{sec:proof-overview}). These issues are expanded upon in Section~\ref{sec:ordered-regions}.

\smallskip
\noindent
\textbf{The cost of domain walls:} In the Ising model, we saw that the probability that a given contour of length $\ell$ is a domain wall is at most $\exp(-2\beta\ell)$. Thus domain walls are `penalized' by a factor of $\exp(-2\beta)$ per edge, and this penalty can be strengthened as needed by taking $\beta$ large. For the proper $q$-coloring model we will need to develop a similar bound for the domain walls between ordered regions, for the disordered regions not corresponding to any dominant pattern and for the regions of overlap between different ordered regions. However, the proper $q$-coloring model has no temperature parameter (it is already the zero-temperature limit of the anti-ferromagnetic $q$-state Potts model). Indeed, as the proper $q$-coloring model is a uniform measure on the allowed configurations, the `penalties' on such `bad' regions must be \emph{entropically driven} as opposed to the \emph{energetically driven penalty} in the Ising model. In other words, one needs to show that there are significantly less configurations which are consistent with the presence of a given bad region than the overall number of configurations. Such bounds are proved using entropy inequalities as expanded upon in Section~\ref{sec:proof-overview-breakup}.

\begin{figure}
	\centering
	\includegraphics[scale=0.55]{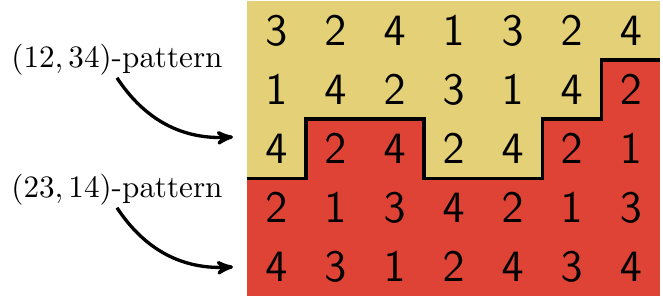}\qquad\qquad
	\includegraphics[scale=0.55]{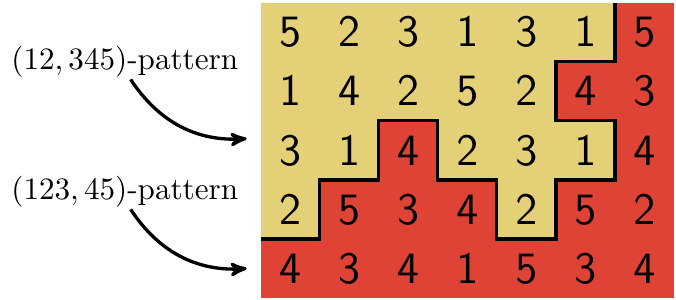}
	\caption{An interface between two regions associated to different dominant patterns for proper $q$-colorings (left: $q=4$, right: $q=5$).}
	\label{fig:interfaces}
\end{figure}

To gain intuition for the `penalty' associated to domain walls, let us analyze the entropic loss in the toy scenario in which the $P_0$-pattern is disturbed by a single `droplet' of a different dominant pattern $P=(A,B)$; see Figure~\ref{fig:interfaces}. 
More precisely, let $\emptyset \neq U\subset\Z^d$ be such that $U^+\subset\Lambda$ and let $n(U)$ be the number of proper colorings of $\Lambda$, for which $U^+$ is in the $P$-pattern and $(\Lambda\setminus U)^+$ is in the $P_0$-pattern. When $q$ is even, we claim that
\begin{equation} \frac{n(U)}{n(\emptyset)} \le \left(\frac{q-2}{q}\right)^{|\intextB U|},
\end{equation}
with equality if and only if $|A_0 \Delta A|=2$.
When $q$ is odd, we claim that
\begin{equation}\label{eq:single droplet in ordered pattern odd q}
\frac{n(U)}{n(\emptyset)} \le \left(\frac{q-1}{q+1}\right)^{\frac{1}{2d} |\partial U|},
\end{equation}
with equality if and only if either $U$ is an odd set and $A_0 \subset A$ or $U$ is an even set and $B_0 \subset B$.
To see these claims, note that in the colorings counted in $n(U)$, the set of allowed colors for an even vertex $u$ is $A$ if $u \in U \setminus \intB U$, is $A_0$ if $u \notin U^+$, and is $A \cap A_0$ if $u \in \intextB U$. Similarly, for an odd vertex it is $B$, $B_0$ or $B \cap B_0$. Moreover, every such choice leads to a coloring counted in $n(U)$.
Thus,
\begin{align*}
 \frac{n(U)}{n(\emptyset)} =\left(\frac{|A|}{|A_0|}\right)^{|U^-  \cap \Even|} \left(\frac{|B|}{|B_0|}\right)^{|U^- \cap \Odd|} \left(\frac{|A \cap A_0|}{|A_0|}\right)^{|\intextB U \cap \Even|} \left(\frac{|B \cap B_0|}{|B_0|}\right)^{|\intextB U \cap \Odd|} ,
\end{align*}
where $U^- := U \setminus \intB U$.
When $q$ is even, using that $|A|=|A_0|=|B|=|B_0|=\frac q2$ and $|A \cap A_0|=|B \cap B_0|$, we get that
\[ \frac{n(U)}{n(\emptyset)} = \left(\frac{2|A \cap A_0|}{q}\right)^{|\intextB U|} ,\]
which implies the claimed inequality and the equality case.
When $q$ is odd, consider first the case that $|A_0|<|A|$ (equivalently, $|B_0|>|B|$). The ratio $\frac{n(U)}{n(\emptyset)}$ is then maximized when $A_0 \subset A$. In this case, using that $|A_0|=|A \cap A_0|=|B|=|B \cap B_0|=\frac{q-1}2$ and $|A|=|B_0|=\frac{q+1}2$, we obtain that
\[ \frac{n(U)}{n(\emptyset)} = \left(\frac{q-1}{q+1}\right)^{|U^+ \cap \Odd| - |U^- \cap \Even|} .\]
Finally, $2d|U^- \cap \Even| + |\partial U| \le 2d|U^+ \cap \Odd|$, since $|\{e \in E(\Z^d) : e \cap U \neq \emptyset\}|$ is sandwiched between the left- and right-hand sides, and equality is attained if and only if $U$ is an odd set. The case that $|B_0|<|B|$ is treated similarly. In the case that $|A_0|=|A|$, we obtain that $\frac{n(U)}{n(\emptyset)} < (\frac{q-1}{q+1})^{|\intextB U|} < (\frac{q-1}{q+1})^{\frac{1}{2d} |\partial U|}$.

This toy example shows a difference in behavior between the even and odd $q$ cases, with the odd case more difficult due to the lower entropic cost of creating interfaces between $P_0$- and $P$-ordered regions.
It is the odd $q$ case that motivates many of the definitions, including the above-mentioned fact that the ordered region corresponding to a dominant pattern should be either an even set or an odd set. Additionally, the example shows that the domain walls do not carry a high penalty per edge when $q$ is large. Indeed, the right-hand side of~\eqref{eq:single droplet in ordered pattern odd q} has the form $(1-\eps_q)^{\frac{|\partial U|}{2d}}$ with $\eps_q$ tending to zero as $q$ tends to infinity.

Let us return to the method by which we shall show that domain walls are `penalized'.
Recall that in the Ising model, we used a sign-flip transformation, given in~\eqref{eq:ising-sign-flip}, in order to deduce the bound in~\eqref{eq:ising-domain-wall-bound} on the probability that a given contour $\partial U$ is a $-+$ domain wall. Essentially, the flipping of the signs in $U$ eradicated the domain wall by transforming the order in $U$ (near the boundary) from $-$ to $+$. To obtain an analogous bound in the proper coloring model, we will use a more involved transformation, which, roughly speaking, permutes the colors in each ordered region, perhaps also shifting them by one lattice site, so as to make the pattern there agree with the boundary pattern, and erases the colors in the `bad' regions, replacing them with fresh samples of the boundary pattern. Much of the technical work is then focused on showing that this transformation indeed `repairs' the coloring, i.e., establishing suitable analogues of~\eqref{eq:ising-domain-wall-cost} and~\eqref{eq:ising-domain-wall-bound}. This is further explained in Section~\ref{sec:proof-overview-breakup}.

\smallskip
\noindent
\textbf{The number of contours:} The third step in the Peierls argument involves a sum over contours analogous to the sum performed in~\eqref{eq:Ising estimate via contours} for the Ising model. For proper colorings, however, the bound obtained for the presence of a single droplet in~\eqref{eq:single droplet in ordered pattern odd q} is insufficient for the sum to converge as, at least in high dimensions, the number of contours as estimated in~\eqref{eq:bounds on the number of contours} grows much more rapidly than the reciprocal of the bound. The intuition for the remedy comes from the fact, mentioned above, that in order for the bound~\eqref{eq:single droplet in ordered pattern odd q} to be saturated, the set $U$ needs to be even or odd. We thus proceed by considering the properties of such sets.

\begin{figure}
	\centering
	\includegraphics[scale=0.05]{oddcut-sample1.pdf}
	\caption{A large odd cutset in $\Z^2$.}
	\label{fig:odd cutset}
\end{figure}

An \emph{odd cutset} (or odd contour) is the edge boundary $\partial U$ of a domain $U\subset\Z^d$ which is either an even or an odd set; see Figure~\ref{fig:odd cutset}. Let $\bar{N}_\ell$ be the number of odd cutsets of length $|\partial U|=\ell$ with the origin in $U$. Roman Koteck\'y~\cite{kotecky2009} asked whether $\bar{N}_\ell$ is significantly smaller than $N_\ell$ in high dimensions (recall~\eqref{eq:bounds on the number of contours}); see also~\cite[Open question 10]{peled2010high}. This was addressed by Feldheim and the second author~\cite{feldheim2018growth} who showed that
\begin{equation}\label{eq:number of odd cutsets}
  2^{\left(1 + 2^{-2d}\right)\frac{\ell}{2d}}\le \bar{N}_\ell\le 2^{\big(1 + \tfrac{C\log^{3/2} d}{\sqrt{d}}\big)\frac{\ell}{2d}}
\end{equation}
for all dimensions $d\ge 2$ and all sufficiently large $\ell$ which are multiples of $2d$. The divisibility constraint is imposed as $\bar{N}_\ell = 0$ when $\ell$ is not a multiple of $2d$ (see,e.g.,~\cite[Lemma 1.3]{feldheim2018growth}). Thus, in high dimensions, the number of odd cutsets $\bar{N}_\ell$ grows roughly as $2^{\frac{\ell}{2d}}$ while the number of contours $N_\ell$ grows roughly at the much faster rate $\exp(C\frac{\ell}{d}\log d)$. However, comparing the bound~\eqref{eq:single droplet in ordered pattern odd q} to the number of odd cutsets~\eqref{eq:number of odd cutsets}, we see that even if the sum over contours in the Peierls argument is restricted to odd cutsets, the bound is again insufficient for the sum to converge, for any $q\ge 3$.

\begin{figure}[t]
  \centering
  \begin{tabular}{ccc}
    \begin{tikzpicture}

  \path[use as bounding box] (-0.250,-0.250) rectangle (7.000,7.000);

  \draw[line width = 0.2, color = black, densely dotted] (1.000,0.500) -- (1.000,0.250) -- (1.250,0.250) -- (1.250,0.500);
  \draw[line width = 0.2, color = black, densely dotted] (1.500,0.500) -- (1.500,0.250) -- (1.750,0.250) -- (1.750,0.500);
  \draw[line width = 0.2, color = black, densely dotted] (2.000,0.500) -- (2.000,0.250) -- (2.250,0.250) -- (2.250,0.500);
  \draw[line width = 0.2, color = black, densely dotted] (2.500,0.500) -- (2.500,0.250) -- (2.750,0.250) -- (2.750,0.500);
  \draw[line width = 0.2, color = black, densely dotted] (3.000,0.500) -- (3.000,0.250) -- (3.250,0.250) -- (3.250,0.500);
  \draw[line width = 0.2, color = black, densely dotted] (3.500,0.500) -- (3.500,0.250) -- (3.750,0.250) -- (3.750,0.500);
  \draw[line width = 0.2, color = black, densely dotted] (4.000,0.500) -- (4.000,0.250) -- (4.250,0.250) -- (4.250,0.500);
  \draw[line width = 0.2, color = black, densely dotted] (4.500,0.500) -- (4.500,0.250) -- (4.750,0.250) -- (4.750,0.500);
  \draw[line width = 0.2, color = black, densely dotted] (5.000,0.500) -- (5.000,0.250) -- (5.250,0.250) -- (5.250,0.500);
  \draw[line width = 0.2, color = black, densely dotted] (5.500,0.500) -- (5.500,0.250) -- (5.750,0.250) -- (5.750,0.500);
  \draw[line width = 0.2, color = black, densely dotted] (1.000,6.250) -- (1.000,6.500) -- (1.250,6.500) -- (1.250,6.250);
  \draw[line width = 0.2, color = black, densely dotted] (1.500,6.250) -- (1.500,6.500) -- (1.750,6.500) -- (1.750,6.250);
  \draw[line width = 0.2, color = black, densely dotted] (2.000,6.250) -- (2.000,6.500) -- (2.250,6.500) -- (2.250,6.250);
  \draw[line width = 0.2, color = black, densely dotted] (2.500,6.250) -- (2.500,6.500) -- (2.750,6.500) -- (2.750,6.250);
  \draw[line width = 0.2, color = black, densely dotted] (3.000,6.250) -- (3.000,6.500) -- (3.250,6.500) -- (3.250,6.250);
  \draw[line width = 0.2, color = black, densely dotted] (3.500,6.250) -- (3.500,6.500) -- (3.750,6.500) -- (3.750,6.250);
  \draw[line width = 0.2, color = black, densely dotted] (4.000,6.250) -- (4.000,6.500) -- (4.250,6.500) -- (4.250,6.250);
  \draw[line width = 0.2, color = black, densely dotted] (4.500,6.250) -- (4.500,6.500) -- (4.750,6.500) -- (4.750,6.250);
  \draw[line width = 0.2, color = black, densely dotted] (5.000,6.250) -- (5.000,6.500) -- (5.250,6.500) -- (5.250,6.250);
  \draw[line width = 0.2, color = black, densely dotted] (5.500,6.250) -- (5.500,6.500) -- (5.750,6.500) -- (5.750,6.250);
  \draw[line width = 0.2, color = black, densely dotted] (0.500,1.000) -- (0.250,1.000) -- (0.250,1.250) -- (0.500,1.250);
  \draw[line width = 0.2, color = black, densely dotted] (0.500,1.500) -- (0.250,1.500) -- (0.250,1.750) -- (0.500,1.750);
  \draw[line width = 0.2, color = black, densely dotted] (0.500,2.000) -- (0.250,2.000) -- (0.250,2.250) -- (0.500,2.250);
  \draw[line width = 0.2, color = black, densely dotted] (0.500,2.500) -- (0.250,2.500) -- (0.250,2.750) -- (0.500,2.750);
  \draw[line width = 0.2, color = black, densely dotted] (0.500,3.000) -- (0.250,3.000) -- (0.250,3.250) -- (0.500,3.250);
  \draw[line width = 0.2, color = black, densely dotted] (0.500,3.500) -- (0.250,3.500) -- (0.250,3.750) -- (0.500,3.750);
  \draw[line width = 0.2, color = black, densely dotted] (0.500,4.000) -- (0.250,4.000) -- (0.250,4.250) -- (0.500,4.250);
  \draw[line width = 0.2, color = black, densely dotted] (0.500,4.500) -- (0.250,4.500) -- (0.250,4.750) -- (0.500,4.750);
  \draw[line width = 0.2, color = black, densely dotted] (0.500,5.000) -- (0.250,5.000) -- (0.250,5.250) -- (0.500,5.250);
  \draw[line width = 0.2, color = black, densely dotted] (0.500,5.500) -- (0.250,5.500) -- (0.250,5.750) -- (0.500,5.750);
  \draw[line width = 0.2, color = black, densely dotted] (6.250,1.000) -- (6.500,1.000) -- (6.500,1.250) -- (6.250,1.250);
  \draw[line width = 0.2, color = black, densely dotted] (6.250,1.500) -- (6.500,1.500) -- (6.500,1.750) -- (6.250,1.750);
  \draw[line width = 0.2, color = black, densely dotted] (6.250,2.000) -- (6.500,2.000) -- (6.500,2.250) -- (6.250,2.250);
  \draw[line width = 0.2, color = black, densely dotted] (6.250,2.500) -- (6.500,2.500) -- (6.500,2.750) -- (6.250,2.750);
  \draw[line width = 0.2, color = black, densely dotted] (6.250,3.000) -- (6.500,3.000) -- (6.500,3.250) -- (6.250,3.250);
  \draw[line width = 0.2, color = black, densely dotted] (6.250,3.500) -- (6.500,3.500) -- (6.500,3.750) -- (6.250,3.750);
  \draw[line width = 0.2, color = black, densely dotted] (6.250,4.000) -- (6.500,4.000) -- (6.500,4.250) -- (6.250,4.250);
  \draw[line width = 0.2, color = black, densely dotted] (6.250,4.500) -- (6.500,4.500) -- (6.500,4.750) -- (6.250,4.750);
  \draw[line width = 0.2, color = black, densely dotted] (6.250,5.000) -- (6.500,5.000) -- (6.500,5.250) -- (6.250,5.250);
  \draw[line width = 0.2, color = black, densely dotted] (6.250,5.500) -- (6.500,5.500) -- (6.500,5.750) -- (6.250,5.750);

  \draw[line width = 0.01cm, color = black] (0.750,0.750) -- (0.750,0.500) -- (1.000,0.500)-- (1.000,0.750) -- (1.250,0.750) -- (1.250,0.500) -- (1.500,0.500)-- (1.500,0.750) -- (1.750,0.750) -- (1.750,0.500) -- (2.000,0.500)-- (2.000,0.750) -- (2.250,0.750) -- (2.250,0.500) -- (2.500,0.500)-- (2.500,0.750) -- (2.750,0.750) -- (2.750,0.500) -- (3.000,0.500)-- (3.000,0.750) -- (3.250,0.750) -- (3.250,0.500) -- (3.500,0.500)-- (3.500,0.750) -- (3.750,0.750) -- (3.750,0.500) -- (4.000,0.500)-- (4.000,0.750) -- (4.250,0.750) -- (4.250,0.500) -- (4.500,0.500)-- (4.500,0.750) -- (4.750,0.750) -- (4.750,0.500) -- (5.000,0.500)-- (5.000,0.750) -- (5.250,0.750) -- (5.250,0.500) -- (5.500,0.500)-- (5.500,0.750) -- (5.750,0.750) -- (5.750,0.500) -- (6.000,0.500) -- (6.000,0.750);

  \draw[line width = 0.01cm, color = black] (0.750,6.000) -- (0.750,6.250) -- (1.000,6.250)-- (1.000,6.000) -- (1.250,6.000) -- (1.250,6.250) -- (1.500,6.250)-- (1.500,6.000) -- (1.750,6.000) -- (1.750,6.250) -- (2.000,6.250)-- (2.000,6.000) -- (2.250,6.000) -- (2.250,6.250) -- (2.500,6.250)-- (2.500,6.000) -- (2.750,6.000) -- (2.750,6.250) -- (3.000,6.250)-- (3.000,6.000) -- (3.250,6.000) -- (3.250,6.250) -- (3.500,6.250)-- (3.500,6.000) -- (3.750,6.000) -- (3.750,6.250) -- (4.000,6.250)-- (4.000,6.000) -- (4.250,6.000) -- (4.250,6.250) -- (4.500,6.250)-- (4.500,6.000) -- (4.750,6.000) -- (4.750,6.250) -- (5.000,6.250)-- (5.000,6.000) -- (5.250,6.000) -- (5.250,6.250) -- (5.500,6.250)-- (5.500,6.000) -- (5.750,6.000) -- (5.750,6.250) -- (6.000,6.250) -- (6.000,6.000);

  \draw[line width = 0.01cm, color = black] (0.750,0.750) -- (0.500,0.750) -- (0.500,1.000)-- (0.750,1.000) -- (0.750,1.250) -- (0.500,1.250) -- (0.500,1.500)-- (0.750,1.500) -- (0.750,1.750) -- (0.500,1.750) -- (0.500,2.000)-- (0.750,2.000) -- (0.750,2.250) -- (0.500,2.250) -- (0.500,2.500)-- (0.750,2.500) -- (0.750,2.750) -- (0.500,2.750) -- (0.500,3.000)-- (0.750,3.000) -- (0.750,3.250) -- (0.500,3.250) -- (0.500,3.500)-- (0.750,3.500) -- (0.750,3.750) -- (0.500,3.750) -- (0.500,4.000)-- (0.750,4.000) -- (0.750,4.250) -- (0.500,4.250) -- (0.500,4.500)-- (0.750,4.500) -- (0.750,4.750) -- (0.500,4.750) -- (0.500,5.000)-- (0.750,5.000) -- (0.750,5.250) -- (0.500,5.250) -- (0.500,5.500)-- (0.750,5.500) -- (0.750,5.750) -- (0.500,5.750) -- (0.500,6.000) -- (0.750,6.000);

  \draw[line width = 0.01cm, color = black] (6.000,0.750) -- (6.250,0.750) -- (6.250,1.000)-- (6.000,1.000) -- (6.000,1.250) -- (6.250,1.250) -- (6.250,1.500)-- (6.000,1.500) -- (6.000,1.750) -- (6.250,1.750) -- (6.250,2.000)-- (6.000,2.000) -- (6.000,2.250) -- (6.250,2.250) -- (6.250,2.500)-- (6.000,2.500) -- (6.000,2.750) -- (6.250,2.750) -- (6.250,3.000)-- (6.000,3.000) -- (6.000,3.250) -- (6.250,3.250) -- (6.250,3.500)-- (6.000,3.500) -- (6.000,3.750) -- (6.250,3.750) -- (6.250,4.000)-- (6.000,4.000) -- (6.000,4.250) -- (6.250,4.250) -- (6.250,4.500)-- (6.000,4.500) -- (6.000,4.750) -- (6.250,4.750) -- (6.250,5.000)-- (6.000,5.000) -- (6.000,5.250) -- (6.250,5.250) -- (6.250,5.500)-- (6.000,5.500) -- (6.000,5.750) -- (6.250,5.750) -- (6.250,6.000) -- (6.000,6.000);

\end{tikzpicture} & \hspace{0.2in} & \begin{tikzpicture}

  \path[use as bounding box] (-0.250,-0.250) rectangle (7.000,7.000);

  \draw[line width = 0.01cm, color = black] (0.750,0.750) -- (0.750,0.500) -- (1.000,0.500)-- (1.000,0.250) -- (1.250,0.250) -- (1.250,0.500) -- (1.500,0.500)-- (1.500,0.750) -- (1.750,0.750) -- (1.750,0.500) -- (2.000,0.500)-- (2.000,0.750) -- (2.250,0.750) -- (2.250,0.500) -- (2.500,0.500)-- (2.500,0.250) -- (2.750,0.250) -- (2.750,0.500) -- (3.000,0.500)-- (3.000,0.750) -- (3.250,0.750) -- (3.250,0.500) -- (3.500,0.500)-- (3.500,0.250) -- (3.750,0.250) -- (3.750,0.500) -- (4.000,0.500)-- (4.000,0.750) -- (4.250,0.750) -- (4.250,0.500) -- (4.500,0.500)-- (4.500,0.250) -- (4.750,0.250) -- (4.750,0.500) -- (5.000,0.500)-- (5.000,0.250) -- (5.250,0.250) -- (5.250,0.500) -- (5.500,0.500)-- (5.500,0.250) -- (5.750,0.250) -- (5.750,0.500) -- (6.000,0.500) -- (6.000,0.750);
  \draw[line width = 0.2, color = black, densely dotted] (4.750,0.250) -- (4.750,0.000) -- (5.000,0.000) -- (5.000,0.250);
  \draw[line width = 0.2, color = black, densely dotted] (5.250,0.250) -- (5.250,0.000) -- (5.500,0.000) -- (5.500,0.250);
  \draw[line width = 0.01cm, color = black] (0.750,6.000) -- (0.750,6.250) -- (1.000,6.250)-- (1.000,6.000) -- (1.250,6.000) -- (1.250,6.250) -- (1.500,6.250)-- (1.500,6.000) -- (1.750,6.000) -- (1.750,6.250) -- (2.000,6.250)-- (2.000,6.000) -- (2.250,6.000) -- (2.250,6.250) -- (2.500,6.250)-- (2.500,6.500) -- (2.750,6.500) -- (2.750,6.250) -- (3.000,6.250)-- (3.000,6.000) -- (3.250,6.000) -- (3.250,6.250) -- (3.500,6.250)-- (3.500,6.500) -- (3.750,6.500) -- (3.750,6.250) -- (4.000,6.250)-- (4.000,6.000) -- (4.250,6.000) -- (4.250,6.250) -- (4.500,6.250)-- (4.500,6.000) -- (4.750,6.000) -- (4.750,6.250) -- (5.000,6.250)-- (5.000,6.500) -- (5.250,6.500) -- (5.250,6.250) -- (5.500,6.250)-- (5.500,6.500) -- (5.750,6.500) -- (5.750,6.250) -- (6.000,6.250) -- (6.000,6.000);
  \draw[line width = 0.2, color = black, densely dotted] (5.250,6.500) -- (5.250,6.750) -- (5.500,6.750) -- (5.500,6.500);
  \draw[line width = 0.01cm, color = black] (0.750,0.750) -- (0.500,0.750) -- (0.500,1.000)-- (0.250,1.000) -- (0.250,1.250) -- (0.500,1.250) -- (0.500,1.500)-- (0.750,1.500) -- (0.750,1.750) -- (0.500,1.750) -- (0.500,2.000)-- (0.750,2.000) -- (0.750,2.250) -- (0.500,2.250) -- (0.500,2.500)-- (0.250,2.500) -- (0.250,2.750) -- (0.500,2.750) -- (0.500,3.000)-- (0.750,3.000) -- (0.750,3.250) -- (0.500,3.250) -- (0.500,3.500)-- (0.250,3.500) -- (0.250,3.750) -- (0.500,3.750) -- (0.500,4.000)-- (0.750,4.000) -- (0.750,4.250) -- (0.500,4.250) -- (0.500,4.500)-- (0.250,4.500) -- (0.250,4.750) -- (0.500,4.750) -- (0.500,5.000)-- (0.250,5.000) -- (0.250,5.250) -- (0.500,5.250) -- (0.500,5.500)-- (0.750,5.500) -- (0.750,5.750) -- (0.500,5.750) -- (0.500,6.000) -- (0.750,6.000);
  \draw[line width = 0.2, color = black, densely dotted] (0.250,4.750) -- (0.000,4.750) -- (0.000,5.000) -- (0.250,5.000);
  \draw[line width = 0.01cm, color = black] (6.000,0.750) -- (6.250,0.750) -- (6.250,1.000)-- (6.500,1.000) -- (6.500,1.250) -- (6.250,1.250) -- (6.250,1.500)-- (6.500,1.500) -- (6.500,1.750) -- (6.250,1.750) -- (6.250,2.000)-- (6.500,2.000) -- (6.500,2.250) -- (6.250,2.250) -- (6.250,2.500)-- (6.000,2.500) -- (6.000,2.750) -- (6.250,2.750) -- (6.250,3.000)-- (6.500,3.000) -- (6.500,3.250) -- (6.250,3.250) -- (6.250,3.500)-- (6.000,3.500) -- (6.000,3.750) -- (6.250,3.750) -- (6.250,4.000)-- (6.500,4.000) -- (6.500,4.250) -- (6.250,4.250) -- (6.250,4.500)-- (6.500,4.500) -- (6.500,4.750) -- (6.250,4.750) -- (6.250,5.000)-- (6.000,5.000) -- (6.000,5.250) -- (6.250,5.250) -- (6.250,5.500)-- (6.000,5.500) -- (6.000,5.750) -- (6.250,5.750) -- (6.250,6.000) -- (6.000,6.000);
  \draw[line width = 0.2, color = black, densely dotted] (6.500,1.250) -- (6.750,1.250) -- (6.750,1.500) -- (6.500,1.500);
  \draw[line width = 0.2, color = black, densely dotted] (6.500,1.750) -- (6.750,1.750) -- (6.750,2.000) -- (6.500,2.000);
  \draw[line width = 0.2, color = black, densely dotted] (6.500,4.250) -- (6.750,4.250) -- (6.750,4.500) -- (6.500,4.500);

\end{tikzpicture}
  \end{tabular}
  \caption{Odd cutsets approximating the boundary of a cube. On the left, every second boundary vertex can be ``pushed out'' independently of other vertices, yielding, in $d$ dimensions, $2^{\left(\frac{1}{2d}-o_d(1)\right)\ell}$ odd cutsets with $\ell$ edges, as $\ell\to\infty$. On the right, for every odd cutset obtained in such way, we have
the additional option of independently ``pushing out'' vertices all
of whose $2d-2$ neighbors were ``pushed out'' in the first stage.
Combining these two stages leads to the lower bound in~\eqref{eq:number of odd cutsets}. Figure taken from~\cite{peled2014odd}.}
  \label{fig:cubes}
\end{figure}

The fact that the number of odd cutsets of given length grows significantly more slowly than the number of contours of the same length is indicative of a deeper structural difference. Typical odd cutsets have been shown to have a \emph{macroscopic shape} or \emph{approximation} (e.g., the boundary of an axis-parallel box; see Figure~\ref{fig:cubes}) from which they deviate on the microscopic scale, while general contours should scale to integrated super-Brownian excursion~\cite{lubensky1979statistics, slade1999lattice}. The distinction between these very different behaviors is akin to the breathing transition undergone by random surfaces~\cite[Section~7.3]{fernandez2013random}. This phenomenon was first used by Sapozhenko in studying enumeration problems on bipartite graphs and posets~\cite{sapozhenko1987onthen,sapozhenko1989number,sapozhenko1991number} and has been exploited in several works~\cite{Galvin2003hammingcube,galvin2004slow,galvin2004phase,galvin2007sampling,galvin2007torpid,galvin2008sampling,peled2010high,galvin2012phase,peled2014odd, feldheim2015long, peled2018rigidity, peledspinka2018spin} to provide a natural \emph{coarse-graining} scheme for odd cutsets, grouping them according to their approximation, and noting that the number of such approximations is significantly smaller in high dimensions (of order at most $\exp\big(\big(\tfrac{\log d}{d}\big)^{3/2} \ell\big)$) than the number of odd cutsets themselves.

The version of the Peierls argument used in the proof of Theorem~\ref{thm:long-range-order} also makes use of the above-mentioned coarse graining scheme. To allow this, ordered regions of the coloring are defined in such a way that they are always even or odd sets. Then, the third step of the Peierls argument is performed by summing over approximations to the odd cutsets instead of on the cutsets themselves. This twist complicates also the second step of the Peierls argument, as it necessitates that the bound obtained on the probability that a given contour is a domain wall (or a disordered region, or a region of overlap), will be extended to the case when only an approximation to the contour is prescribed. Approximations to odd cutsets are further discussed in Section~\ref{sec:approx}.

\subsection{Ordered and disordered regions}
\label{sec:ordered-regions}

Given a proper $q$-coloring $f$ of $\Z^d$, we wish first to identify regions where $f$ follows, in a suitable sense, a dominant pattern. A first idea is that the decision regarding a vertex $v$ will be made based on the values that $f$ takes on the \emph{neighbors} of $v$. Indeed, the color that $v$ takes cannot itself be sufficient as it has only $q$ options whereas there are many more dominant patterns, but the colors of the neighbors turn out to suit the job. A second idea, motivated by the toy scenario described earlier and also by questions of approximation of contours which will be soon described, is that each region will be a (regular) even or odd set. More precisely, the region associated with a dominant pattern $(A,B)$ is an even set if $|A|\le|B|$ and an odd set if $|A|>|B|$ (thus odd sets appear only if $q$ is odd). Let us now describe the regions precisely. Let $\phasedom$ be the set of all dominant patterns. For each $P=(A,B)\in\phasedom$, define the terms
\begin{equation}\label{eq:P-even-odd}
  \text{$P$-even} = \begin{cases}
    \text{even} & |A|\le |B|\\
    \text{odd} & |A|>|B|
  \end{cases}\qquad\text{and similarly}\qquad \text{$P$-odd} = \begin{cases}
    \text{odd} & |A|\le |B|\\
    \text{even} & |A|>|B|
  \end{cases}.
\end{equation}
Thus, for instance, if $|A|\le |B|$ then even vertices (having even sum of coordinates) are $P$-even and odd vertices are $P$-odd. The region associated to $P$ is denoted $Z_P(f)$ and defined by
\begin{equation}\label{eq:Z-def}
Z_P = Z_P(f) := \big\{ v \in \Z^d : v\text{ is $P$-odd},~ N(v)\text{ is in the $P$-pattern} \big\}^+.
\end{equation}
Figure~\ref{fig:breakup} depicts these sets in examples.
For technical reasons, only $P$-odd vertices whose neighbors are in the $P$-pattern are included in $Z_P$, and then $Z_P$ is taken to be the smallest $P$-even set containing them. Note that a $P$-odd vertex in $Z_P$ is not itself required to be in the $P$-pattern, whereas a $P$-even vertex in $Z_P$ is necessarily in the $P$-pattern, but need not have its neighbors in the $P$-pattern. In addition, there may be $P$-even vertices which are not in $Z_P$ although their neighbors are in the $P$-pattern. These somewhat undesirable consequences of our definition are allowed in order to ensure that $Z_P$ is a regular $P$-even set, which will be important in the proof.

\begin{figure}
	\centering
	\includegraphics[scale=0.355]{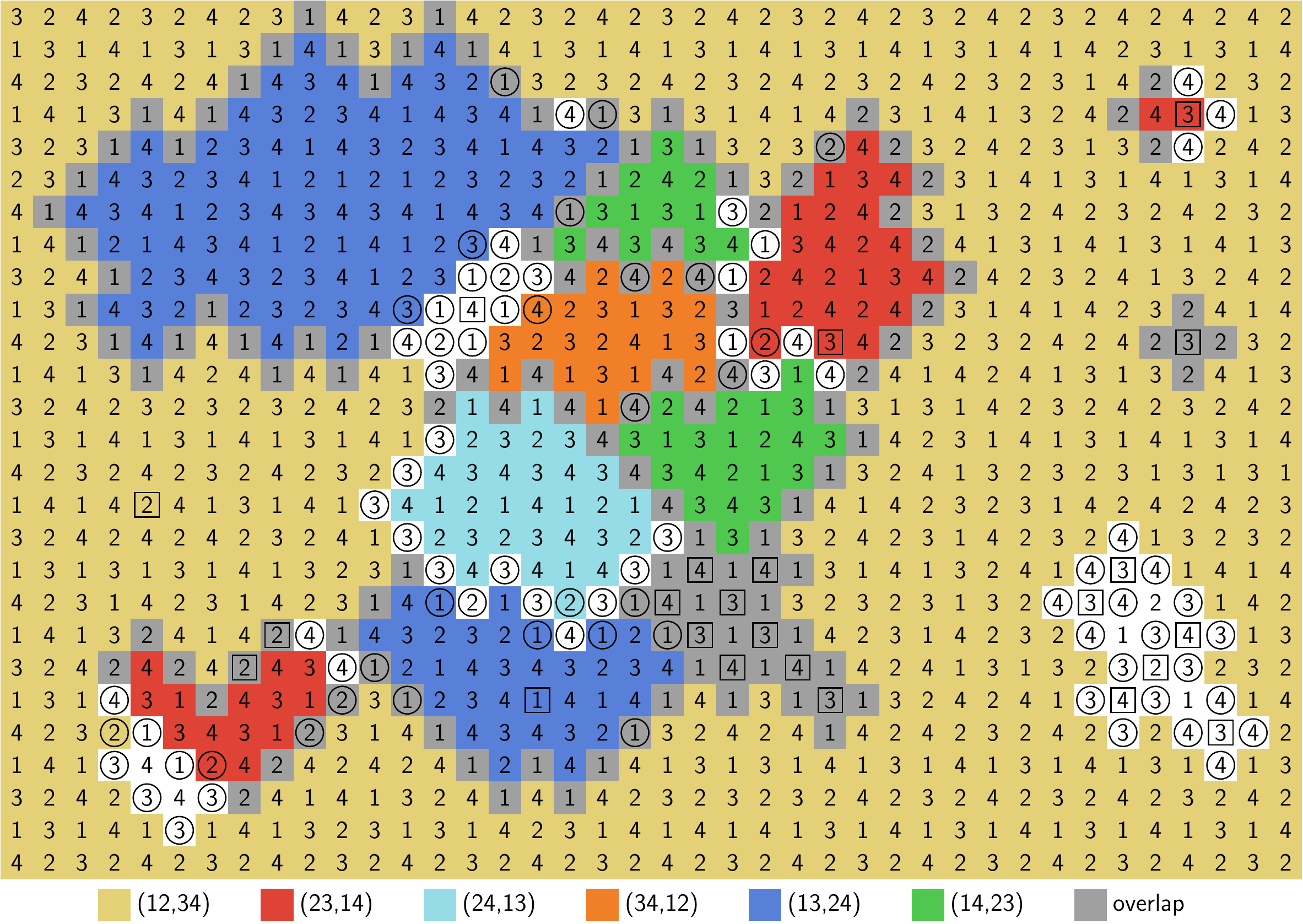}\vspace{4pt}
	
	\includegraphics[scale=0.355]{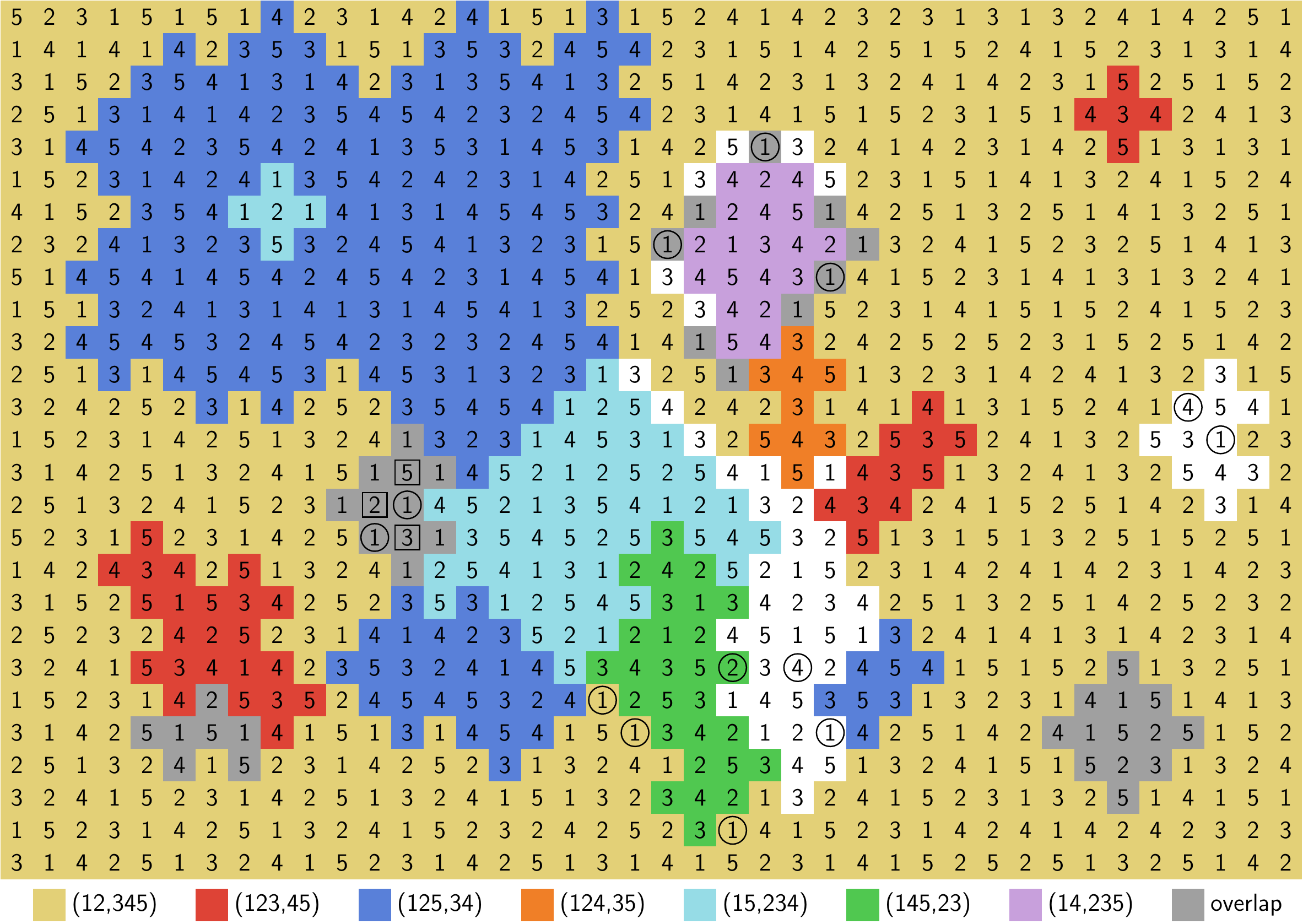}
	\captionsetup{width=0.95\textwidth, font=small}
	\caption{Proper $q$-colorings (top: $q=4$, bottom: $q=5$) and the associated identification of ordered and disordered regions (each $Z_P$ has a different color, with gray indicating $Z_\overlap$ and white indicating $Z_\bad$). Non-dominant vertices (defined in Section~\ref{sec:upper_bounds_closer_look}) are depicted: squares indicate vertices whose neighbors are assigned less than $\lfloor \frac q2 \rfloor$ different values, whereas circles indicate vertices whose neighbors are assigned more than $\lceil \frac q2 \rceil$ different values.}
	\label{fig:breakup}
\end{figure}

Having defined the regions $(Z_P)$, let us examine more closely their inter-relations. It is possible for a vertex $v$ to belong to two (or more) of the $Z_P$ and also possible that it lies outside all of the $Z_P$. These possibilities are captured by the following definitions:
\[ Z_\overlap := \bigcup_{P \neq Q} (Z_P \cap Z_Q) \qquad\text{and}\qquad Z_\bad := \bigcap_P (Z_P)^c \]
(see Figure~\ref{fig:breakup}). Regions of this type, along with the boundaries of $Z_P$, are regions where the coloring $f$ does not achieve its maximal entropy per vertex, in a way which is quantified later. It will be our task to prove that such regions are not numerous and this will lead to a proof of Theorem~\ref{thm:long-range-order}. To this end, we define
\begin{equation}\label{eq:Z_*-def}
Z_* := \bigcup_P \intextB Z_P \cup Z_\overlap \cup Z_\bad .
\end{equation}

The region $Z_*$ plays a similar role in our analysis as the contours used in arguments of the Peierls or Pirogov-Sinai type.
Recall that in the Ising model, a configuration may have many domain walls and that long-range order (Theorem~\ref{thm:low temperature order Ising}) was shown there by focusing on a single contour surrounding a given vertex (see~\eqref{eq:Ising probability of excitation}). Here too, we would like to isolate a single ``contour'' from within $Z_*$ which ``surrounds'' a given vertex $v$. We call this a breakup seen from $v$, which we explain further in the following section.

\subsection{The unlikeliness of breakups}
\label{sec:proof-overview-breakup}
With Theorem~\ref{thm:long-range-order} in mind, let $f$ be sampled from $\Pr_{\Lambda,P_0}$ and fix a vertex $v\in\Lambda$. It is convenient to extend $f$ to a coloring of $\Z^d$ by coloring vertices of $\Lambda^c$ independently and uniformly from $A_0$ or $B_0$ according to their parity (so that they are in the $P_0$-pattern). The collection $(Z_P)_P$ then identifies ordered and disordered regions in $f$. Our goal is to show that $v$ is typically in the $P_0$-pattern.
One checks that $Z_{P} \setminus Z_{\overlap}$ is in the $P$-pattern, and therefore it suffices to show that, with high probability, $Z_{P_0}$ is the unique set among $(Z_P)_P$ to which $v$ belongs. This, in turn, follows by showing that there is a path from $v$ to infinity avoiding $Z_*$. If no such path exists, there needs to be a \emph{connected} component of $Z_*^+$ which disconnects~$v$ from infinity. Our focus is then on these connected components and this motivates the following notion of a breakup seen from $v$, which encodes the partial information from $(Z_P)_P$ relevant to these components.

\begin{figure}
	\centering
	\includegraphics[scale=0.204]{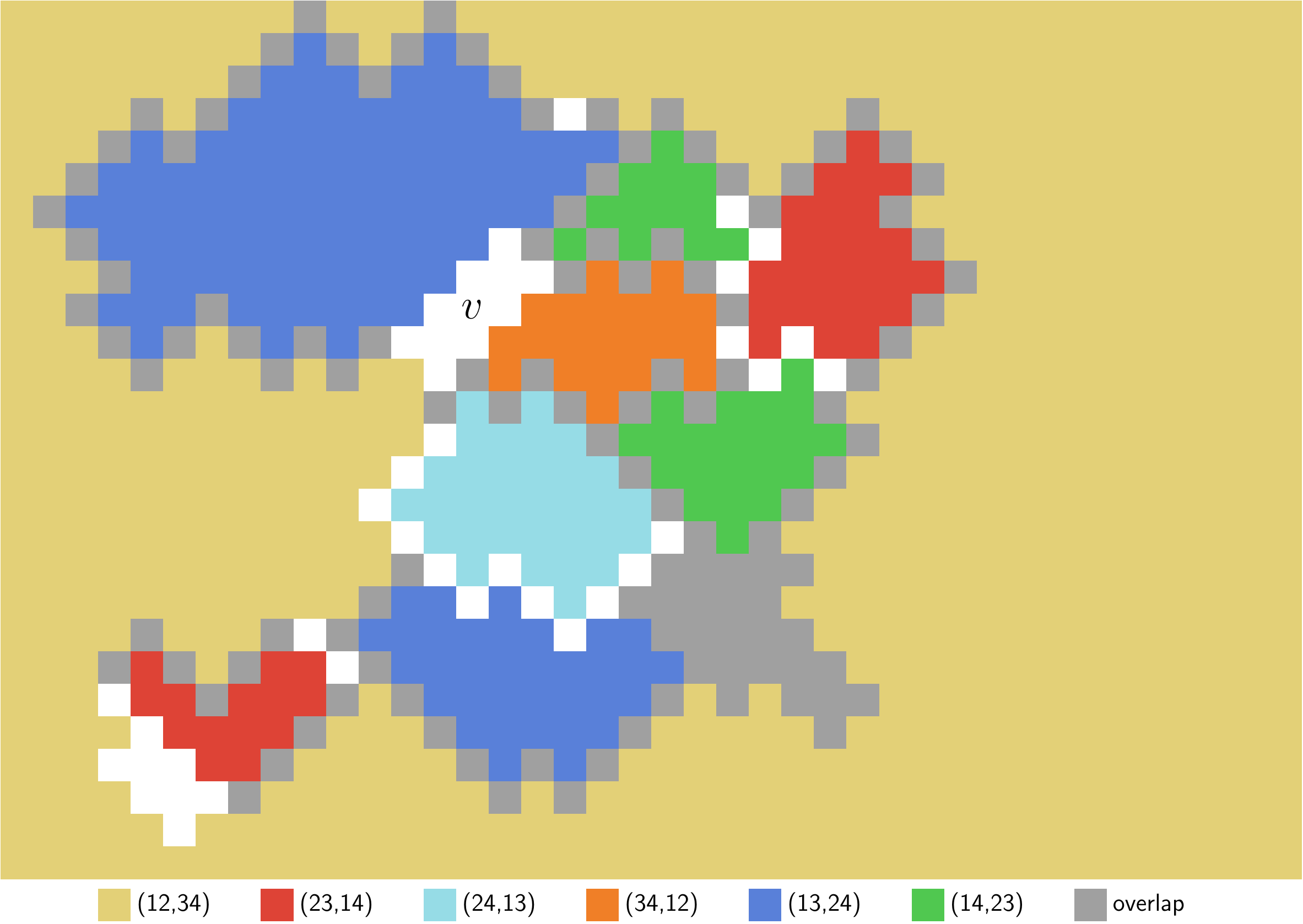}~
	\includegraphics[scale=0.204]{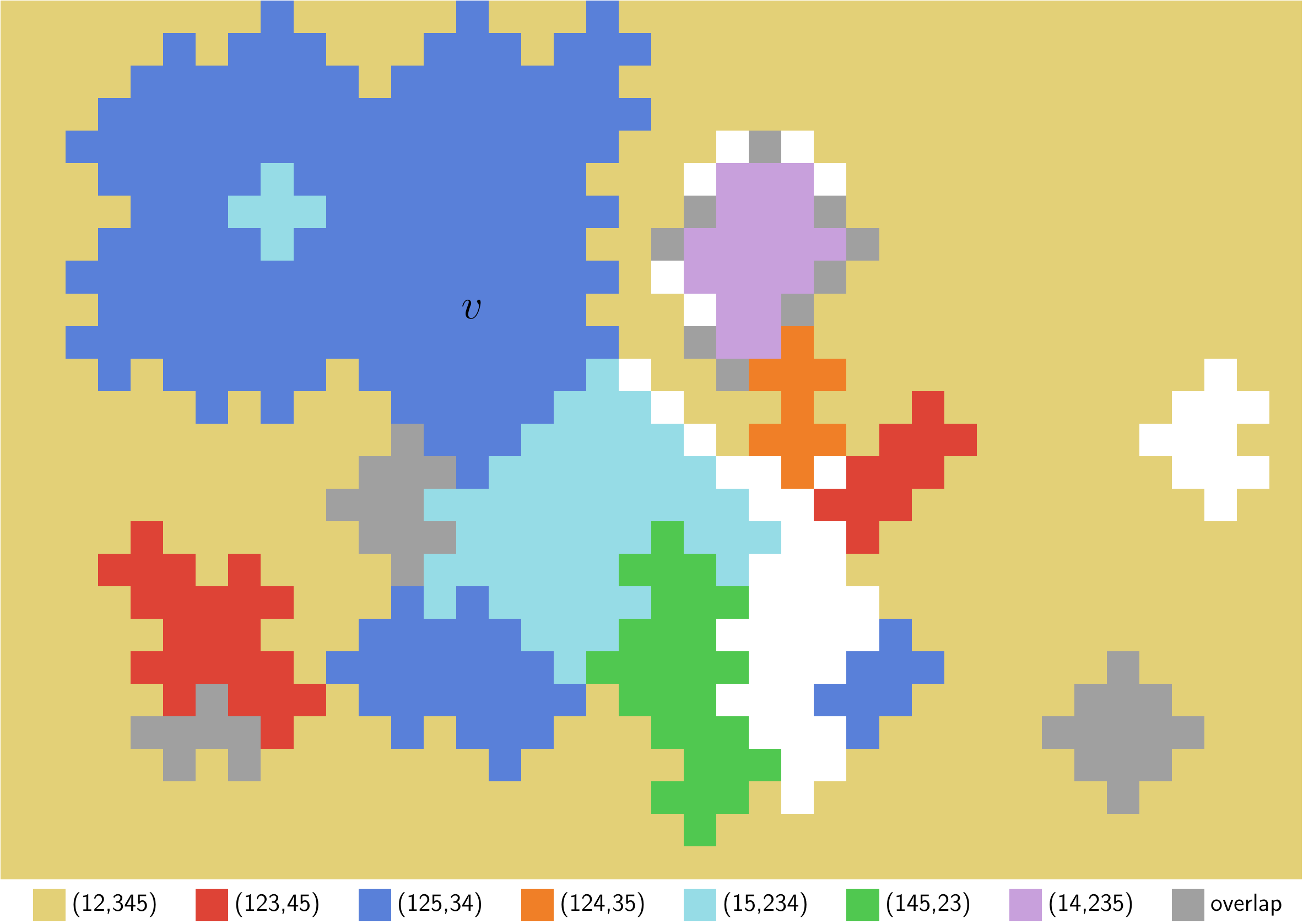}
	\captionsetup{width=0.95\textwidth, font=small}
	\caption{Breakups seen from $v$ of the colorings shown in Figure~\ref{fig:breakup} (left: $q=4$, right: $q=5$). Such breakups are not unique; for instance, the $q=5$ case may further include the small violation of the boundary pattern in the top-right corner of Figure~\ref{fig:breakup} (bottom).
Each $X_P$ has a different color, with gray indicating $X_\overlap$ and white indicating $X_\bad$ (the information of a breakup also includes the classification of the region $X_\overlap$ into various $X_P$, though this is not depicted in the figure).}
	\label{fig:breakup2}
\end{figure}

A \emph{breakup} is a collection $(X_P)_{P\in\phasedom}$ of subsets of $\Z^d$, from which one defines $X_\overlap,X_\bad,X_*$ in the same manner as $Z_\overlap,Z_\bad,Z_*$ is defined from $(Z_P)_P$, with the property that the $(X_P)_P$ coincide with the $(Z_P)_P$ in the neighborhood of $X_*$ in the sense that $X_P \cap X_*^{+5} = Z_P \cap X_*^{+5}$ for each $P$. The definition implies that each $X_P$ is a regular $P$-even set, a property important for the approximations described in the section below. A breakup is \emph{seen from $v$} if $X_*^{+5}$ is composed of a connected component of $Z_*^{+5}$ which disconnects $v$ from infinity. The choice to consider connected components of $Z_*^{+5}$ rather than just connected components of $Z_*^+$ is related to the fact that this implies that near $X_*$ (in its $5$-neighborhood) there are no additional violations of the pure dominant pattern coloring. This will be convenient in the proof (though the specific number $5$ is not important and could just as well be taken larger).
 Figure~\ref{fig:breakup2} shows possible breakups seen from $v$.
 
A ``breakup seen from $v$'' is the analogue of a ``domain wall surrounding $v$'' in the Ising model. As explained above, to obtain Theorem~\ref{thm:long-range-order}, it suffices to bound the probability that there exists a breakup seen from $v$ (namely, to bound it by the right-hand-side of~\eqref{eq:main_thm_bound}). The first goal toward this is to show that any given collection $(X_P)_P$ is unlikely to be a breakup.

We proceed to explain, for a given collection $X=(X_P)_P$, how to bound the probability that $X$ is a breakup.
To state a precise bound, we must first explain how to measure the ``size'' of a breakup. While in the Ising model, the size of a contour was given by a single number, namely its length, here the size of a breakup is described by three numbers, one for each ingredient comprising $X_*$ (recall~\eqref{eq:Z_*-def}). Specifically, denote
\begin{equation}\label{eq:breakup-size}
L:=\Big|\bigcup_P \partial X_P\Big|,\qquad M:=|X_\overlap|, \qquad N:=|X_\bad| .
\end{equation}
We emphasize the $L$ is defined in terms of the size of the edge-boundaries of $X_P$, not their vertex boundaries.
The goal is then to prove the following quantitative bound (which is the analogue of~\eqref{eq:ising-domain-wall-bound} in the Ising model):
\begin{equation}\label{eq:breakup-prob-bound}
\Pr_{\Lambda,P_0}(X\text{ is a breakup}) \le \exp\left(- \tfrac cq \big( \tfrac Ld+\tfrac Mq+\tfrac{N}{q^2} \big)\right) ,
\end{equation}
where $c>0$ is a universal constant. The reader may wish to compare this bound to the bound~\eqref{eq:single droplet in ordered pattern odd q} obtained in the toy scenario.
In the full proof, the arguments need to be adapted to the case that only an approximation of $X$ is given rather than $X$ itself, but this adaptation is not the essence of the argument so our focus in the overview is on the case that $X$ is given.

\subsubsection{The repair transformation}\label{sec:repair transformation}
Let $\Omega_X$ be the set of proper colorings for which $X$ is a breakup.
To establish the desired bound on $\Pr_{\Lambda,P_0}(\Omega_X)$, we apply the following one-to-many operation to every coloring $f \in \Omega_X$:
\begin{enumerate}[(i)]
 \itemsep0em
 \item Erase the colors at all vertices of $X_*$.
 \item For each connected component $D$ of $X_P\setminus X_*$, apply a permutation $\varphi$ taking $P$ to $P_0$ to the colors of $f$ on $D$, and also, if $P=(A,B)$ is such that $|A|>|B|$, then shift the configuration in $D$ by a single lattice site in the $(1,0,\ldots,0)$ direction~(such a shift was first used by Dobrushin for the hard-core model~\cite{dobrushin1968problem}).
 \item Fill colors following the $P_0$-pattern in all remaining vertices.
\end{enumerate}
See Figure~\ref{fig:repairmap} for an illustration.

Noting that the resulting configuration is always a proper coloring, and that no entropy is lost in step (ii), it remains to show that the entropy gain in step (iii) is much larger than the entropy loss in step (i). The gain in step (iii) is either $\log \lfloor \tfrac{q}{2} \rfloor$ or $\log \lceil \tfrac{q}{2} \rceil$ per vertex according to its parity, making the entropy gain an easily computable quantity. The main challenge is thus to bound the loss in step (i), and the method used for its resolution is described in the next three sections.

\begin{figure}
	\centering
	\captionsetup{font=small}
	\begin{subfigure}[t]{0.5\textwidth}
		\includegraphics[scale=0.204]{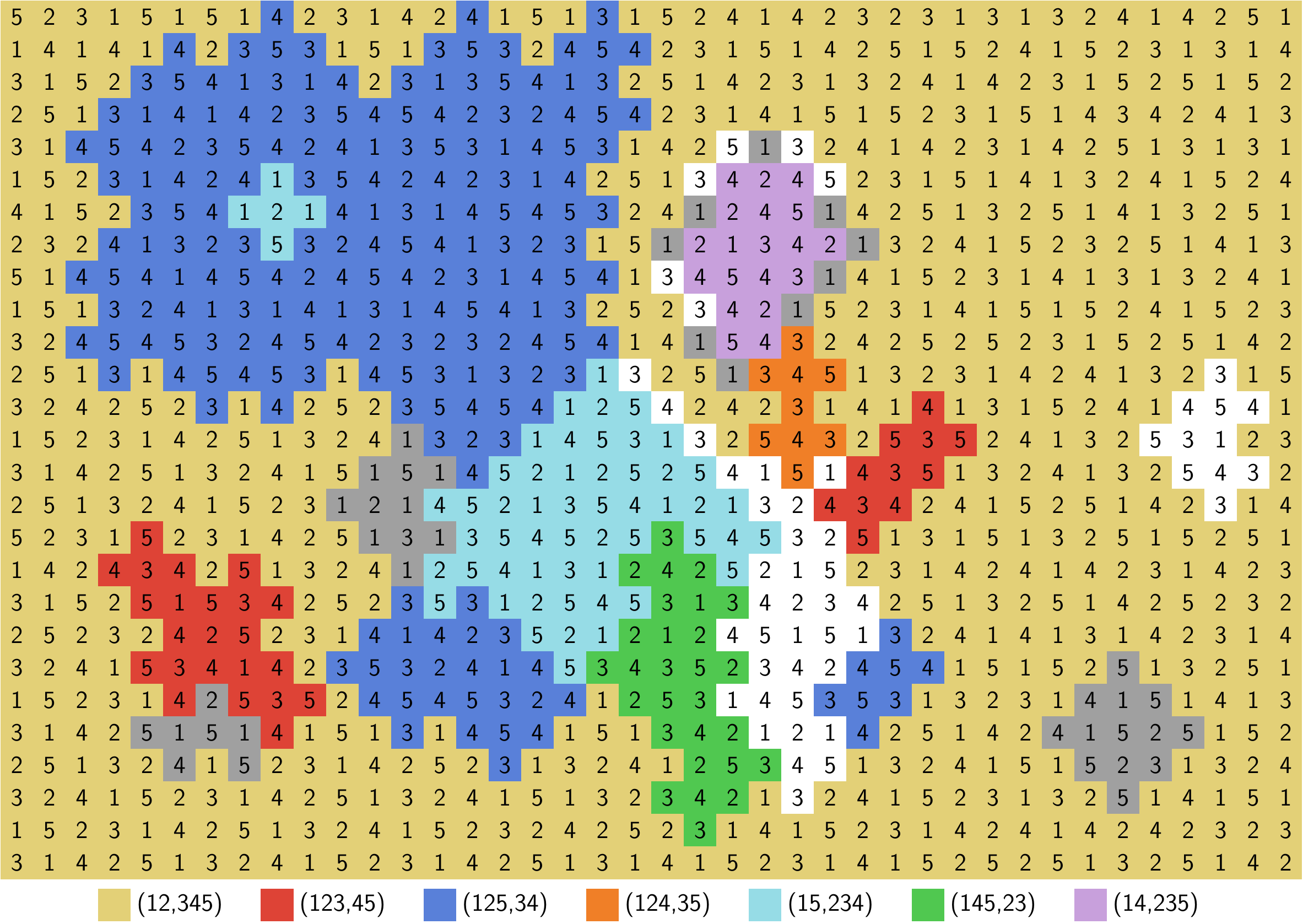}
		\caption{A coloring having a breakup $X$.}
	\end{subfigure}\,\,%
	\begin{subfigure}[t]{0.5\textwidth}
		\includegraphics[scale=0.204]{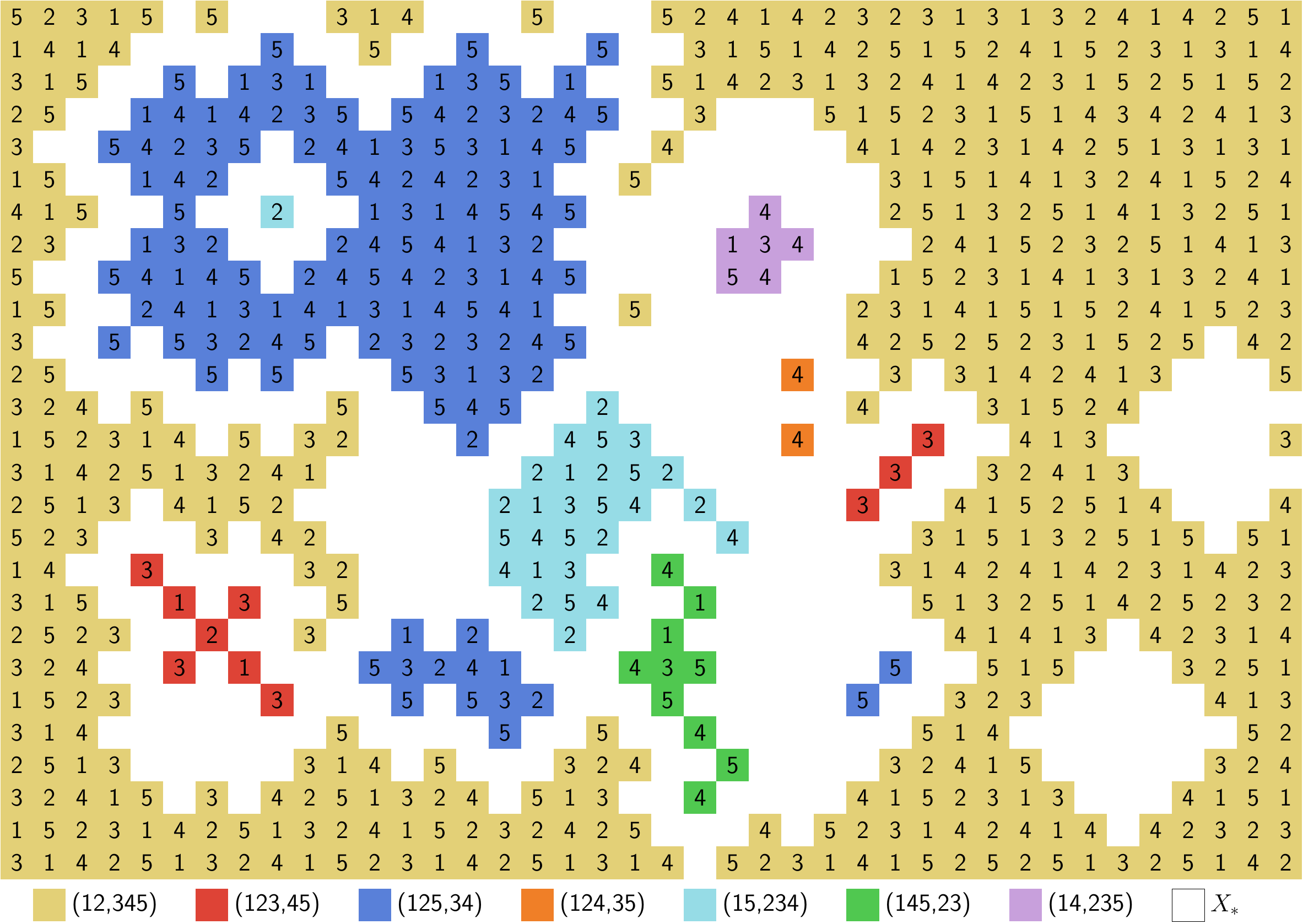}
		\caption{Step (i): colors in $X_*$ are erased.}
	\end{subfigure}
	\vspace{5pt}
	
	\begin{subfigure}[t]{0.5\textwidth}
		\includegraphics[scale=0.204]{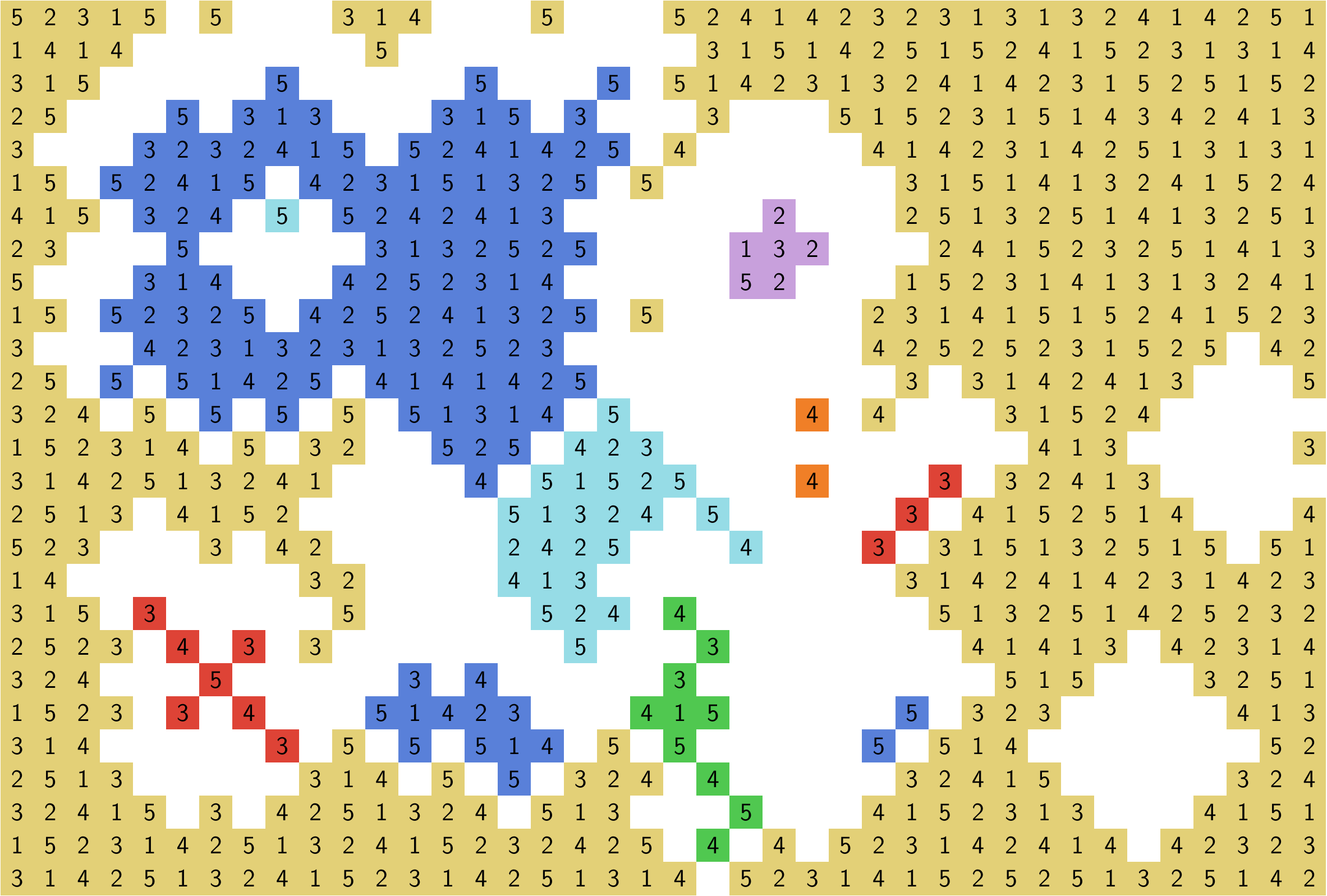}
		\caption{Step (ii): colors in $X_P$ are permuted and shifted.}
	\end{subfigure}\,\,%
	\begin{subfigure}[t]{0.51\textwidth}
		\includegraphics[scale=0.204]{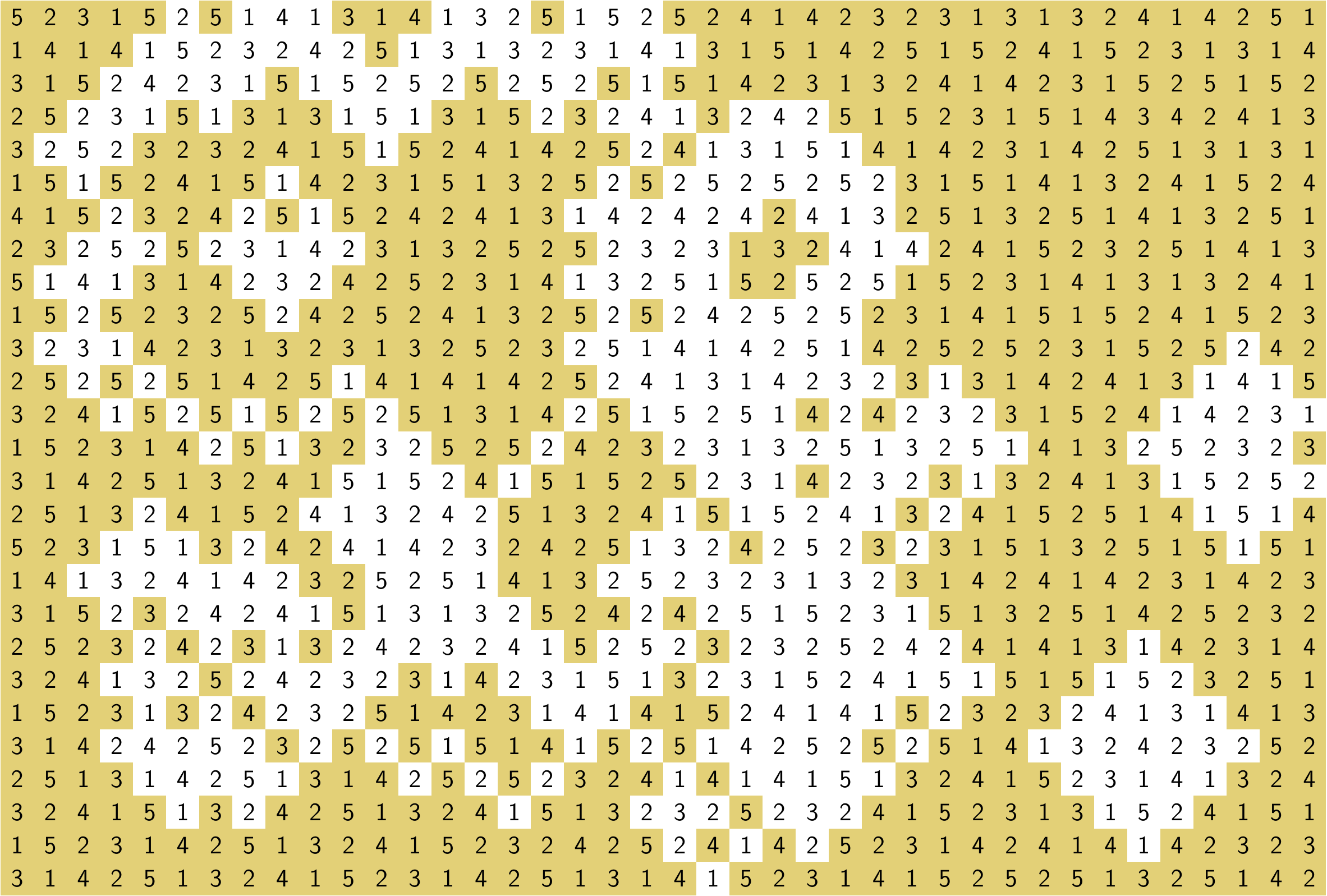}
		\caption{Step (iii): empty sites are colored in the $P_0$-pattern.}
	\end{subfigure}
	\captionsetup{font=normal}
	\caption{The repair transformation applied to the $5$-coloring of Figure~\ref{fig:breakup} with the breakup of Figure~\ref{fig:breakup2}.}
	\label{fig:repairmap}
\end{figure}

\subsubsection{Entropy methods}\label{sec:entropy}

Galvin--Tetali~\cite{galvin2004weighted}, following Kahn--Lawrentz~\cite{kahn1999generalized} and Kahn~\cite{kahn2001entropy}, found a simple and powerful bound on the number of proper $q$-colorings, and more generally graph homomorphisms, on regular bipartite graphs. The bound uses entropy methods, or more specifically, Shearer's inequality~\cite{chung1986some}. We briefly recall the definition of Shannon entropy and some of its basic properties (see, e.g., \cite{mceliece2002theory} for a more thorough discussion).

Let $Z$ be a discrete random variable.
The \emph{Shannon entropy} of $Z$ is
\[ \Ent(Z) := -\sum_z \Pr(Z=z) \log \Pr(Z=z) ,\]
where we use the convention that such sums are always over the support of the random variable in question. Given an event $A$, the conditional entropy $\Ent(Z \mid A)$ is simply the entropy of the random variable $Z$ obtained by conditioning on $A$.
Given another discrete random variable $Y$, the conditional entropy of $Z$ given $Y$ is then defined as $\Ent(Z \mid Y) := \E[\Ent(Z \mid Y=y)]$.
This gives rise to the following chain rule:
\begin{equation}\label{eq:entropy-chain-rule}
\Ent(Y,Z) = \Ent(Y) + \Ent(Z \mid Y) ,
\end{equation}
where $\Ent(Y,Z)$ is shorthand for the entropy of $(Y,Z)$.
A simple application of Jensen's inequality shows that $\Ent(Z) \le \log |\supp Z|$, where $\supp Z$ is the support of $Z$, with equality if and only if $Z$ is a uniform random variable. Another application of Jensen's inequality gives the useful property that $\Ent(Z \mid Y) \le \Ent(Z \mid \phi(Y))$ for any function $\phi$.
This, together with the chain rule, implies that entropy is subadditive. That is, if $Z_1,\dots,Z_n$ are discrete random variables, then
\begin{equation}\label{eq:entropy-subadditivity}
\Ent(Z_1,\dots,Z_n) \le \Ent(Z_1) + \cdots + \Ent(Z_n) .
\end{equation}
We now state Shearer's inequality, which is a powerful extension of this subadditivity.

\begin{lemma}[Shearer's inequality~\cite{chung1986some}]\label{lem:shearer}
Let $Z_1,\dots,Z_n$ be discrete random variables. Let $\cI$ be a collection of subsets of $\{1,\dots,n\}$ such that $|\{I \in \cI : i \in I\}| \ge k$ for some $k \ge 1$ and every~$1 \le i \le n$.
Then
\[ \Ent(Z_1,\dots,Z_n) \le \frac{1}{k} \sum_{I \in \cI} \Ent((Z_i)_{i\in I}) .\]
\end{lemma}

We now explain the bound of Galvin--Tetali for proper colorings. Let $G=((V_1, V_2), E)$ be a finite bipartite regular graph of degree $\Delta$ at every vertex. Let $f$ be a uniformly sampled proper $q$-coloring of $G$, regarded as a collection of random variables indexed by the vertices of $G$. Then
\begin{equation}
  \Ent(f) = \Ent(f|_{V_1}) + \Ent(f|_{V_2} \mid f|_{V_1}) = \Ent(f|_{V_1}) + \sum_{v\in V_2} \Ent(f(v) \mid f|_{N(v)}).
\end{equation}
Now cover $V_1$ by the sets $\{N(v)\}_{v\in V_2}$, so that each vertex in $V_1$ is covered exactly $\Delta$ times. Applying Shearer's inequality, we get that
\begin{equation}
  \Ent(f|_{V_1}) \le \frac{1}{\Delta}\sum_{v\in V_2} \Ent(f|_{N(v)}).
\end{equation}
Altogether,
\begin{equation}
  \Ent(f) \le \frac{1}{\Delta}\sum_{v\in V_2} \Big[ \Ent(f|_{N(v)}) + \Delta \Ent(f(v) \mid f|_{N(v)}) \Big].
\end{equation}
The expression inside the sum is easily identified with $\Ent(g)$, where $g$ is a uniformly sampled proper $q$-coloring of the $\Delta$-regular complete bipartite graph. In conclusion,
\begin{equation}
  \Ent(f) \le \frac{|V_2|}{\Delta}\Ent(g) = \frac{|V_1 \cup V_2|}{2\Delta} \Ent(g).
\end{equation}
Since $\exp(\Ent(f))$ is the number of proper colorings of $G$ (and, similarly, $\exp(\Ent(g))$ is the number of proper colorings of the $\Delta$-regular complete bipartite graph), this shows that the maximal number of proper $q$-colorings is attained when the graph $G$ is a disjoint union of $\Delta$-regular complete bipartite graphs. It also yields explicit, and relatively simple, upper bounds on the number of proper colorings of $G$.

\subsubsection{Upper bounds on entropy loss}
\label{sec:entropy-bounds}

Recall from Section~\ref{sec:repair transformation} that $X=(X_P)_P$ is fixed, that $\Omega_X$ is the set of proper colorings for which $X$ is a breakup, and that in order to the bound the probability of $\Omega_X$, we must bound the entropy loss resulting from step (i) of the repair transformation.
As this entropy loss is to be compared with the entropy gain from step (iii), which is either $\log \lfloor \tfrac{q}{2} \rfloor$ or $\log \lceil \tfrac{q}{2} \rceil$ per vertex, according to the parity of the vertex, we need to show (roughly speaking) that the entropy per vertex in $X_*$ is less than $\frac12 \log (\lfloor \tfrac{q}{2} \rfloor\lceil \tfrac{q}{2} \rceil) - \varepsilon$ for some constant $\varepsilon>0$.

Let $f$ be sampled from $\Pr_{\Lambda,P_0}$ conditioned on $f\in\Omega_X$. Our goal is to bound the entropy of~$f|_{X_*}$. The basic idea for this is to use the method described in Section~\ref{sec:entropy}, with two main differences that need to be addressed: first, as $X_*$ is not a regular graph in itself (it is a subset of a regular graph), we will need to handle its boundary with special care; second, we cannot simply compare to the complete bipartite graph, which would yield the insufficient bound of $\frac12 \log (\lfloor \tfrac{q}{2} \rfloor\lceil \tfrac{q}{2} \rceil) + \varepsilon$, but rather we must take into account the constraints imposed by the breakup on the coloring on $X_*$ (recall that this is a `bad' region, consisting of overlapping ordered regions, disordered regions and boundaries of ordered regions). We proceed to describe how this is done.

For convenience, let us define $F$ to be the configuration coinciding with $f$ on $X_*$ and equaling a fixed symbol $\star$ on $X_*^c$. Thus, $F$ has the same entropy as $f|_{X_*}$, so that it suffices to bound the entropy of $F$.
In a similar manner as in Section~\ref{sec:entropy}, applying Shearer's inequality to the collection of random variables $(F_v)_{v \in \Even}$ with the collection of covering sets $\cI = \{ N(v) \}_{v \in \Odd}$, yields
\[ \Ent(F)
 = \Ent(F|_{\Even}) + \Ent(F|_{\Odd} \mid F|_{\Even})
 \le \sum_{v \in \Odd} \left[ \tfrac{\Ent(F|_{N(v)})}{2d}  + \Ent\big(F(v) \mid F|_{N(v)}\big) \right]. \]
Averaging this with the inequality obtained by reversing the roles of odd and even, writing $\Ent(F|_{N(v)})$ as $\Ent(F(N(v))) + \Ent(F|_{N(v)}\mid F(N(v)))$ and bounding $\Ent(F(v) \mid F|_{N(v)})$ by $\Ent(F(v) \mid F(N(v)))$, yields that
\begin{equation}\label{eq:entropy-bound-overview}
\Ent(F) \le \frac{1}{2}\sum_{v} \bigg[ \underbrace{\tfrac{\Ent\big(F(N(v))\big)}{2d}}_{\textup{I}} + \underbrace{\tfrac{\Ent\big(F|_{N(v)}~\mid~ F(N(v))\big)}{2d} + \Ent\big(F(v) \mid F(N(v))\big)}_{\textup{II}} \bigg].
\end{equation}
Of course, the terms corresponding to vertices $v$ at distance two or more from $X_*$ equal zero as $F$ is deterministic in their neighborhood.
The boundary terms corresponding to vertices $v$ in $\intextB X_*$ need to be handled with careful bookkeeping, which we do not elaborate on here.
The advantage of this bound is that it is local, with each term involving only the values of $F$ on a vertex and its neighbors. Each term admits the simple bounds
\begin{equation}\label{eq:entropy-simple-bounds}
\textup{I}\le \tfrac{q\log 2}{2d} \qquad\text{and}\qquad \textup{II}\le\log (\lfloor \tfrac{q}{2} \rfloor  \lceil \tfrac{q}{2} \rceil),
\end{equation}
which only take into account the fact that $f$ is a proper coloring, i.e., that $F(v)\notin F(N(v))$. Note that the bound on $\textup{I}$ can be made arbitrarily small by taking the dimension high enough, so that we are mostly concerned with improving the bound on~$\textup{II}$ (though near the boundary of $X_*$, we also need to improve the former bound). Specifically, it suffices to prove a bound of the form $\textup{II} \le \log (\lfloor \tfrac{q}{2} \rfloor  \lceil \tfrac{q}{2} \rceil) - \varepsilon$, where $\varepsilon>0$ is a constant which may depend on $q$ (as a power-law if we wish to obtain a power-law dependence between $d$ and $q$ in the final result), but otherwise does not depend on $d$.
The main contribution to the stated bounds~\eqref{eq:entropy-simple-bounds} comes from the possibility that $(F(v),F|_{N(v)})$ is approximately uniformly distributed in $A \times B^{2d}$ for some dominant pattern $(A,B)$. To obtain stronger bounds, we use additional information implied by the knowledge that $f \in \Omega_X$.

Let us illustrate this idea through examples.
Consider a vertex $v \in X_\overlap$, which for concreteness, we assume to be even. By definition, there exist distinct dominant patterns $P=(A,B)$ and $Q=(A',B')$ such that $v \in X_P \cap X_Q$. Suppose first that $v$ is both $P$-even and $Q$-even (so that $|A|=|A'|=\lfloor \frac q2 \rfloor$). Recalling~\eqref{eq:Z-def}, one may check that $v$ is both in the $P$-phase and in the $Q$-phase, so that $f(v) \in A \cap A'$. In particular, $(f(v),f|_{N(v)})$ belongs to $(A \cap A', (\{1,\dots,q\} \setminus (A \cap A'))^{2d})$. Hence,
\[ \textup{II} \le \log (|A \cap A'| \cdot (q - |A \cap A'|)) \le \log ((\lfloor \tfrac q2 \rfloor-1)(\lceil \tfrac q2 \rceil+1)) \le \log (\lfloor \tfrac{q}{2} \rfloor  \lceil \tfrac{q}{2} \rceil - 1) .\]
Next, suppose instead that $v$ is $P$-even and $Q$-odd (so that $q$ is necessarily odd, $|A|=\lfloor \frac q2 \rfloor$ and $|A'|=\lceil \frac q2 \rceil$).
Using~\eqref{eq:Z-def}, one may again check that $v$ is in the $P$-phase and that $N(v)$ in the $Q$-phase, so that $f(v) \in A$ and $f(N(v)) \subset B'$. Hence,
\[ \textup{II} \le \log (|A| \cdot |B'|) = 2 \log \lfloor \tfrac q2 \rfloor = \log (\lfloor \tfrac{q}{2} \rfloor  \lceil \tfrac{q}{2} \rceil - \lfloor \tfrac{q}{2} \rfloor) \le \log (\lfloor \tfrac{q}{2} \rfloor  \lceil \tfrac{q}{2} \rceil - 1) .\]
The case when $v$ is both $P$-odd and $Q$-odd is similar.

To handle vertices in $X_\bad$ or $\intextB X_P$ requires more work.
For vertices $v \in X_\bad$, the deterministic information implied by $f \in \Omega_X$ does not suffice to obtain a good bound. Indeed, the conditional entropies in $\textup{II}$ are averages over entropies on events of the form $\{f(N(v))=A\}$ with $A \subset \{1,\dots,q\}$, and only in certain cases do we have good control on these entropies (for instance, when $|A| \notin \{\lfloor \tfrac{q}{2} \rfloor, \lceil \tfrac{q}{2} \rceil \}$). This is overcome by controlling the probabilities of such events. For vertices $v \in\intextB X_P$, the problem is of a different nature. Intuitively, the loss of entropy is not tied to the vertex-boundary $\intextB X_P$, but rather to the edge-boundary $\partial X_P$. Indeed, given an edge $(w,v) \in \dpartial X_P$ (i.e., $\{w,v\} \in \partial X_P$ with $w \in X_P$ and $v \notin X_P$), \eqref{eq:Z-def} implies that $f(w) \in A$ and $f(N(v)) \not\subset A$, where $A$ is a side of $P$ having $|A|=\lfloor \frac q2 \rfloor$. After $\textup{II}$ is rewritten as a sum over neighbors $u$ of $v$, the summand corresponding to $u=w$ can be effectively bounded in this case.
These ideas are explained in more detail below.

\subsubsection{Upper bounds on entropy loss -- a closer look}\label{sec:upper_bounds_closer_look}

We elaborate here on the type of additional information we shall use in order to improve the simple bounds~\eqref{eq:entropy-simple-bounds} on~\textup{I} and~\textup{II} of~\eqref{eq:entropy-bound-overview}.
Note first that, using the definition of conditional entropy, \textup{II} can be written more explicitly as an average:
\begin{equation}\label{eq:entropy-IIa}
\textup{II} = \sum_{A \subset \{1,\dots,q\}} \Big[\underbrace{\tfrac{\Ent\big(F|_{N(v)} ~\mid~ F(N(v))=A\big)}{2d} + \Ent\big(F(v) \mid F(N(v))=A\big)}_{\textup{II'}} \Big] \Pr(F(N(v))=A).
\end{equation}
Observe that each term admits the bound $\textup{II'} \le \log (\lfloor \tfrac{q}{2} \rfloor  \lceil \tfrac{q}{2} \rceil)$. Using this bound for all terms leads to the same bound on \textup{II}, as in~\eqref{eq:entropy-simple-bounds}. The above expression shows that in order to improve the bound on~\textup{II}, it suffices to improve the bound on \textup{II'} for many (in a probabilistic sense) values of $A$.
Furthermore, by identifying $F|_{N(v)}$ with the collection of random variables $\{F(u)\}_{u \sim v}$ and using the subadditivity of entropy, we get the bound
\begin{equation}\label{eq:entropy-IIb}
\textup{II'} \le \frac1{2d} \sum_{u \sim v} \Big[\Ent\big(F(u) \mid F(N(v))=A\big) + \Ent\big(F(v) \mid F(N(v))=A\big) \Big].
\end{equation}
Again, each summand is at most $\log (\lfloor \tfrac{q}{2} \rfloor  \lceil \tfrac{q}{2} \rceil)$, and in order to obtain the desired improvement on the bound on \textup{II'}, it suffices to improve this bound for many neighbors $u$ of $v$.

The improved bounds are based on four notions --- \emph{non-dominant vertices}, \emph{restricted edges}, vertices having \emph{unbalanced neighborhoods} and vertices having a \emph{unique pattern} --- all of which we now define. These notions are somewhat abstract (and not directly related to the breakup) in order to allow sufficient flexibility for the proof of both~\eqref{eq:breakup-prob-bound} and a similar version when only an approximation of $X$ is specified.

Let $f \colon \Z^d \to \{1,\dots,q\}$ be a proper coloring and let $\Omega$ be a collection of proper colorings of $\Z^d$.
The four notions implicitly depend on $f$ and~$\Omega$. Let $v \in \Z^d$ be a vertex. With~\eqref{eq:entropy-IIa} and~\eqref{eq:entropy-IIb} in mind, let us denote $A := f(N(v))$.
We say that

\begin{itemize}
	\item
	$v$ is \emph{non-dominant} (in $f$) if $|A| \notin \{\lfloor \frac{q}{2} \rfloor, \lceil \tfrac{q}{2} \rceil \}$.
Thus, a vertex is non-dominant if the set of colors which appear on its neighbors does not determine a dominant pattern. In this case, we have $\textup{II'} \le \log (|A|(q-|A|)) \le \log (\lfloor \tfrac{q}{2} \rfloor  \lceil \tfrac{q}{2} \rceil - 1)$.

	\item
	Given a neighbor $u$ of $v$, the directed edge $(v,u)$ is \emph{restricted} (in $(f,\Omega)$) if either $\{g(u) : g \in \Omega,~ g(N(v)) = A \} \neq A$ or $\{g(v) : g \in \Omega,~ g(N(v)) = A \} \neq A^c$.
Thus, roughly speaking, $(v,u)$ is restricted if upon inspection of the set of values which appears on the neighbors of~$v$, one is guaranteed that either $u$ or $v$ cannot take all possible values which they should typically take, i.e., either $u$ cannot take some value in $A$, or $v$ cannot take some value in $A^c$.
In the former case, the corresponding summand in~\eqref{eq:entropy-IIb} is at most $\log((|A|-1)(q-|A|))$, and in the latter case, it is at most $\log(|A|(q-|A|-1))$, both of which are bounded by $\log (\lfloor \tfrac{q}{2} \rfloor  \lceil \tfrac{q}{2} \rceil - \lfloor \frac q2 \rfloor)$.

	\item
	$v$ has an \emph{unbalanced neighborhood} (in $f$) if $|\{ u\sim v : f(u) = i \}| \le \frac dq$ for some $i \in A$.
As $A$ increases, the set of values which $v$ may take, namely $A^c$, is reduced, resulting in a trade-off in the entropy contribution at $v$ quantified by \textup{II}.
In order to have high entropy, if some neighbor of $v$ takes a value $i$, many other neighbors of $v$ should take advantage of this as well. The neighborhood of $v$ is therefore deemed unbalanced if some value is taken by few (but at least one) neighbors of $v$. Indeed, a standard Chernoff bound shows that, in this case, the number of possibilities for $f|_{N(v)}$ is at most $e^{-cd/q} |A|^{2d}$, so that $\textup{II'} \le \log (\lfloor \tfrac{q}{2} \rfloor  \lceil \tfrac{q}{2} \rceil) - \frac {c}{2q}$.

	\item
	$v$ has a \emph{unique pattern} (in $\Omega$) if there exists $A' \subset \{1,\dots,q\}$ such that, for every $g \in \Omega$, either $g(N(v)) = A'$ or $v$ is non-dominant in $g$ or all out-directed edges in $\partial v$ are restricted in~$(g,\Omega)$. We may more appropriately term this notion as a unique high-entropy pattern or unique unrestricted pattern, the reason being that there is at most one choice for $g(N(v))$ which does not lead to a reduction of entropy at $v$ by making $v$ non-dominant or causing all out-directed edges in $\partial v$ to be restricted. This notion is of a slightly different nature than the previous ones (e.g., it does not depend on $f$). In particular, it is not used to improve the bound on \textup{II}, but rather on \textup{I}. Indeed, in this case, one may argue that either the above cases (non-dominant pattern or restricted edges) occur with substantial probability (so that \textup{II} is effectively bounded), or else it must be the case that $\textup{I} \le e^{-cd/q^2}$.

\end{itemize}

\smallskip

Let us now return to the situation at hand, where $X=(X_P)_P$ is fixed, $\Omega_X$ is the set of proper colorings for which $X$ is a breakup, and $f$ is sampled from $\Pr_{\Lambda,P_0}$ conditioned on $f\in\Omega_X$.
Let $X^f_\unbal$ be the set of vertices in $X_*$ which have unbalanced neighborhoods in $f$, let $X^f_\nondom$ be the set of vertices in $X_*$ which are non-dominant in $f$, let $X^{\Omega,f}_\rest$ be the set of directed edges $(v,u)$ with $v \in X_*$ which are restricted in $(f,\Omega)$ and let $X^\Omega_\unique$ be the set of vertices in $X_*$ which have a unique pattern in $\Omega$.

The repair transformation explained in Section~\ref{sec:repair transformation}, together with the above bounds on the entropy loss in step (i) of the repair transformation, eventually yield (with careful bookkeeping near the boundary of $X_*$) the following bound on the probability of $\Omega_X$ and, more generally, of subevents of $\Omega_X$. For an event $\Omega \subset \Omega_X$, define
	\[ k(\Omega) := \min_{f \in \Omega} ~ \big|S^f_\unbal\big| + \tfrac{1}{q} \big|S^f_\nondom\big| + \tfrac{1}{d}\big|S^{\Omega,f}_\rest\big| .\]
Then, for any event $\Omega \subset \Omega_X$,
\begin{equation}\label{eq:bound-on-subevents-of-breakup}
\Pr_{\Lambda,P_0}(\Omega) \le \exp\left[-\tfrac cq k(\Omega) + \tfrac qd \big|X_* \setminus X^\Omega_\unique\big| +e^{-cd/q^2}|X_*| \right] ,
\end{equation}
where $c>0$ is a universal constant.
Thus, roughly speaking, if $\Omega$ is a subevent of $\Omega_X$ on which there are almost surely many vertices with unbalanced neighborhoods or non-dominant patterns or many restricted edges, then it must be an unlikely event.

We conclude with a short outline as to how~\eqref{eq:bound-on-subevents-of-breakup} is used to prove to the desired bound~\eqref{eq:breakup-prob-bound} on the probability of $\Omega_X$ itself.
Unfortunately, concluding~\eqref{eq:breakup-prob-bound} from~\eqref{eq:bound-on-subevents-of-breakup} is not straightforward, as the latter (when applied directly to the event $\Omega=\Omega_X$) gives an insufficient bound on the probability of $\Omega_X$.
The difficulty here is that, while $k(\Omega_X)$ is always large in comparison to $L$ and~$M$, it is not necessarily large in comparison to~$N$ (recall $L,M,N$ from~\eqref{eq:breakup-size}).
Indeed, it can be shown that every in-directed edge in $\partial X_P$ (i.e., a directed edge $(v,u)$ such that $\{v,u\} \in \partial X_P$ with $u \in X_P$ and $v \notin X_P$) is necessarily restricted in $f$ and every edge incident to $X_\overlap$ is either restricted in $f$ or incident to a non-dominant vertex in $f$, so that
\[ \tfrac{1}{q}\big|S^f_\nondom\big| + \tfrac{1}{d}\big|S^{\Omega_X,f}_\rest\big| \ge \tfrac{L}{2d} + \tfrac{M}{4q} .\]
Unfortunately, $X_\bad$ need not contain enough restricted edges (or non-dominant vertices or vertices having unbalanced neighborhood) -- the main reason being that, when $q$ is even, $X_\bad$ may contain even vertices $v$ for which $|f(N(v))|=\tfrac{q}{2}$, and when $q$ is odd, $X_\bad$ may contain vertices $v$ for which $|f(N(v))|=\lceil \tfrac{q}{2} \rceil$ (recall that this is one of the undesirable consequences of the definition of ordered regions in~\eqref{eq:Z-def}).
Instead, to obtain a good bound, we shall apply~\eqref{eq:bound-on-subevents-of-breakup} to subevents $\Omega \subset \Omega_X$ on which we have additional information about the coloring on the set $X_\bad$.
For suitably chosen subevents (obtained, roughly speaking, by revealing the colors on a sparse subset of $X_\bad$), the number of restricted edges in $X_\bad$ (and non-dominant vertices and vertices having unbalanced neighborhood) increases enough to ensure that
\[ k(\Omega) \ge \tfrac{L}{3d} + \tfrac{M}{6q} + \tfrac{N}{18q^2} .\]
As the entropy of this additional information is negligible with our assumptions, this will allow us to obtain~\eqref{eq:breakup-prob-bound} by taking a union bound over the subevents~$\Omega$.

\subsection{Odd cutsets and their coarse-grained approximations}
\label{sec:approx}

So far we have seen that a given breakup is unlikely.
Recall that the number of odd cutsets (and hence the number of breakups) is too large to allow us to rule out the existence of a breakup seen from $v$ by using a simple union bound (compare the lower bound in~\eqref{eq:number of odd cutsets} to~\eqref{eq:breakup-prob-bound} or~\eqref{eq:single droplet in ordered pattern odd q}). To that end, as explained, we use a coarse-graining scheme for breakups (which we refer to as approximations of breakups). This coarse-graining scheme builds on a known coarse-graining scheme for odd cutsets (which may be thought of as a ``single-droplet'' breakup; recall the toy scenario from Section~\ref{sec:difficulties}).

It is natural to approximate breakups by approximating each $X_P$ separately. This can indeed be done, but due to the amount of dominant phases, it leads to a version of Theorem~\ref{thm:long-range-order} which requires the dimension $d$ to be larger than an exponential function of $q$, rather than the stated power-law dependence~\eqref{eq:dim-assump}. Instead, we use a more sophisticated scheme which takes into account the interplay between the different $X_P$.

An approximation of a breakup $X=(X_P)_P$ is a collection $A = ((A_P)_{P \in \phasedom}, A^*,A^{**})$ of subsets of $\Z^d$, with the following properties: $A_P \subset X_P \subset A_P \cup A^{**}$ for all $P$, so that $A_P$ represents the region known to be in the corresponding set of the breakup and $A^{**}$ indicates the (joint) region which is unknown to belong to some $X_P$.
On a subset $A^* \subset A^{**}$ of the unknown region, one has more information, namely, every vertex in $A^*$ has many neighbors in $\bigcup_P A_P$. Additional properties ensure that the unknown region is not large and that it is only present near $X_*$.
Thus, an approximation provides enough information to recover the breakup everywhere but near $X_*$. We do not give a precise definition here. See Figure~\ref{fig:breakup-approx} for an illustration.

\begin{figure}
	\centering
	\includegraphics[scale=0.204]{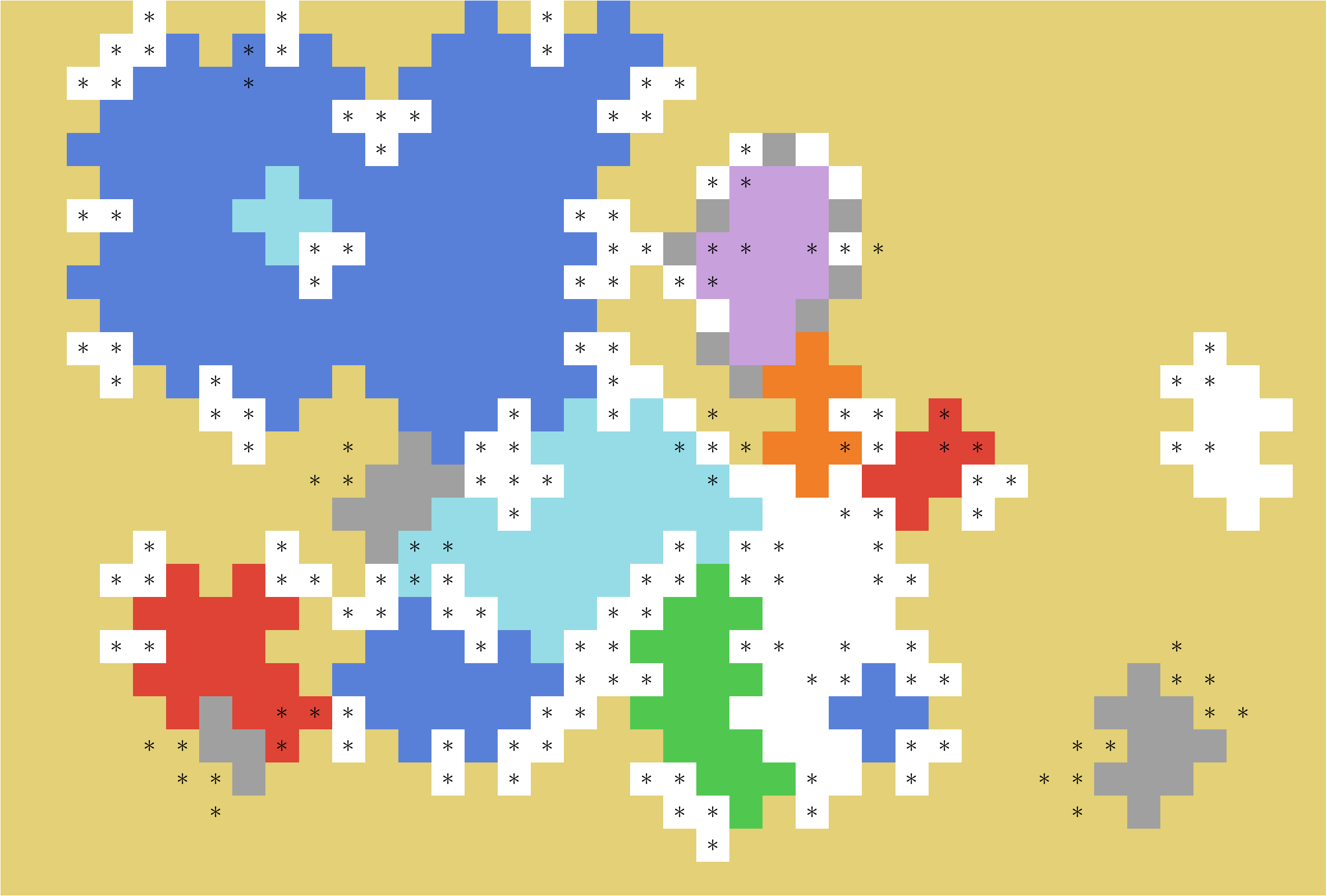}~
	\includegraphics[scale=0.204]{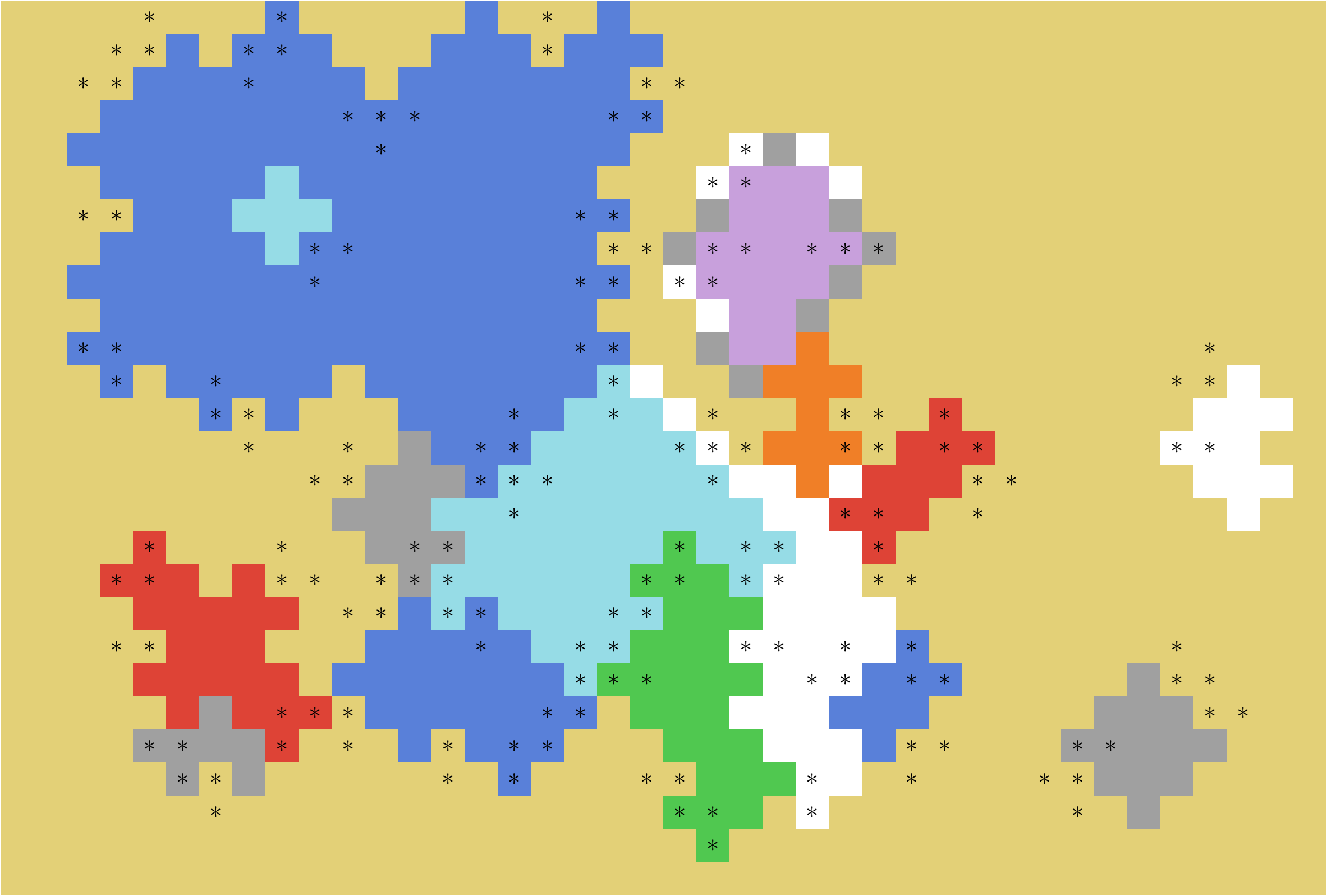}\vspace{-3pt}
    \includegraphics[scale=0.26]{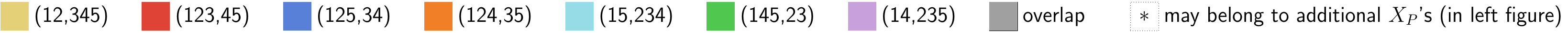}
	\captionsetup{width=0.95\textwidth, font=small}
	\caption{An illustration of an approximation (left) of a breakup (right). An approximation provides partial information on a breakup. On the left, the colors represent regions belonging to a single $A_P$ -- these regions are known to belong to $X_P$ and are not known to belong to any other $X_Q$. A gray background indicates regions belonging to two or more $A_P$ -- these regions are known to belong to $X_\overlap$. A white background indicates regions which do not belong to any $A_P$. A star ($*$) indicates the set $A^{**}$ -- these regions may belong to additional $X_P$'s. In particular, a white background with no star indicates regions which do not belong to $A^{**}$ or to any $A_P$ -- these regions are known to belong to $X_\bad$. The stars are also shown on the right to ease comparison.}
	\label{fig:breakup-approx}
\end{figure}

The advantage of approximations is that one may find a relatively small family $\cA$ of them (much smaller than the number of breakups) with the property that every breakup seen from $v$ is approximated by some element in $\cA$. Thus, by extending the bound~\eqref{eq:breakup-prob-bound} from the situation where the breakup is given to the situation where only an approximation of the breakup is given (the derivation of this bound follows similar lines to that of~\eqref{eq:breakup-prob-bound}), one may then obtain Theorem~\ref{thm:long-range-order} by taking a union bound over $\cA$.

We do not elaborate further on the approximations of breakups in this overview, but rather content ourselves with some explanation on the approximations of odd cutsets (on which the more elaborate coarse-graining scheme is based).

Recall that a subset of $\Z^d$ is called \emph{odd} (\emph{even}) if its internal vertex-boundary consists solely of odd (even) vertices, and that such a set is \emph{regular} if both it and its complement contain no isolated vertices. Recall also that an \emph{odd cutset} is the edge-boundary of a domain. We shall henceforth identify odd cutsets with their associated domain, so that an odd cutset is a non-empty finite odd set which is connected and has a connected complement. Note that an odd cutset which is not a singleton is necessarily regular odd.

An \emph{approximation} (for odd sets) is a pair $A=(A_\ins,A_\out)$ of disjoint subsets of $\Z^d$ such that $A_\ins$ is odd and $A_\out$ is even.
We say that $A$ approximates an odd set $S$ if $A_\ins \subset S$ and $A_\out \subset S^c$. Thus, we think of $A_\ins$ as the set of vertices known to be in $S$,
$A_\out$ as the vertices known to be outside $S$, and $A_* := (A_\ins \cup A_\out)^c$ as the vertices whose association is unknown. See Figure~\ref{fig:approx} for an illustration of these notions.

\begin{figure}
	\captionsetup{width=0.87\textwidth}
	\captionsetup[subfigure]{justification=centering}
    \begin{subfigure}[t]{.35\textwidth}
        \includegraphics[scale=0.29]{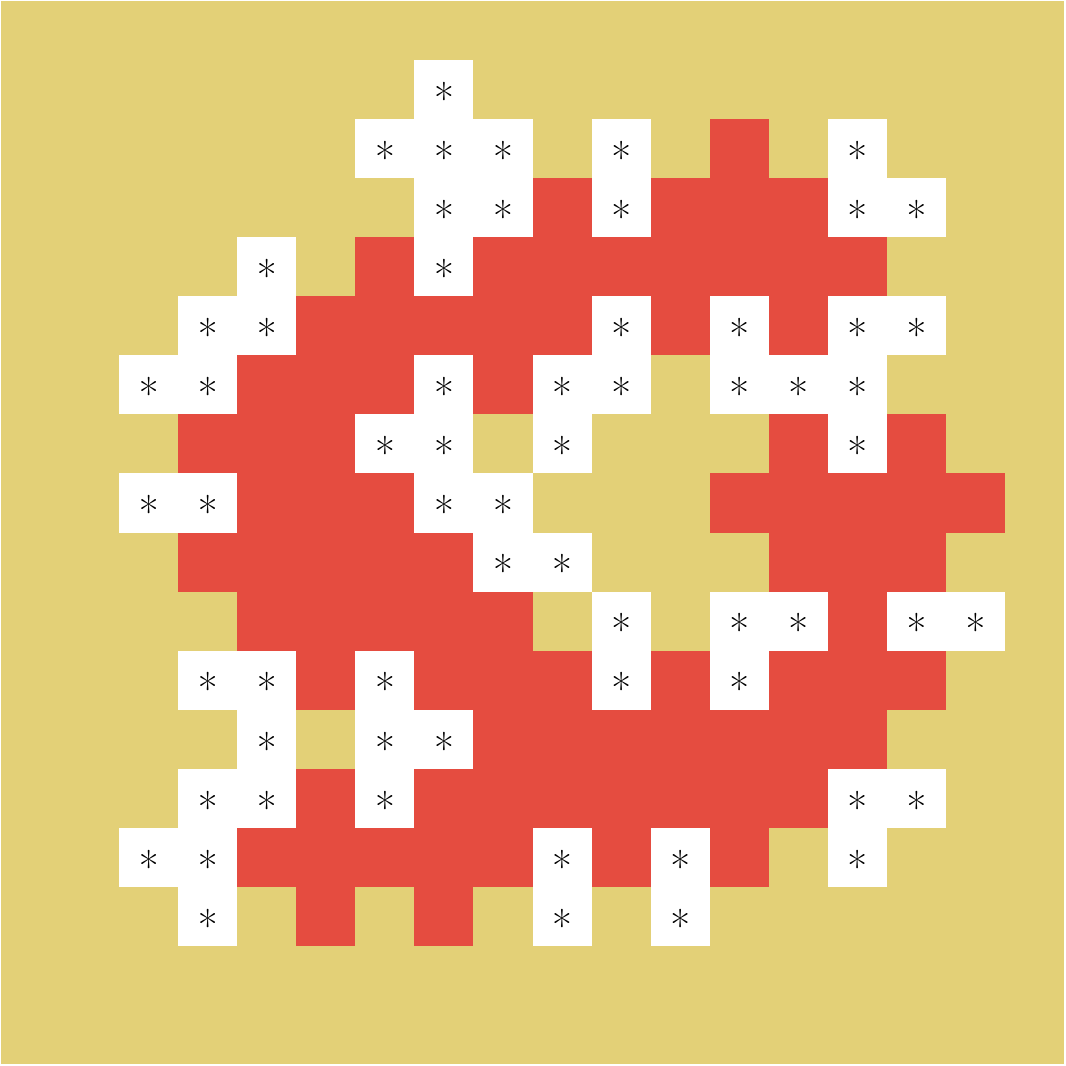}
        \caption{An approximation $A$.}
        \label{fig:approx-sample}
    \end{subfigure}%
    \begin{subfigure}{2pt}
		~
    \end{subfigure}%
    \begin{subfigure}[t]{.65\textwidth}
        \includegraphics[scale=0.29]{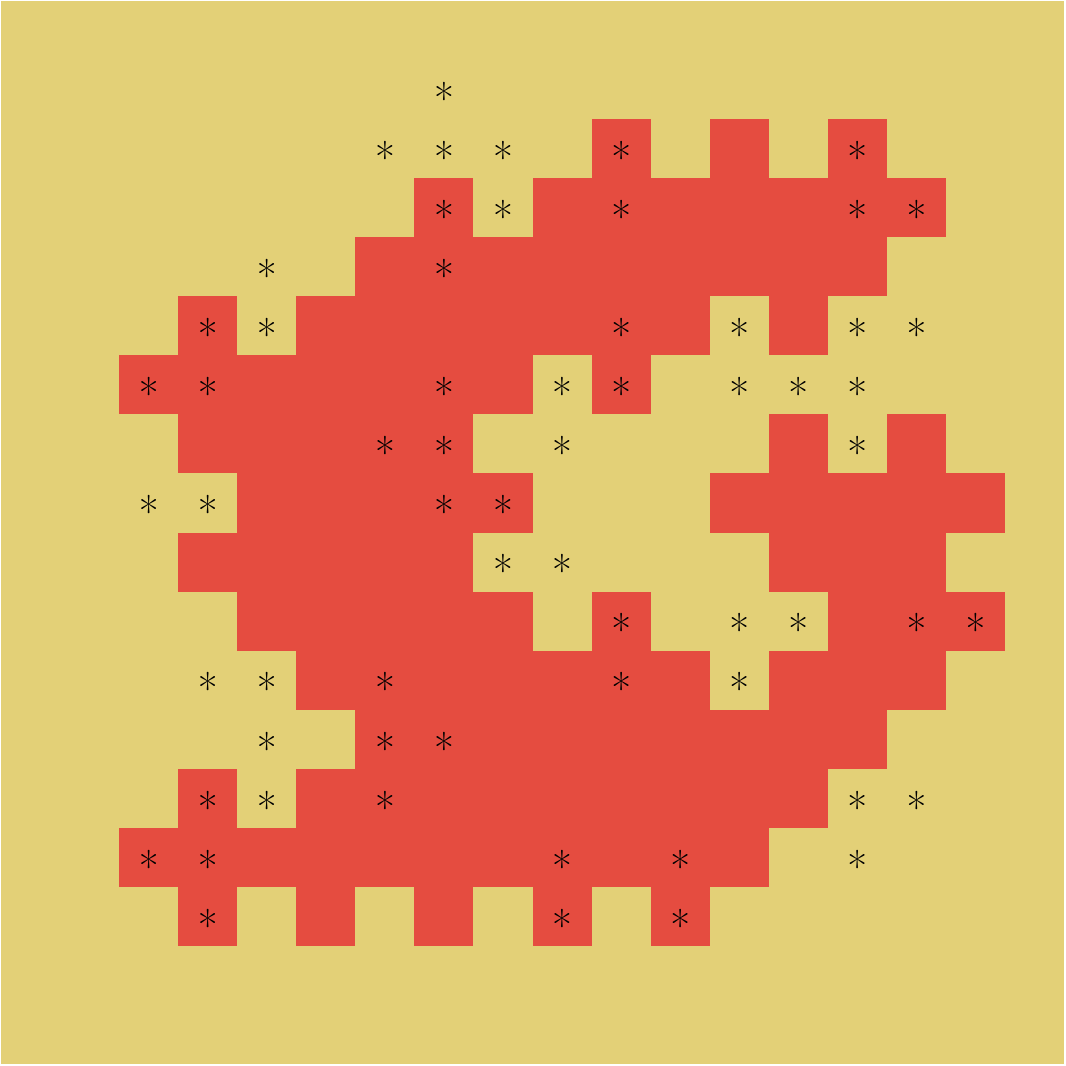}~~%
        \includegraphics[scale=0.29]{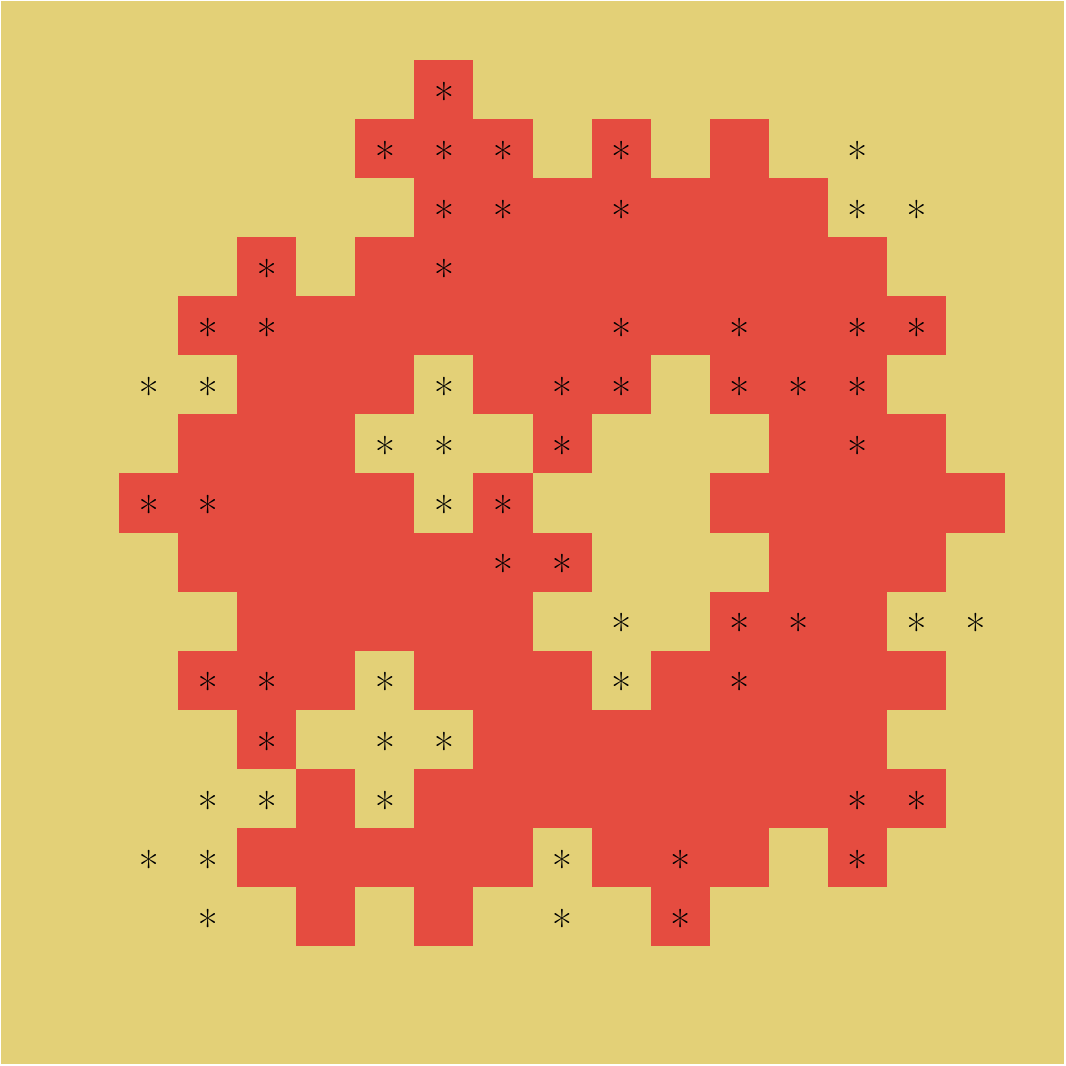}
        \caption{Two possible regular odd sets approximated by $A$.}
        \label{fig:approx-sets}
    \end{subfigure}
    \caption{An approximation and two regular odd sets approximated by it are illustrated. Vertices belonging to $A_\ins$ ($A_\out$) are known to be in $S$ ($S^c$); these are depicted in~(\subref{fig:approx-sample}) by a red (yellow) background. The remaining vertices belong to $A_* = (A_\ins \cup A_\out)^c$ and are unknown to be in $S$ or $S^c$; these are depicted by $*$ and a white background.}
    \label{fig:approx}
\end{figure}

Note that an approximation of a regular odd set, as defined above, need not be a ``good approximation'' in any sense as we have not required much of it.
We now describe a particular sense in which an approximation may be good.
Let $1 \le t < 2d$.
A \emph{$t$-approximation} is an approximation $A$ such that the subgraph of $\Z^d$ induced by $A_*$ has maximum degree at most $t$ and has no isolated vertices.
It is instructive to notice that if a $t$-approximation $A$ approximates a regular odd set $S$, then any unknown vertex is near the boundary in the sense that $A_* \subset (\intextB S)^+$. Of course, the smaller $t$ is, the better the approximation is as it places a stronger constraint on the unknown set $A_*$.

The main result about approximations is that it is possible to construct a small family of $\sqrt{d}$-approximations which contains an approximation of every odd cutset containing the origin and having a given length.

\begin{proposition}\label{lem:family-of-approx}
	There exists a constant $C>0$ such that for any integer $\ell \ge 1$, there exists a family $\cA$ of $\sqrt{d}$-approximations of size
	\[ |\cA| \le \exp\left(C\ell \big(\tfrac{\log d}{d}\big)^{3/2}\right) \]
	such that every odd cutset $S$ containing the origin and having length $|\partial S|=\ell$ is approximated by some element in $\cA$.
\end{proposition}

We do not prove Proposition~\ref{lem:family-of-approx} here; see~\cite{feldheim2018growth} for a full proof.
Instead, to give a flavor of how one obtains such a coarse-graining scheme, we prove a result which serves as the basis for its construction.

We say that a set $W$ \emph{separates} $S$ if every edge in $\partial S$ has an endpoint in $W$.
The following proposition shows that for every regular odd set $S$, we can find a small set $U$ such that $N(U)$ separates $S$.

\begin{proposition}\label{prop:existence-of-U}
	Let $S \subset \Z^d$ be a regular odd set with $\ell:=|\partial S|$ boundary edges. Then there exists a set $U \subset (\intextB S)^+$
of size at most $40\ell d^{-3/2}\sqrt{\log d}$ such that $N(U)$ separates $S$.
\end{proposition}

Such a set $U$ provides a coarse description of the boundary of $S$, and therefore already serves as some weak approximation of $S$. Since the set $U$ is small and also nearly connected (it is connected in a slightly spread-out lattice), its enumeration is not too costly and thereby yields a small family which approximates the required odd cutsets in a weak sense. With some additional work, this weak sense can then be upgraded to the stronger sense of a $\sqrt{d}$-approximation. We do not go into details of this latter part here.

For the proof of the above proposition, we shall require two elementary combinatorial facts about graphs, both of which hold for a general graph $G=(V,E)$ of maximum degree~$\Delta$. For $U \subset V$ and $t>0$, denote
\[ N_t(U) := \{ v \in V : |N(v) \cap U| \ge t \} .\]

\begin{lemma}\label{lem:sizes}
	Let $G=(V,E)$ be a graph of maximum degree $\Delta$.
	Then, for any finite $U \subset V$ and $t>0$,
	\[ |N_t(U)| \le \frac{2d}{t} \cdot |U| .\]
\end{lemma}
\begin{proof}
	This follows from a simple double counting argument.
	\[ t |N_t(U)|
		\le \sum_{v \in N_t(U)} |N(v) \cap U|
		= \sum_{u \in U} \sum_{v \in N_t(U)} \1_{N(u)}(v)
		= \sum_{u \in U} |N(u) \cap N_t(U)|
		\le \Delta |U| . \qedhere \]
\end{proof}

The next lemma follows from a classical result of Lov{\'a}sz~\cite[Corollary~2]{lovasz1975ratio} about fractional vertex covers, applied to a weight function assigning a weight of $\frac1t$ to each vertex of $S$. We give a short probabilistic proof of this.

\begin{lemma}\label{lem:existence-of-covering2}
Let $G=(V,E)$ be a graph of maximum degree $\Delta$, let $A \subset V$ be finite and let $t \ge 1$. Then there exists a set $B \subset A$ of size $|B|~\hspace{-4pt}\le~\hspace{-4pt}\frac{1+\log \Delta}{t} |A|$ such that $N_t(A) \subset N(B)$.
\end{lemma}
\begin{proof}
We may assume that $\log \Delta \le t$, since otherwise we may take $B:=A$.
Let $Z$ be a random subset of $A$, where each vertex $v \in A$ is chosen independently with probability $p:= \frac{\log \Delta}{t}$.
We will show that
\[ \E\big[|Z| + |N_t(A) \setminus N(Z)|\big] \leq \frac{1+\log \Delta}{t} |A| .\]
The lemma will then follow by considering such an instance of $Z$, and taking $B$ to be $Z$, together with a neighbor in $A$ of each vertex in $N_t(A) \setminus N(Z)$.

For $v \in N_t(A)$, we have
\[ \Pr(v \notin N(Z)) \le \Pr(\text{Bin}(\lceil t \rceil,p) =0) = (1-p)^{\lceil t \rceil} \le e^{-tp} = \tfrac 1\Delta .\]
Hence,
\[ \E |N_t(A) \setminus N(Z)| \le \tfrac1\Delta |N_t(A)| \le \tfrac 1t |A| . \]
Thus, since $\E|Z| = p |A| = \tfrac 1t |A| \log \Delta$, the proof is complete.
\end{proof}

We also require the following basic property of odd sets.

\begin{lemma}\label{lem:four-cycle-property}
	Let $S$ be an odd set and let $\{u,v\} \in \partial S$. Then, for any unit vector $e \in \Z^d$, either $\{u,u+e\}$ or $\{v,v+e\}$ belongs to $\partial S$. In particular,
	\[ |\partial u \cap \partial S| + |\partial v \cap \partial S| \ge 2d .\]
\end{lemma}

\begin{proof}
	Assume without loss of generality that $u$ is odd. Since $S$ is odd, we have $u \in S$ and $v \notin S$.
	Similarly, if $u+e \in S$ then $v+e \in S$. Thus, either $\{u,u+e\} \in \partial S$ or $\{v,v+e\} \in \partial S$.
\end{proof}

For a set $S$, denote the {\em revealed vertices} in $S$ by
\[ S^{\rev} := \{ v \in \Z^d ~:~ |\partial v \cap \partial S| \geq d \} \subset \intextB S .\]
That is, a vertex is revealed if it sees the boundary in at least half of the $2d$ directions.
The following is an immediate corollary of Lemma~\ref{lem:four-cycle-property}.

\begin{corollary}\label{cor:revealed-separate}
	Let $S$ be an odd set. Then $S^{\rev}$ separates $S$. \qed
\end{corollary}

We are now ready to prove Proposition~\ref{prop:existence-of-U}.

\begin{figure}
	\centering
	\includegraphics[scale=1.2]{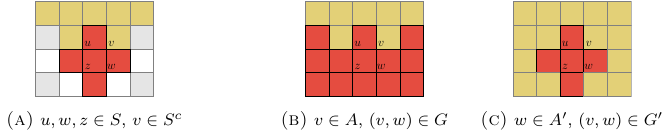}
	\caption{Constructing the separating set. In~\textsc{(a)}, a revealed vertex $u \in S$ is depicted along with a neighbor $z \in S$. Every four-cycle $(u,v,w,z)$ such that $v \in S^c$ (and hence $w \in S$) falls into one of two types. Either $v$ has at least $s$ boundary edges as shown in~\textsc{(b)}, or $w$ has at least $2d-s$ boundary edges as in~\textsc{(c)}. At least half of all such four-cycles belong to the same type. If it is the first type, then $u$ is adjacent to many vertices which have many boundary edges and one such neighbor of $u$ is included in $B$; if it the second type, then $z$ is adjacent to many vertices which have almost all their edges in the boundary and $z$ is included in $B'$. The set $U_S$ is obtained by taking the union of $B$ and $B'$.}
	\label{fig:existence-of-U}
\end{figure}

\begin{proof}[Proof of Proposition~\ref{prop:existence-of-U}]
	Let $S \subset \Z^d$ be a regular odd set and denote $\ell:=|\partial S|$.
	Note that $\partial S = \partial S^c$ implies that $S^{\rev} = (S^c)^{\rev}$. Thus, in light of Corollary~\ref{cor:revealed-separate}, it suffices to show that, for each $R\in\{S,S^c\}$, there exists a set $U_R \subset N(\intB R)$ such that $R \cap R^{\rev} \subset N(U_R)$ and $|U_R| \le 20\ell d^{-3/2} \sqrt{\log d}$.
	Indeed, the lemma then follows by taking $U := U_S \cup U_{S^c}$.
	Since $S$ and $S^c$ are symmetric (up to parity), we may consider the case $R=S$.
	The proof is accompanied by Figure~\ref{fig:existence-of-U}.
	
	Denote $s:=\sqrt{d \log d}$ and $t:=d/4$, and define
	\[ A := \extB S \cap N_s(\intB S) \quad\text{and}\quad A' := \intB S \cap N_{2d-s}(\extB S). \]
	Observe that, by Lemma~\ref{lem:sizes},
	\[ |A| \le \frac{\ell}{s} \quad\text{ and }\quad |A'| \le \frac{\ell}{2d-s}. \]
	We now use Lemma~\ref{lem:existence-of-covering2} with $A$ to obtain a set $B \subset A \subset \extB S$ such that
	\[ |B|\le\frac{4\log d}t|A| \quad\text{and}\quad N_t(A) \subset N(B). \]
	We also define $B':= S \cap N_t(A')$. By Lemma~\ref{lem:sizes}, we have
	\[ |B'| \le \frac{s}{t}|A'| .\]
	Finally, we define $U_S:=B \cup B'$. Clearly, $U_S \subset N(\intB S)$ and
	\[ |U_S| \le \frac{4\ell \log d}{ts} +\frac{s\ell}{t(2d-s)} \le \frac{16\ell \log d}{d\sqrt{d\log d}} +\frac{4\ell\sqrt{d\log d}}{d^2} =
	\frac{20\ell \sqrt{\log d}}{d^{3/2}} .\]
	
	It remains to show that $S \cap S^{\rev} \subset N(U_S)$.
	Towards showing this, let $u \in S \cap S^{\rev} = \intB S \cap N_d(\extB S)$.
	Since $S$ is regular, there exists a vertex $z \in N(u) \cap S$.
	Let $F$ denote the set of pairs $(v,w)$ such that $(u,v,w,z)$ is a four-cycle and $v \in \extB S$, and note that $|F| \ge d-1$.
	Denote
	\[ G := \{ (v,w) \in F ~:~ v \in A \} \quad\text{and}\quad G' := \{ (v,w) \in F ~:~ w \in A' \} .\]
	Note that, by Lemma~\ref{lem:four-cycle-property}, $F = G \cup G'$ and, for any $(v,w) \in F$, we have $w \in S$.
	Since $F = G \cup G'$, either $|G|$ or $|G'|$ is at least $|F|/2 \ge t$.	
	Now observe that
	if $|G| \ge t$ then $u \in N_t(A) \subset N(B)$, while if $|G'| \ge t$ then $z \in N_t(A')$ so that $u \in N(B')$.
	Therefore, we have shown that $u \in N(U_S)$.
\end{proof}

\section{Discussion and open questions}

In the three lectures, we considered the behavior of uniformly sampled proper $q$-colorings of (domains in) the $\Z^d$ lattice, placing particular emphasis on the decay of correlations between the colors assigned to different sites. Three different decay rates were distinguished -- exponential decay (disordered regime), power-law decay (criticality), and no decay (long-range order) -- and results were presented pertaining to each possibility. Among the remaining questions of interest, we wish to point out the following.

\medbreak
\noindent\textbf{General values of $q$ and $d$:} Proper $q$-colorings of $\Z^d$ are disordered when $q>4d$, in the sense of satisfying strong spatial mixing, by Dobrushin's uniqueness condition (see Section~\ref{sec:Dobrushin uniqueness}). This is improved to $q>\alpha\cdot 2d - \gamma$ with $\alpha$ the solution of $\alpha\log\alpha=1$ (so that $\alpha\approx 1.76$) and $\gamma = \frac{4\alpha^3-6\alpha^2-3\alpha+4}{2(\alpha^2-1)}\approx 0.47$ by Goldberg--Martin--Paterson~\cite{goldberg2005strong}. In contrast, the model exhibits long-range order when $d \ge Cq^{10} \log^2 q$, for an absolute constant $C>0$, by the results of~\cite{peled2018rigidity} (see Section~\ref{sec:long-range order}). What is the largest value $q_c(d)$ for which long-range order still holds? Does $q_c(d)$ grow linearly with $d$? If so, what is the limiting value of $q_c(d)/d$ as $d$ tends to infinity?

\medbreak
\noindent\textbf{The anti-ferromagnetic Potts model:} Recall from Section~\ref{sec:Dobrushin uniqueness} that the anti-ferromagnetic (AF) $q$-state Potts model on a graph $G=(V,E)$ at inverse temperature $\beta\ge 0$ is the measure assigning probability proportional to
\begin{equation*}
  \exp\left(-\beta\sum_{\{u,v\}\in E} \1_{f(u) = f(v)}\right)
\end{equation*}
to every $f:V\to\{1,\ldots, q\}$. This measure tends to the uniform distribution on proper $q$-colorings when $\beta$ tends to infinity. What is the behavior of the AF $q$-state Potts model on $\Z^d$? For comparison, in certain classical models it is known that there exists $\beta_c(d)$ so that the model is disordered when $\beta>\beta_c(d)$ and is ordered for $\beta<\beta_c(d)$. At the critical point $\beta = \beta_c(d)$ there are models, such as the two-dimensional ferromagnetic Ising model, for which correlations decay as a power-law (second-order phase transition), and there are models, such as the two-dimensional $q$-state ferromagnetic Potts model with $q>4$~\cite{duminil2016discontinuity,ray2019short}, for which correlations need not decay with distance (first-order phase transition). For the AF $q$-state Potts model on $\Z^d$, Dobrushin's uniqueness condition (Section~\ref{sec:Dobrushin uniqueness}) implies strong spatial mixing when either $\beta\le \frac{C_q}{d}$ or $q>4d$. Long-range order is proved when $q$ and $1/\beta$ are at most $C d^\alpha$ for some $\alpha>0$ using an extension of the techniques of Section~\ref{sec:long-range order}~\cite{peledspinka2018spin}. What is the behavior for intermediate values of the temperature? For instance, is it the case that at all temperatures the correlations decay either at an exponential rate, at a power-law rate, or exhibit no decay at all? Is there a unique transition (critical) point between these regimes? What is the behavior of the model at criticality?

We are not aware of mathematically rigorous results on such intermediate-temperature regimes but the interested reader is directed to the paper of Rahman--Rush--Swendsen~\cite{rahman1998intermediate}, where the $3$-state model in three dimensions is considered, conflicting predictions regarding Permutationally-Symmetric-Sublattice (PSS) and Rotationally-Symmetric (RS) phases are surveyed and the controversy between them is addressed.

\medbreak
\noindent\textbf{Other lattices:} The presented results on the disordered regime (Section~\ref{sec:Dobrushin uniqueness}) apply on general graphs. In contrast, the long-range order result discussed in Section~\ref{sec:long-range order} strongly relies on the bipartite structure of $\Z^d$. It is natural to ask for the behavior of proper colorings, or more generally AF Potts models, on other lattices. Irregularities in a lattice (i.e., having different sublattice densities) often promote the formation of order. This may be used, for instance, to find for each $q$ a \emph{planar} lattice on which the proper $q$-coloring model is ordered~\cite{huang2013two}. However, irregularities also modify the nature of the resulting phase, leading to long-range order in which a single spin value appears on most of the lower-density sublattice~\cite{kotecky2014entropy}, and may lead to partially ordered states~\cite{qin2014partial}.

On the triangular lattice, even the anti-ferromagnetic Ising model (the case $q=2$) is not well understood. As the triangular lattice is not properly 2-colorable, the zero-temperature limit of the AF Ising model becomes the uniform distribution on 2-colorings with the minimal number of non-proper edges. Such colorings turn out to be in bijection with perfect matchings (the dimer model) of the domain and thus have a special integrable structure which may be exploited. This allows to prove power-law decay of correlations at the zero-temperature limit. It is predicted, but not proved, that correlations decay exponentially fast at any positive temperature.

There is another perspective from which the AF Ising model on the triangular lattice is conjectured to behave critically at all temperatures, including infinite temperature ($\beta=0$), where it coincides with critical site percolation. The domain walls between the two color classes form a collection of self-avoiding, non-intersecting loops which are naturally viewed as subsets of edges of the dual hexagonal lattice (this is the loop $O(1)$ model with $x\ge1$; see~\cite[Chapter 3]{peled2019lectures}). These loops are predicted to have fractal structure, with loops appearing at every scale, and to tend to a conformally invariant scaling limit, the CLE process with parameter $\kappa=6$.

Lastly, as mentioned before, proper $4$-colorings of the triangular lattice are predicted to behave critically in the sense of exhibiting power-law decay of correlations, but there are currently no mathematically rigorous results on this case. This model is equivalent to the loop $O(2)$ model with $x=\infty$ (see~\cite[Chapter 3]{peled2019lectures}).

\medbreak
\noindent\textbf{Positive association:} In the study of ferromagnetic models, such as the ferromagnetic Ising or Potts models, a major role is played by positive association inequalities. These state that events which are increasing in a suitable partial order on configurations are positively correlated. Typically, the inequalities are proved by verifying the lattice condition of Fortuin--Kasteleyn--Ginibre~\cite{fortuin1971correlation}. For anti-ferromagnetic models, there is no obvious partial order on configurations that one may use. However, on bipartite graphs the following idea has been suggested: As the colors assigned to adjacent vertices must be different, it may be that the colors assigned to vertices of a single bipartition class have a tendency to be similar. This is the motivation for the following question of the first author and Jeff Kahn~\cite[page 76, question (5)]{OberwolfachCombinatoricsWorkshop2017}: Let $G$ be a finite bipartite graph with bipartition classes $(A,B)$. Let $q\ge3$ and let $f$ be a uniformly sampled proper $q$-coloring of $G$. Does the set $f^{-1}(1)\cap A$ have positive association? Precisely, if $X, Y : \{0, 1\}^A \to \R$ are increasing in the pointwise partial order, does
\begin{equation}\label{eq:positive association}
  \E(X(f^{-1}(1)\cap A) Y(f^{-1}(1)\cap A))\ge \E(X(f^{-1}(1)\cap A)) \cdot \E(Y(f^{-1}(1)\cap A))
\end{equation}
hold? Even in the case of $4$ vertices $u, v, x, y\in A$, it is unknown whether
\begin{equation*}
  \P[f(u) = f(v) = f(x) = f(y) = 1] \ge \P[f(u) = f(v) = 1] \cdot \P[f(x) = f(y) = 1].
\end{equation*}
It is not difficult, however, to show for an arbitrary subset $S\subset A$ that
\begin{equation*}
  \P(f(u) = 1 \text{ and } f(S) = \{1\})\ge \P(f(u) = 1)\cdot \P(f(S) = \{1\}) = \tfrac{1}{q} \P(f(S) = \{1\})
\end{equation*}
 by considering Kempe chains containing $u$; see~\cite[Appendix~A]{ferreira1999antiferromagnetic} for related results valid also at positive temperature. One may further ask whether the FKG lattice condition holds for the collection of indicator random variables $(\1{\{f(v)=1\}})$, $v\in A$, which would then imply~\eqref{eq:positive association}. However, through computer search, we found a counterexample to the lattice condition for 3-colorings of a $9$-vertex graph; see Figure~\ref{fig:dreidel}. We note that there are models for which positive association has been established in the absence of the FKG lattice condition; The two examples in~\cite{fishburn1988match,kahn2007positive} were brought to our attention by Jeff Kahn.

\begin{figure}[t]
 \centering
 \includegraphics[scale=1]{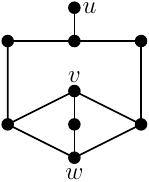}
 \caption{The ``dreidel'' graph -- a counterexample to the FKG lattice condition. The three depicted vertices $u,v,w$ belong to the same bipartition class $A$. Counting 3-colorings directly, one sees that $\P(f(u)=1 \mid f(v)=1) = \frac{23}{56} > \frac{9}{22} = \P(f(u)=1 \mid f(v)=f(w)=1)$, which violates the FKG lattice condition for the collection $(\1{\{f(x)=1\}})_{x \in A}$. In addition, one may check that the events $\{f(u)=f(v)\}$ and $\{f(v)=f(w)\}$ are negatively correlated, showing that the collection $(\1\{f(x)=f(y)\})_{x,y \in A}$ does not have positive association in this example.}
 \label{fig:dreidel}
\end{figure}

\medbreak
\noindent\textbf{Disordered regime on general graphs:} Dobrushin's uniqueness condition (Section~\ref{sec:Dobrushin uniqueness}) shows that proper $q$-colorings of a general graph $G=(V,E)$ with maximal degree $\Delta$ satisfy strong spatial mixing when $q>2\Delta$. What is the sharp dependence in terms of $\Delta$ for this to hold? In this regard we recall again that Chen--Delcourt--Moitra--Perarnau--Postle~\cite{chen2019improved} (improving Vigoda~\cite{vigoda2000improved}) proved that a natural ``flip dynamics'' mixes in time $O(|V|\log|V|)$, and Glauber dynamics mixes in time polynomial in $|V|$, when $q\ge(\frac{11}{6}-\eps)\Delta$ for a fixed $\eps>0$. It seems quite plausible that strong spatial mixing can also be deduced under this assumption. However, it is in fact conjectured that Glauber dynamics mixes rapidly for proper $q$-colorings already when $q\ge \Delta+2$ (see~\cite{jerrum1995very, chen2019improved}). This is not the case for $q=\Delta+1$ as there may exist proper $(\Delta+1)$-colorings which cannot be modified on any vertex while remaining proper~\cite{lubin1993comment} (such colorings trivially cannot exist with $q\ge \Delta+2$ colors). For instance, on the graph $\{0,1,2,3,4\}^2$ endowed with \emph{periodic boundary conditions}, the proper $5$-coloring $f$ defined by $f(x,y) = x+2y\pmod 5$ gives such an example (it is, however, possible to modify this coloring on pairs of adjacent vertices and leave it proper. \emph{Frozen} colorings of $\Z^d$, which cannot be modified on any finite set of vertices, exist only when $q\le d+1$~\cite{alon2019mixing}).

The chromatic number of a general graph is at most $\Delta+1$ (as a greedy coloring shows) and the examples of the clique and odd cycle show that this bound is tight (for other graphs, the chromatic number is at most $\Delta$). Kahn~\cite[page 76, question (5)]{OberwolfachCombinatoricsWorkshop2017} asks whether correlations always decay for proper $(\Delta+1)$-colorings sampled without boundary conditions. Precisely, does there exist a sequence $(\eps_n)$ tending to zero so that the following holds: Let $G$ be a graph with maximal degree $\Delta$. Let $x,y$ be vertices in $G$ at graph distance $r$. Let $f$ be a uniformly sampled proper $(\Delta+1)$-coloring of $G$. Then $|\P(f(x) = f(y)) - \frac{1}{\Delta+1}|\le \eps_r$. In this regard we mention that the $\Delta$-regular tree has a unique Gibbs measure for proper $q$-colorings if and only if $q\ge \Delta+1$~\cite{brightwell2002random,jonasson2002uniqueness}.

\bibliographystyle{amsplain}
\bibliography{biblio}

\end{document}